\documentclass[reqno,11pt]{amsart}
\usepackage{amsmath,amsfonts,amssymb,amsxtra,latexsym,amscd,enumerate,amsthm,verbatim,mathrsfs}

\usepackage{hyperref}
\usepackage{graphicx}
\usepackage{cancel}
\usepackage{cite}

\usepackage[dvipsnames]{xcolor}

\usepackage{hyperref}
\hypersetup{colorlinks=true, pdfstartview=FitV,linkcolor=Red,citecolor=Red, urlcolor=Red}
\definecolor{labelkey}{rgb}{1,0,0}

\usepackage[margin=1.30in]{geometry}
\numberwithin{equation}{section}

\def\eps{\varepsilon}

\newtheorem{theorem}{Theorem}[section]
\newtheorem{lemma}[theorem]{Lemma}
\newtheorem{corollary}[theorem]{Corollary}

\newtheorem{proposition}[theorem]{Proposition}

\newtheorem{remark}[theorem]{Remark}

\title{The Massless Electron limit of the Vlasov-Poisson-Landau system}
\author{Patrick Flynn and Yan Guo}

\begin{document}

\begin{abstract}
Due to ion-electron collisions, it is impossible to derive any
two-fluid model for plasma as a direct hydrodynamic limit of the
Vlasov-Poisson-Landau system for ions and electrons. At the same time, electrons are much lighter than their ion counterparts.
 In this work,  we derive the massless electron limit of the Vlasov-Poisson-Landau system. This is done via a re-scaling of the electron velocity, leading to multiple velocity scales. Importantly, we demonstrate that ion-electron collisions vanish in this limit, due to special structure of the Landau collisions. We also show that this is invalid for the
classical Boltzmann kernel with hard sphere interaction. 
This mechanism serves as the first  step for the derivation of the
two-fluid model for ions from a two-species kinetic equation.

\end{abstract}

\maketitle

\tableofcontents

\section{Introduction} 

A plasma is a collection of fast-moving, charged particles.
It is believed that more than 95 percent of matter in the universe
takes the form of a plasma. Besides this, a major impetus for the study of plasmas is to obtain
nuclear fusion as a means of energy production. Due to its complex and rich nature, there are
three main distinct families of models (kinetic, two-fluid and magnetohydrodynamic) for
describing a plasma in different physical regimes. 
The two-fluid models (Euler-Poisson and Euler-Maxwell systems)
describe dynamics of two \textit{distinct} and \textit{separate} compressible
ion and electron fluids, interacting with their self-consistent electromagnetic
field. 
Such two-fluid models are important both from physical as well as
mathematical standpoint, serving as an origin for most well-known dispersive
PDEs, for instance KdV  \cite{berezin1967nonlinear, washimi1966propagation}, NLS \cite{shimizu1972automodulation}, Zakharov \cite{zakharov1972collapse}, etc.

It is an important question to derive and justify two-fluid theory from more
basic kinetic models for plasma. Consider a classical kinetic model for ions
and electrons, the Vlasov-Poisson-Landau system, which models the distribution functions $F_+(t,x,v)$ and $F_-(t,x,v)$ for ions ($+$) and electrons ($-$) respectively:
\begin{equation}\label{eq:VPL}
\begin{aligned}
\{\partial_t  + v \cdot \nabla_x  +\frac{Z\mathsf e}{m_+} E\cdot \nabla_v\} F_+ &= 2\pi \mathsf e^4\ln(\Lambda) \{Z^4Q_{++}(F_+, F_+) + Z^2Q_{-+}(F_-, F_+)\}, \\
\{\partial_t  + v \cdot \nabla_x  -\frac{\mathsf e}{m_-} E\cdot \nabla_v\}   F_- &= 2\pi \mathsf e^4\ln(\Lambda) \{Q_{--}(F_-, F_-) + Z^2Q_{+-}(F_+, F_-)\}, \\
 - \Delta_x \phi &=  4\pi \mathsf e(Z n_+ - n_-) .
\end{aligned}
\end{equation}
In the above, $\mathsf e$ is the electron charge,\footnote{Not to be confused with the mathematical constant $e \approx2.718$.} and $Z$ is the atomic number of the ion species. The ion and electron masses are $m_+$ and $m_-$ respectively. These distribution functions solve the system
We also have the densities $n_\pm = \int_{\mathbf R^3} F_\pm dv $, the electric potential $\phi$, the electric field  $E :=- \nabla_x \phi$, and  the Landau collision operators $Q_{++}, Q_{-+}, Q_{--}$ and $Q_{+-}$ (tabulated in Section \ref{app:non-dim} below). We refer the reader to \cite{hinton1983collisional}, and Chapter 2 of \cite{balescu1988transport} for the physical justification of these models.

It is well known that any straightforward fluid limit must lead to
\begin{align}
Z^{4}Q_{++}(F_{+},F_{+})+Z^{2}Q_{-+}(F_{-},F_{+}) &  =0,\\
Q_{+-}(F_{+},F_{+})+Z^{2}Q_{--}(F_{-},F_{+}) &  =0.
\end{align}
However, due to the presence
of ion and electron interactions $Q_{+-}$ and $Q_{-+}$, the only possible solutions to the above are of the form\footnote{See equation (40) in \cite{hinton1983collisional}.}
\begin{align}
F_{+} &  = n_+ (\frac{m_{+}}{2\pi T})^{3/2}\exp\left(-\frac{m_{+} |v-u|^2}{ 2T}\right),\\
F_{-} &  =n_-(\frac{m_{-}}{2\pi T})^{3/2}\exp\left(-\frac{m_{-}|v-u|^2}{ 2T}\right).
\end{align}
Here, the electron and ion  fluids exhibit the \textit{same} velocity $u(t,x)$ and temperature $T(t,x)$, which \textit{excludes} the possibility of
justification of any two-fluid models from \eqref{eq:VPL}.
 It is well-known, however,
that it is possible to derive two-fluid models from two-species kinetic equations \textit{if
ion-electron collisions} $Q_{+-}$ \textit{ and } $Q_{-+}$ \textit{are disregarded}. This was done in the context of the Vlasov-Poisson-Boltzmann system \cite{guo2009local,guo2010global, guo2021global}, and more recently for single-species Landau-type equations \cite{ouyang2022hilbert,duan2022compressible}. We also highlight the work \cite{bardos_maxwellboltzmann_2018}, which derives the fluid limit for the Vlasov-Poisson-Boltzmann system, after first deriving the massless electron limit. This served as a major inspiration for our work.
Importantly, however, none of these works consider the inter-species collisions. 

In all physical situations, the electron mass $m_-$ is much smaller than the ion mass $m_+$. For instance,  $\frac{m_{-}%
}{m_{+}}\approx 0.005$ for a hydrogen plasma. This small parameter has been
exploited in different contexts in plasma studies both from physical as well
as mathematical standpoints. For instance, one can derive the Euler-Poisson system for ions in this limit \cite{grenier2020derivation}.  By sending the electron mass to zero, the density of the electrons (and, in turn, the electrostatic potential) becomes wholly determined by the density of the ions. This eliminates the need to solve a fluid or kinetic equation for the electrons separately from the ions, leading to a simpler model.

When the ions and electrons are near thermal equilibrium (say, at temperature $1$), a typical ion speed will be comparable to $ (\frac{1}{m_+})^{\frac{1}{2}}$, while a typical electron speed will be comparable to $ (\frac{1}{m_-})^{\frac{1}{2}}$. Because of these divergent velocity scales, it is necessary to
 introduce the parameter $\eps := (\frac{m_-}{m_+})^\frac{1}{2}$, and the rescaled electron
velocity
\begin{align}
\xi=\eps v.
\end{align}
Applying this rescaling the system \eqref{eq:VPL}, we get the rescaled system \eqref{eq:VPLRescaled}  below. 
The ion-electron collisions then exhibit  \textit{two velocity scales}.
Our main result is that by taking the limit $\varepsilon
\rightarrow0$ in the two-scale system, we are able derive the massless electron system \eqref{eq:VPLIon} below. A salient feature of this derived equation is that the \textit{ion-electron collisions are no longer present}.
In the future, we hope this work will clear the way for the derivation of the Euler-Poisson system for ions. This would be accomplished by first by taking $\eps  \rightarrow0$, and then taking hydrodynamic limit $\kappa
\rightarrow0 $, where $\kappa> 0$ is the analogue of the Knudsen number for our system.

Besides the relevance to the fluid limit, the massless electron limit in kinetic theory is of independent interest. 
There are a number of works that have handled some version of this limit. For instance, \cite{degond1996asymptotics,degond1996transport} give formal expansions for the kinetic equations with Boltzmann and (non-Coulombic) Landau-type collisions for systems of particles with disparate masses, including the cross collision terms.
The works \cite{bouchut1995long,herda2016massless}, give derivations the massless electron limit for Vlasov-Poisson system with linear Fokker-Planck collisions (and in the latter case, with an magnetic field). The aforementioned work \cite{bardos_maxwellboltzmann_2018} derives the analogue of \eqref{eq:VPLIon} for the Vlasov-Poisson-Boltzmann system. We note that they require a regularity assumption, and do not include ion-electron collisions.

There is also a growing literature on the Vlasov-Poisson system with  massless electrons. We refer the reader to \cite{han2011quasineutral,han2017quasineutral,griffin2020singular,bouchut1991global,han2013vlasov,griffin2018global,cesbron2021global,griffin2021global,gagnebin2022landau,griffin2019recent} and the references therein. In particular, the overview \cite{griffin2019recent} has a formal derivation of the system with a Landau collision term.

\subsection{The Rescaled Vlasov-Poisson-Landau system}
%
%
 In Section \ref{app:non-dim}, we apply a rescaling procedure to the system \eqref{eq:VPL}. In particular, we use the rescaled electron velocity $\xi = \eps v$, as aforementioned. 
  This system depends on three positive parameters: $\kappa>0$ (defined in Section \ref{app:non-dim}), the atomic number $Z$, and $\eps =( \frac{m_-}{m_+})^{\frac{1}{2}}$. 
The rescaled versions of the distribution functions $F_+^\eps(t,x,v)$ and $F_-^\eps(t,x,\xi)$ then solve the system
\begin{equation}\label{eq:VPLRescaled}
\begin{aligned}
\{\partial_t  +v \cdot \nabla_x  +Z E^\eps \cdot \nabla_v\} F_+^\eps &= \kappa^{-1} \{Z^3Q(F_+^\eps, F_+^\eps) + Z^2Q_{-+}^\eps(F_-^\eps, F_+^\eps)\}, \\
\{\eps \partial_t  + \xi \cdot \nabla_x  - E^\eps\cdot \nabla_\xi \}   F_-^\eps &= \kappa^{-1}\{Q(F_-^\eps, F_-^\eps) + Z Q_{+-}^\eps(F_+^\eps, F_-^\eps)\}, \\
 - \Delta_x  \phi^\eps &=  4\pi ( n_+^\eps - n_-^\eps)  
\end{aligned}
\end{equation}
 Above, $\phi^\eps$ and $E^\eps = -\nabla_x \phi^\eps$ are the electric potential and field respectively. The functions $n_\pm(t,x)$ are the charge densities, defined by
\begin{align}
n_+^\eps &= \int_{\mathbf R^3} F_+^\eps dv , \ \ \ n_-^\eps =\int_{\mathbf R^3} F_-^\eps d\xi.
\end{align}
 The collision operators are given by
\begin{align}
Q(G_1,G_2) &= \nabla_v \cdot\int_{\mathbf R^3} \Phi(v- v') \left\{G_1(v') \nabla_v G_2(v) - G_2(v) \nabla_{v'} G_1 (v')\right\}dv', \\
Q_{-+}^\eps (G_1,G_2) &=\nabla_{v} \cdot\int_{\mathbf R^3} \Phi(\eps v-   \xi') \left\{\eps G_1(\xi') \nabla_{v} G_2(v) -  G_2(v) \nabla_{\xi '}G_1(\xi')\right\} d\xi ' , \label{eq:Q_-+}\\
Q_{+-}^\eps (G_1,G_2) &=\nabla_{\xi } \cdot\int_{\mathbf R^3} \Phi(\xi-\eps v') \left\{ G_1(v') \nabla_{\xi } G_2(\xi) -  \eps G_2(\xi) \nabla_{v'}G_1(v')\right\} dv'.  \label{eq:Q_+-}
\end{align}
Here, $\Phi$ gives the Landau collision kernel: for each $z \in \mathbf R^3$, 
\begin{equation}
\Phi(z) := \frac{1}{|z|}\left (I - \frac{z\otimes z}{|z|^2}\right).
\end{equation}
From here on  this paper, we only consider the case where $Z = \kappa =1$, and fix the domain $(x,v) \in \mathbf T^3 \times \mathbf R^3$, where $\mathbf T^3 = (\mathbf R/\mathbf Z)^3$ denotes the 3D torus.

%
We recall that the system \eqref{eq:VPLRescaled} comes equipped with conservation of mass, momentum and energy:
\begin{align}
&\frac{d}{dt} \int_{\mathbf T^3 \times \mathbf R^3} F_+^\eps dx dv = \frac{d}{dt} \int_{\mathbf T^3 \times \mathbf R^3} F_-^\eps dx d\xi  = 0,\label{eq:massCons}\\
 &\frac{d}{dt}  \left(\int_{\mathbf T^3 \times \mathbf R^3} vF_+^\eps dx dv + \eps \int_{\mathbf T^3 \times \mathbf R^3} \xi F_-^\eps dx d\xi \right)=0, \label{eq:momCons}\\
 &\frac{d}{dt}  \left(\int_{\mathbf T^3 \times \mathbf R^3} \frac{|v|^2}{2} F_+ ^\eps dx dv +  \int_{\mathbf T^3 \times \mathbf R^3} \frac{|\xi|^2}{2} F_-^\eps dx d\xi  + \frac{1}{8\pi} \int_{\mathbf T^3} |E^\eps|^2 dx \right)=0.\label{eq:energyCons}
\end{align}

From \eqref{eq:VPLRescaled}, it is clear that as $\eps \to 0$, the the term involving $\partial_t$ in the electron equation vanishes. This is a singular limit. Under appropriate circumstances, the electron distribution formally converges to an equilibrium solution at every time, evolving quasi-statically according to the ions. The $\xi $-dependence of these equilibrium solutions are Maxwellians (i.e. Gaussians). We denote the Maxwellians by
\begin{equation}
\mu_q(\xi) := \left(\frac{q}{2\pi}\right)^{\frac{3}{2}} e^{-\frac{q}{2}|\xi|^2}, \quad q > 0. 
\end{equation}
 On the other hand, the $x$ dependence is determined by the electric field, specifically $\lim_{\eps \to 0}\ln(n_-^\eps)$ is proportional to  $\phi^0 = \lim_{\eps \to 0} \phi^\eps$ (up to an additive constant).
The main claim of this paper is that for some $\beta(t) >0$, we have
\begin{align}
\lim_{\eps \to 0} F_-^\eps(t,x,\xi ) & =\mu_{\beta(t)}(\xi) e^{\beta(t) \phi^0(t,x)}  \label{eq:F_-lim}
\end{align}
Here, $\beta^0$ is the inverse temperature. This is known as the Maxwell-Boltzmann approximation. Here, we shift $\phi^0$ by an additive constant to ensure $\int_{\mathbf T^3} e^{\beta(t)\phi^0(t,x)} dx = 1$.

 The above claim is imprecise, because we have not yet made clear our assumptions, nor the sense in which this limit holds. Nevertheless, when the above holds, then the formal limit of \eqref{eq:VPLRescaled}, combined with conservation of energy, leads to the following equation, which we refer to as Vlasov-Poisson-Landau for ions:
\begin{equation}\label{eq:VPLIon}
\begin{aligned}
&\{\partial_t  +v \cdot \nabla_x  - E^0 \cdot \nabla_v\} F_+^0 = Q(F_+^0, F_+^0)\\
&\frac{d}{dt} \left\{\frac{3}{2\beta} + \iint_{\mathbf T^3 \times \mathbf R^3} \frac{|v|^2}{2} F_+^0dx dv + \frac{1}{8\pi}  \int_{\mathbf T^3} |E^0|^2 dx \right\} = 0,\\
&-\Delta_x  \phi^0= 4\pi (n_+^0 - e^{\beta  \phi^0}).
\end{aligned}
\end{equation}
A formal derivation of this system is given in Section \ref{sec:formalDeriv}.



\subsection{Main result}

The goal of this paper is to show that for appropriate initial data, this formal limit holds on a small time interval. We first provide local well-posedness theorems for the systems \eqref{eq:VPLRescaled} and \eqref{eq:VPLIon} for general initial data. To do this, we must first define the function spaces which capture the dissipation rate from the collision kernels. We define the matrix $\sigma_{ij}(v) = \Phi_{ij} * \mu$, with $i, j \in \{1,2,3\}$. Given $u\in \mathbf R^3$, we have that
\begin{equation}\label{eq:sigmaComp}
u^T \sigma(v)u  \sim \frac{1}{\langle v\rangle^3} |P_v u|^2 +  \frac{1}{\langle v\rangle}|P_{v^\perp} u|^2 \end{equation}
Here, $\langle \cdot\rangle = \sqrt{1+|\cdot|^2}$ is the Japanese bracket. The matrices $P_v$ and $P_{v^\perp}$ denote the orthogonal projections onto $\mathrm{span}\{v\}$ and $\{v\}^\perp$ respectively.
Using this matrix, we define the following spaces which will capture the dissipation produced by the various collision operators.
Let $\mathcal H_\sigma$ denote the Hilbert space with inner product (with $\psi= \psi(v), \zeta=\zeta(v)$)
\begin{equation}
\| \psi\|_{\mathcal H_\sigma} ^2= \langle \sigma_{ij}\partial_{v_i} \psi,\partial_{v_j}\psi\rangle_{L^2} +  \langle \mathrm {tr}(\sigma )\psi,\psi\rangle_{L^2}
\end{equation}
Here and throughout the paper, we use repeated index notation, i.e. $u_ku'_k = \sum_{k=1}^3 u_k u'_k$ given any two vectors $u,u' \in \mathbf R^3$. 
The above inner product captures the dissipation associated with self-collisions when linearized near Maxwellian  \cite{degond_dispersion_1997, 
guo_landau_2002,guo2012vlasov}. We also define the semi-norms
\begin{align}
\|\psi\|_{\dot{\mathcal H}_\sigma}^2& = \langle \sigma_{ij}(v)\partial_{v_i} \psi,\partial_{v_j}\psi\rangle_{L^2_v}\\
\|\psi\|_{\mathcal H_{\sigma;\eps}^+}^2 &=\eps \langle \sigma_{ij}(\eps v)\partial_{v_i} \psi,\partial_{v_j}\psi\rangle_{L^2_v}\label{eq:H+}\\
\|\psi\|_{\mathcal H^-_{\sigma;\eps}}^2 &=\frac{1}{\eps} \langle \sigma_{ij}(\frac{v}{\eps})\partial_{\xi _i} \psi,\partial_{\xi _j}\psi\rangle_{L^2_\xi}. \label{eq:H-}
\end{align}
Using these norms, we now define the spaces $\mathfrak E$, $\mathfrak D$ and $\mathfrak D^\pm_\eps$ by their (semi-)norms which we use to control $F_\pm$.
 Fix the parameters $m_k = 5(k+1)$ for $k \in \{0,1,2\}$, and $s =3$. \footnote{Although these parameters may take on a range of values, it is simpler to fix their values to exact constants.}  Given $u = u(x,v)$ (or $u = u(x,\xi)$), we define
\begin{align}
\|u\|_{\mathfrak E}^2 &:= \|\langle v\rangle^{m_2} u\|_{L^2_{x,v}}^2 + \|\langle v\rangle^{m_1}\langle \nabla_x\rangle^s  u\|_{L^2_{x,v}}^2, \\
\|u\|_{\mathfrak D}^2 &:= \|\langle v\rangle^{m_2} u\|_{L^2_{x}(\dot{\mathcal H}_\sigma)_v}^2 + \|\langle v\rangle^{m_1}\langle \nabla_x\rangle^s  u\|_{L^2_{x}(\dot{\mathcal H}_\sigma)_v}^2,\\
\|u\|_{\mathfrak D^\pm_\eps}^2 &:= \|\langle v\rangle^{m_2} u\|_{L^2_{x}(\dot{\mathcal H}_\sigma \cap \mathcal H_{\sigma;\eps}^\pm)_v}^2 + \|\langle v\rangle^{m_1}\langle \nabla_x\rangle^s  u\|_{L^2_{x}(\dot{\mathcal H}_\sigma \cap \mathcal H_{\sigma;\eps}^\pm)_v}^2.
\end{align}
We note that as spaces, $\mathfrak D$ and $ \mathfrak D^\pm_\eps$ are the same. However, the norms are only similar up to a constant which goes to infinity as $\eps \to 0$.

To state the main theorem, 
we will also make use of the following spaces to measure the error in the ion distribution:
\begin{align}
\|u\|_{\mathfrak E'}^2 &:= \|\langle v\rangle^{m_1} u\|_{L^2_{x,v}}^2 + \|\langle v\rangle^{m_0}\langle \nabla_x\rangle^s  u\|_{L^2_{x,v}}^2, \\
\|u\|_{\mathfrak D'}^2 &:= \|\langle v\rangle^{m_1} u\|_{L^2_{x}(\dot{\mathcal H}_\sigma)_v}^2 + \|\langle v\rangle^{m_0}\langle \nabla_x\rangle^s  u\|_{L^2_{x}(\dot{\mathcal H}_\sigma)_v}^2.
\end{align}
Finally, to measure the distance of $F_-^\eps$ from its limit, we let   $\eta >0$ be a constant to be determined later, and define $q_\alpha = e^{\alpha \eta} \beta_{in}$ for each $\alpha  \in \{1,2,3\}$, where $\beta_{in} = \beta(0)$. Then, for $ \alpha\in\{1,2\}$, we define
 \begin{equation}
 \begin{aligned}
 \mathscr E_{-,\alpha}^\eps &:= \|e^{q_\alpha |\xi|^2/4} \langle \nabla_x \rangle^s\{F_{-}^\eps  - \mu_{\beta } e^{\beta \phi^0}\}\|_{L^2_{x,\xi }}^2 \\
 &\quad + \|e^{q_{\alpha +1}|\xi|^2/4} \{F_{-}^\eps  - \mu_{\beta} e^{\beta \phi^0}\}\|_{L^2_{x,\xi }}^2 \\
 \end{aligned}
 \end{equation}
and
 \begin{equation}
 \begin{aligned}
   \mathscr D_{-,\alpha}^\eps& := \|e^{q_\alpha |\xi|^2/4} \langle \nabla_x \rangle^s\{F_{-}^\eps  - \mu_{\beta} e^{\beta \phi^0}\}\|_{L^2_x(\mathcal H_\sigma\cap \mathcal H_{\sigma;\eps}^-)_\xi }^2\\
   &\quad  + \|e^{q_{\alpha +1}|\xi|^2/4} \{F_{-}^\eps  - \mu_{\beta} e^{\beta \phi^0}\}\|_{L^2_x(\mathcal H_\sigma\cap \mathcal H_{\sigma;\eps}^-)_\xi }^2.
 \end{aligned}
 \end{equation}
Also, throughout the rest of the paper, we will use the subscript ``$in$" to denote the initial value of some quantity, for instance, $\phi^0_{in}(x) = \phi^0(0,x)$. 

We can now state our main theorem:
\begin{theorem} \label{thm:derivation}

Suppose for some $T_0 >0$ and $M>0$, there exists a weak solution  $(F_+^0,\beta, \phi^0)$  to \eqref{eq:VPLRescaled}with $0\leq  F_+^0 \in C([0,T_0];\mathfrak E) \cap L^2([0,T_0];\mathfrak D)$, $\beta \in C([0,T_0])$ and $\phi^0 \in C([0,T];H^{s+2})$,
 with
\begin{align}
\sup_{t \in [0,T_0]} \{\|\frac{1}{n_{+}^0}\|_{L^\infty([0,T]\times \mathbf T^3)} + \| F_{+}^0\|_{\mathfrak E}\} + \|\ln(\beta(t))\|_{L^\infty([0,T_0])} \leq M.
\end{align}

\noindent\textit{(i)} There exist positive constants $\eta$, and $\overline \eps$ depending only on $M$, as well as $\hat T \in [0,T_0]$ depending on $(F_+^0,\beta,\phi^0)$ such that the following holds. For each $\eps \in (0,\overline \eps]$, take
 $(F_{+,in}^\eps,F_{-,in}^\eps)$, with $0 \leq F_{\pm,in}^\eps \in \mathfrak E$,  satisfying the estimates
 \begin{align}
\sup_{\eps \in (0,\overline \eps]}\|\frac{1}{n_{+,in}^\eps}\|_{L^\infty_x} + \| F_{+,in}^\eps\|_{\mathfrak E}&  \leq M,\\
\|F_{+,in}^\eps - F_{+,in}^0\|_{\mathfrak E'} + \mathscr E_{-,2,in}^\eps &\leq M \eps.
\end{align}
%

%
%
%
Then, for each $\eps \in (0,\overline \eps]$, there exists a solution $(F_+^\eps,F_-^\eps)$ to \eqref{eq:VPLRescaled} with $0\leq F_{\pm}^\eps \in  C([0,\hat T];\mathfrak E) \cap L^2([0,\hat T]\cap \mathfrak D)$. This satisfies  
\begin{align}
\sup_{t \in [0,\hat T]} \| F_+^\eps(t) - F_+^0(t)\|_{\mathfrak E'} \lesssim_M  \eps \label{eq:error1}\\
\sup_{t\in [0,\hat T]} \eps \mathscr E_{-,2}^\eps(t)+ \int_0^{\hat T} \mathscr D_{-,2}^\eps(t) dt \lesssim_M \eps^2\label{eq:error2} 
\end{align}
Moreover, for all $ t\in[0,\hat T]$, 
\begin{align}
 \mathscr E_{-,1}^\eps(t) \lesssim_M \eps^{\frac{1}{2}} e^{-\frac{1}{C_M} (\frac{t}{\eps})^{\frac{2}{3}}} + \eps^{\frac{5}{3}}t^{\frac{1}{3}} + \eps^2.\label{eq:error3}
\end{align}
\textit{(ii)} 
There exist positive constants $\eta$, $\delta$ and $\overline \eps$ depending on $M$, and $\hat T \in(0, T_0]$ depending on $(F_+^0,\beta,\phi^0)$ and $M$ such that the following holds. For all $\eps \in (0,\overline \eps]$, we set $F_{+,in}^\eps = F_{+,in}^0$ and  $F_{-,in}^\eps = F_{-,in}$ (where $F_{-,in}^\eps$ is independent of $\eps$) satisfying 
\begin{align}
\mathscr E_{-,2,in}^\eps \leq \delta.
\end{align}
We note the expression on the left is independent of $\eps$. 
We also impose the condition on the kinetic energy of $F_{-,in}$:
\begin{align}\label{eq:energyCondition0}
\frac{1}{2} \iint_{\mathbf T^3 \times \mathbf R^3}  |\xi|^2 F_{-,in} dx d\xi   -\frac{3}{2\beta_{in}}  +\frac{1}{8\pi}\int_{\mathbf T^3} |\nabla_x \phi_{in}|^2- |\nabla_x \phi^0_{in}|^2 dx = 0.
\end{align}
Then, for each $\eps \in (0,\overline \eps]$, there exists a solution $(F_+^\eps,F_-^\eps)$ to \eqref{eq:VPLRescaled} with $0\leq F_{\pm}^\eps \in  C([0,\hat T];\mathfrak E) \cap L^2([0,\hat T]\cap \mathfrak D)$. This satisfies  
\begin{align}
\sup_{t \in [0,\hat T]} \| F_+^\eps(t) - F_+^0(t)\|_{\mathfrak E'} \lesssim_M  \sqrt{\eps}\delta + \eps \label{eq:error1'}\\
\sup_{t\in [0,\hat T]} \eps \mathscr E_{-,2}^\eps(t)+ \int_0^{\hat T} \mathscr D_{-,2}^\eps(t) dt \lesssim_M \eps ( \delta^2+ \eps)\label{eq:error2'} 
\end{align}
Moreover, for all $ t\in[0,\hat T]$, 
\begin{align}
 \mathscr E_{-,1}^\eps(t) \lesssim_M \delta e^{-\frac{1}{C_M} (\frac{t}{\eps})^{\frac{2}{3}}} + \eps^{\frac{5}{3}}t^{\frac{1}{3}} + \eps \delta.\label{eq:error3'} 
\end{align}
\end{theorem}

\begin{remark}

This theorem depends on the existence of solutions to the system \eqref{eq:VPLIon}. The local well-posedness theory for this system will be given in a forthcoming paper \cite{flynn2023local}.

\end{remark}
%
%

\subsection{Methodology and outline}

Due to the nature of the  rescaled velocity $\xi$ in the
two-scale system \eqref{eq:VPLRescaled}, the Landau kernel $Q_{+-}
^{\varepsilon}$ exhibits a severe singularity in $\xi$ as $\varepsilon \to 0$. This is mirrored by the degeneracy of the $Q_{-+}^\eps$ in this limit. This makes it very difficult to propagate $v$ derivatives of $F_+^\eps$ and $\xi$ derivatives of $F_-^\eps$. 
  The main
technical achievement of this paper is that we obtain uniform bound in
$\varepsilon$ in weighted Sobolev spaces, \textit{without any} $\xi$ \textit{ or} $v$ \textit{derivatives}. In other words, all regularity in $v$ and $\xi$ comes from the diffusive control from the top order part of the collision operators.

We highlight some of the techniques that go into this. In Lemma \ref{lem:PhiUpperLower}, we obtain a \textit{non-perturbative} lower bound for $\Phi* G$ (where $G \geq 0$), which allows us to access diffusive control from collisions. One challenge we encounter is that we are unable to close a bootstrap estimate on $F_+^\eps$ solely in the $\mathfrak E$ norm. Proposition \ref{prop:apriori} roughly states that if $F_+^\eps$ is of size $M$ on $[0,T]$, then we can only say that $F_+^\eps$ will be of size $C_M$ on the same interval. In general, $C_M$ may grow exponentially in $M$.  However, we can close a bootstrap estimate by also assuming $\|F_+^\eps\|_{\mathfrak D'}$ is small on this interval.
This can be done uniformly in $\eps$ by taking $T$ small enough that $\|F_+^0\|_{\mathfrak D'}$ is as small as desired, and then controlling the difference $\|F_+^\eps - F_+^0\|_{\mathfrak D'}$. We do this through the error estimate in Proposition \ref{prop:ionError}. 

The next challenge is to control the difference $F_-^\eps - \mu_{\beta} e^{\beta \phi^0}$. This is roughly the content of  Proposition \ref{prop:derivationRefined}, although we work with what we call the ``intermediary quantities" $(\gamma^\eps,\psi^\eps)$ rather than $(\beta,\phi^0)$.
 This proposition utilizes the linear decay estimates of the linearized collision operator, via the framework developed in \cite{guo_landau_2002,guo2012vlasov}. Because of the small  parameter in front of the $\partial_t$ term in the second line of \eqref{eq:VPLRescaled}, the question of convergence of $F_-^\eps$ to the Maxwellian $\mu_{\gamma^\eps} e^{\gamma^\eps \psi^\eps}$ is in some sense equivalent to the question of asymptotic stability of the Maxwellian. However, there are a number of complications in our context.  First, the underlying Maxwellian has a varying temperature. To deal with this, we take a window of time such that $\gamma^\eps \approx \beta_{in}$, and choose $q_1<q_2<q_3$ (constant in time) close enough to $\beta_{in}$, and work with these constants instead.
 This results in a perturbed version of the linearized collision operator found in \cite{guo_landau_2002,guo2012vlasov}. Showing that this operator retains the  desired coercivity properties requires some care. There is also the ion-electron collision operator $Q_{+-}^\eps$, which acquires a singularity at $\xi = 0$ as $\eps \to 0$. 
 To control this, we use Lemma \ref{lem:PhiUpperLower} to extract a dissipative, albeit singular and degenerate, top order part. We then use the dissipation from the linearization of $Q$ to control the singular lower order terms.

A key feature of the linearized collision operator is that it has a five dimensional kernel, corresponding to the macroscopic quantities of mass, momentum and energy density.  Besides having to work with a perturbation of this operator, there are some additional challenges in adapting the macroscopic estimates from previous works. One issue is the effect of $Q^\eps_{+-}$ on the macroscopic quantities. Surprisingly, this term has a perturbative effect on the mass and energy density, and even contributes an extra dampening effect on the momentum density (see  Lemma \ref{lem:bBd}). A second challenge is how one controls the mean energy of the perturbation.  The mean kinetic energy density of $F_-^\eps- \mu_{\gamma^\eps} e^{\gamma^\eps \psi^\eps}$ is not conserved, and corresponds to neutral mode at the level of the linearization of the equation for $F_-^\eps- \mu_{\gamma^\eps} e^{\gamma^\eps \psi^\eps}$. In contrast, all other components of the macroscopic quantities are either conserved, or dissipated through hypo-coercive effects. In part (i) of Theorem \ref{thm:derivation}, we can control the mean energy because it is initially small. In part (ii), we must impose \eqref{eq:energyCondition0} and use some nonlinear identities to control the average energy density, as in the first step of Proposition \ref{prop:macroscopicEstimates}. This necessitates the use of the the intermediary quantities $(\gamma^\eps,\psi^\eps)$ instead of $(\beta,\phi^0)$, since $\frac{3}{2\gamma^\eps}$ serves as a better approximation of the mean kinetic energy of $F_-^\eps$ than $\frac{3}{2\beta}$.
%

The structure of the paper is as follows. 
Section \ref{sec:formal} contains all of the formal analysis in this paper. In Section \ref{app:non-dim} we derive the rescaled system \eqref{eq:VPLRescaled}. In Section \ref{sec:formalDeriv}, we give a formal derivation of the system \eqref{eq:VPLIon} from \eqref{eq:VPLRescaled}. In Section \ref{sec:formalDeriv}, we show how the analogous derivation does not work in the context of the Vlasov-Poisson-Boltzmann system, and we highlight the difficulties in defining the limiting system as $\eps \to 0$. 

Section \ref{sec:prelims} contains two preliminary results that will be used in the paper. The first is Lemma \ref{lem:LWP}. This is a statement of local well-posedness that is compatible with the main theorem. We do not prove this result, as it follows from straightforward modifications of the \textit{a priori} estimates shown in this work. The second result is Lemma \ref{lem:PhiUpperLower}, which gives upper and lower bounds on $\Phi * G$. This is particularly important for extracting diffusive control via the norms $\dot{\mathcal H}_\sigma$, and $\mathcal H_{\sigma;\eps}^\pm$.

In Section \ref{sec:ion}, we prove \textit{a priori} estimates on the ion distribution $F_+^\eps$. First, we have Proposition \ref{prop:apriori}, which gives control of $F_+^\eps \in L^\infty(\mathfrak E) \cap L^2(\mathfrak D)$ uniform in $\eps$. Then, we have Proposition \ref{prop:ionError}, which gives uniform control of $F_+^\eps - F_+^0 \in L^\infty(\mathfrak E')\cap L^\infty(\mathfrak D')$. 
Finally, with Lemma \ref{lem:PP}, we construct the intermediary quantities $(\gamma^\eps, \psi^\eps)$, which solve the to the system \eqref{eq:PP} below. These are defined in the same way as $(\beta,\phi^0)$, except using $F_+^\eps$ in place of $F_+^0$.

In Section \ref{sec:derivation}, we prove Proposition \ref{prop:derivationRefined}, which gives the error estimate for the electrons under certain assumptions on the solutions. The proof of this is broken up into two components: energy estimates for $F_-^\eps - \mu_{\gamma^\eps} e^{\gamma^\eps \psi^\eps}$ (as in Proposition \ref{prop:energyEstimates}), and macroscopic estimates (as in Lemma \ref{lem:bBd} and Proposition \ref{prop:macroscopicEstimates}).

Finally, in Section \ref{sec:proof}, we prove Theorem \ref{thm:derivation}. For both parts (i) and (ii), we make five bootstrap assumptions, which allow us to satisfy the assumptions made by the various propositions throughout this paper. We then show show that these assumptions can be propagated over a time interval independent of $\eps > 0$. The error estimates then follow from Propositions \ref{prop:ionError} and \ref{prop:derivationRefined}.

\section{Formal analysis}\label{sec:formal}

\subsection{Non-dimensionalization of the Vlasov-Poisson-Landau System}

 \label{app:non-dim}

Here, we show how to rescale the Vlasov-Poisson-Landau system \eqref{eq:VPL} to the non-dimensionalized form \eqref{eq:VPLRescaled}. 
%
The collision operators in \eqref{eq:VPL}  are
\begin{align}
Q_{++ }(G_1, G_2)(v) &=\frac{ 1}{m_+^2} \nabla_v \cdot\int_{\mathbf R^3} \Phi(v- v') \left\{G_1(v') \nabla_v G_2(v) - G_2(v) \nabla_v G_1 (v')\right\}dv', \\
Q_{-+ }(G_1, G_2)(v) &=\frac{1}{m_+}\nabla_v \cdot\int_{\mathbf R^3} \Phi(v- v') \left\{\frac{1}{m_+}G_1(v') \nabla_v G_2(v) -\frac{1}{m_-}  G_2(v) \nabla_v G_1 (v')\right\}dv', \\
Q_{+-}(G_1, G_2)(v) &=\frac{1}{m_-}  \nabla_v\cdot\int_{\mathbf R^3} \Phi(v- v') \left\{ \frac{1}{m_-}  G_1(v') \nabla_v G_2(v) -\frac{1}{m_+}G_2(v) \nabla_v G_1 (v')\right\}dv', \\
Q_{--}(G_1, G_2) (v)&=\frac{1}{m_-^2}\nabla_v\cdot \int_{\mathbf R^3} \Phi(v- v') \left\{G_1(v') \nabla_v G_2(v) - G_2(v) \nabla_v G_1 (v')\right\}dv',
\end{align} 
In order to define the Coulomb logarithm, we assume $F_+$ and $F_-$ have approximately a common temperature  $\theta > 0$ (with the Boltzmann constant set to 1). To have a charge balance, the electrons will have on average $N$ total number of particles  per unit volume, and the ions will have total number $N/Z$ per unit volume. For simplicity, one might take  $F_+$ and $F_-$ to be perturbations of such equilibrium, i.e.
\begin{align}
F_+ &\approx \frac{N}{Z} (\frac{m_+}{\theta})^{3/2}   e^{-\frac{m_+|v|^2}{2\theta}}, \\
F_- &\approx N(\frac{m_-}{\theta})^{3/2}  e^{-\frac{m_-|v|^2}{2\theta}}.
\end{align}
When $F_+$ and $F_-$ are equal to the Maxwellians, they solve \eqref{eq:VPL} with $\phi = 0$.

Next, the Coulomb logarithm $\ln(\Lambda)$ arises as logarithmically divergent integral in the derivation of the Vlasov-Poisson-Landau system from the BBGKY hierarchy \cite{balescu1988transport}, which is then cutoff at length scales where the assumptions of the model break down. Several choices of cutoffs exist, the most common pair being
 the Debye length at large scales
\begin{equation}
\lambda_D = \left(\frac{\theta}{4\pi \mathsf e^2 N}\right)^{1/2}
\end{equation}
and the distance of closest approach at small scales
\begin{align}
b_0 = \frac{3\theta }{Z\mathsf e^2}.
\end{align}
This gives the expression
\begin{equation}
\ln(\Lambda) = \ln\left(\frac{\lambda_D}{b_0}\right)
\end{equation}
Importantly, $\ln(\Lambda)$ does not depend on the $m_-$ or $m_+$. Unfortunately, this choice fails in contexts where $\lambda_D \leq b_0$, for instance when $N$ is large or $\theta$ is small, and another definition of $\ln(\Lambda)$ is needed. For our purposes, we shall simply take $\ln(\Lambda)$ to be some positive parameter independent of the masses.

We define relevant length, velocity and time scales for our problem. This will allow us to redefine a non-dimensionalized equation depending on three parameters: the (square root of the) mass ratio, $Z$, and a dissipation rate.  The relevant velocity scales for each species are
\begin{align}
V_\pm = (\frac{ \theta}{m_\pm})^{1/2}.
\end{align}
These are the the thermal speeds as in \cite{alexandre_landau_2004}.
Next, we define the length scale
\begin{equation}
X = \left(\frac{\theta} {N\mathsf e^2}\right)^{1/2},
\end{equation}
which can be thought of the smallest length scale at which the electrostatic force is significant for a thermal particle. This gives us natural time scales for each species
\begin{equation}
T^\pm = \frac{X}{V^\pm }= \left(\frac{m_\pm}{N\mathsf e^2}\right)^{1/2}
\end{equation}
While $F_+$ and $F_-$ will be re-scaled in $v$ differently, according to $V_+$ and $V_-$ respectively, we will re-scale time by the ion-time scale for both species. The electrons then evolve according to a faster time-scale than the ions. The ratio of these time-scales is the same as the square-root of the mass ratio,
\begin{equation}
\eps = \frac{T_-}{T_+} = \left(\frac{m_-}{m_+}\right)^{1/2}
\end{equation}
The final non-dimensional parameter is the dissipation rate, which gives the size of the collisional effects:
\begin{equation}
\kappa^{-1} = 2\pi \ln(\Lambda) \frac{N^{1/2} \mathsf e^3}{\theta^{3/2}}.
\end{equation}
We now perform the non-dimensionalization.
Define the re-scaled coordinates
\begin{align}
\overline t &= \frac{t}{T_+},\\
\overline x &= \frac{x}{X},\\
\overline v_\pm &= \frac{v}{V_{\pm}},
\end{align}
and define re-scaled versions of $(F_-,F_+)$:
\begin{align}
\tilde F_+(\overline t,\overline x,\overline v_+) &= \frac{ZV_+^3}{N}F_+(t,x,v),\\
\tilde F_-(\overline t,\overline x,\overline v_-) &= \frac{V_-^3}{N}F_-( t, x, v).
\end{align}
Moreover, define
\begin{align}
\tilde n_\pm(\overline t,\overline x) &= \int \tilde F_+(\overline t,\overline x,\overline v_\pm) d\overline v_\pm ,\\
-\Delta_{\overline x} \tilde \phi &= 4\pi (\tilde n_+(\overline t,\overline x)  - \tilde n_-(\overline t,\overline x)), \\
\tilde E(\overline t,\overline x)&= -\nabla_{\overline x}\tilde \phi(\overline t,\overline x).
\end{align}
In particular,
\begin{align}
\tilde E(\overline t,\overline x) = \frac{1}{NX\mathsf e} E(t,x)
\end{align}
The ``Vlasov-Poisson" parts of the equations \eqref{eq:VPL} transform like
\begin{align}
\{\partial_t  + v \cdot \nabla_x  +\frac{Z\mathsf e}{m_+} E\cdot \nabla_v\} F_+=  \frac{N}{ZT_+V_+^3 }\{\partial_{\overline t}  + \overline v_+ \cdot \nabla_{\overline x}  +Z \tilde E\cdot \nabla_{\overline v_+}\}\tilde F_+\\
\{\partial_t  + v \cdot \nabla_x  -\frac{\mathsf e}{m_-} E\cdot \nabla_v\}F_- =  \frac{N}{T_-V_-^3}\{\eps \partial_{\overline t}  + \overline v_- \cdot \nabla_{\overline x}  -\tilde E\cdot \nabla_{\overline v_-}\}\tilde F_-
\end{align}
We now compute the transformations of the collision kernels:
\begin{align}
&Q_{++ }(F_+, F_+)(v) \\
&\quad  =\frac{N^2}{Z^2m_+^2V_+^6} \nabla_{\overline v_+} \cdot\int_{\mathbf R^3} \Phi(\overline v_+- \overline v_+') \left\{\tilde F_+(\overline v'_+) \nabla_{\overline v_+} \tilde F_+(v_+) -\tilde F_+(\overline v_+) \nabla_{\overline v_+} \tilde F_+(\overline v_+')\right\}d\overline v_+',\\
&Q_{-+}(F_-, F_+)(v) \\
&\quad = \frac{N^2}{Zm_+^2V_+^6}\nabla_{\overline v_+} \cdot\int_{\mathbf R^3} \Phi(\eps\overline v_+-  \overline v_-') \left\{\eps\tilde F_-(\overline v'_-) \nabla_{\overline v_+} \tilde F_+(v_+) - \tilde F_+(\overline v_+) \nabla_{\overline v_-}\tilde  F_-(\overline v_-')\right\} d\overline v_-',\\
&Q_{--}(F_-, F_-)(v) \\
&\quad = \frac{N^2}{m_-^2V_-^6}\nabla_{\overline v_-} \cdot\int_{\mathbf R^3} \Phi(\overline v_--  \overline v_-') \left\{\tilde F_-(\overline v'_-) \nabla_{\overline v_-} \tilde F_-(v_-) - \tilde F_-(\overline v_-) \nabla_{\overline v_-}\tilde  F_-(\overline v_-')\right\} d\overline v_-',\\
&Q_{+-}(F_+, F_-)(v) \\
&\quad = \frac{N^2}{Zm_-^2V_-^6}\nabla_{\overline v_-} \cdot\int_{\mathbf R^3} \Phi(\overline v_--  \eps\overline v_+') \left\{\tilde F_+(\overline v'_+) \nabla_{\overline v_-} \tilde F_-(\overline v_-) - \eps \tilde F_-(\overline v_-) \nabla_{\overline v_+}\tilde  F_+(\overline v_+')\right\} d\overline v_+' 
\end{align}
Ignoring $Z$, the ratios of the pre-factors appearing in the expressions above for the collisions, over that of the transport terms for each species are
\begin{equation}
\frac{N^2}{m_\pm^2 V_\pm^6} \left(\frac{N}{T_\pm V_\pm^3}\right)^{-1} = \frac{NX}{m_\pm^2 V_\pm^4} = \frac{N^{1/2}}{\theta^{3/2} \mathsf e}
\end{equation}
Multiplying the above by $2\pi \mathsf e^4 \ln (\Lambda)$ gives $\kappa^{-1}$.

We now give the non-dimensionalized version of \eqref{eq:VPL}. By abuse of notation, we shall use the symbols $(t,x,v)$ to denote the re-scaled coordinates for the ions; we will use $(t,x,\xi )$ to denote the re-scaled coordinates of the electrons (i.e., we replace $\overline t\to t$, $\overline x\to x$,  $ \overline v_+\to v$,  and $\overline v_-\to \xi$). Moreover, we write $F^\eps_{\pm} = \tilde F_\pm$, and similarly for the potential and electric field. Then, these solve \eqref{eq:VPLRescaled}.
%

\subsection{Formal derivation of the ion equation}
\label{sec:formalDeriv}

We give a formal derivation of \eqref{eq:VPLIon} from \eqref{eq:VPLRescaled} by setting $\eps = 0$. Then,
\begin{align}
\{\partial_t  + v \cdot \nabla_x  + E^0\cdot \nabla_v\} F_+^0 &= Q(F_+^0, F_+^0) +    Q_{-+}^0(  F_-^0, F_+^0),  \label{eq:F_+eps0}
\\
\{ \xi  \cdot \nabla_x  -  E^0\cdot \nabla_\xi \}    F_-^0  &= Q(   F_-^0,  F_-^0) +    Q_{+-}^0(F_+^0,   F_-^0). \label{eq:F_-eps0}
\end{align}
where
\begin{align}
  Q_{-+}^0(  F_-^0, F_+^0)(v) =-\nabla_v \cdot \left(\int_{\mathbf R^3} \Phi( \xi')   \nabla_\xi    F_-^0 (\xi') d\xi ' F_+^0(v)\right), \\
  Q_{+-}^0(F_+^0,   F_-^0)(\xi) =\left(\int_{\mathbf R^3} F_+^0(v')dv'\right)  \nabla_\xi \cdot (\Phi(\xi)  \nabla_\xi    F_-^0(\xi)).
\end{align}
 We assume $F_-$ is positive at each $(t,x,\xi )$, smooth in $(x,\xi )$, and decays sufficiently fast in $\xi $. We can then use the entropy identity to deduce possible solutions. Multiplying \eqref{eq:F_-eps0} by $\ln(F_-)$, we have
\begin{align}
0&=\iint_{\mathbf T^3\times \mathbf R^3} (Q( F_-^0, F_-^0) +  Q^0_{-+}(F_+^0, F_-^0))\ln( F_-)dx dv\\
& = -\frac{1}{2}\iiint_{\mathbf T^3\times \mathbf R^3\times \mathbf R^3}   F_-^0(t,x,\xi )  F_-^0(t,x,\xi ') \\
& \quad \quad \mathrm{tr} (\Phi (\xi-\xi') \{\nabla_\xi  \ln( F_-^0(t,x,\xi )) -\nabla_{\xi '} \ln ( F_-^0(t,x,\xi '))\}^{\otimes 2})dxd\xi d\xi ' \\
&\quad  -n_+^0(t,x) \int_{\mathbf T^3 \times \mathbf R^3} F_-^0(t,x,\xi )\mathrm{tr}(\Phi(\xi)\{\nabla_\xi  \ln( F_-^0(t,x,\xi ))\}^{\otimes 2}) dx d\xi .
\end{align}
Note that both terms are nonpositive, and therefore zero. The null space of $\Phi(\xi)$ is $\mathrm{span}\{\xi \}$, so we deduce that $\nabla_\xi  \ln(F_-^0(t,x) (t,x,\xi )) \in \mathrm{span}\{\xi \}$ for all $(t,x,\xi )$. Thus $F_-(t,x,\xi )$ is radial in $\xi $ for each $(t,x)$. Taking $F_-^0(t,x,\xi ) = \tilde F_-^0(t,x,|\xi|)$, we have that
\begin{align}
 0&= \frac{1}{2}\iiint_{\mathbf T^3\times \mathbf R^3\times \mathbf R^3}  \tilde F_-^0(t,x,|\xi|)  \tilde F_-^0(t,x,|\xi'|) \\
& \quad \quad \mathrm{tr} (\Phi (\xi-\xi') \{ \frac{\xi }{|\xi|}\partial_{r} \ln(\tilde F_-^0(t,x,r))|_{r = |\xi|} - \frac{\xi '}{|\xi'|}\partial_{r}\ln ( \tilde F_-^0(t,x,r))|_{r= |\xi'|}\}^{\otimes 2})dxd\xi d\xi '
\end{align}
Then, for every $\xi , \xi'$ so that $\xi  \neq 0$, $\xi '\neq 0$ and $\xi -\xi' \neq 0$, we have
\begin{align}
 \frac{\xi }{|\xi|}\partial_{r} \ln(\tilde F_-^0(t,x,r))|_{r = |\xi|} - \frac{\xi '}{|\xi'|}\partial_{r}\ln ( \tilde F_-^0(t,x,r))|_{r= |\xi'|}\in \mathrm{span} \{\xi -\xi'\}
\end{align}
The above implies that on the set where $\xi $ and $\xi '$ are linearly independent, we have
\begin{align}
 \frac{\xi }{|\xi|}\partial_{r} \ln(\tilde F_-^0(t,x,r))|_{r = |\xi|} = \frac{\xi '}{|\xi'|}\partial_{r}\ln ( \tilde F_-^0(t,x,r))|_{r= |\xi'|}
 \end{align}
For any given values of $|\xi|, |\xi'| >0$, we can find $\xi ,\xi '$ which are linearly independent. Thus, for all $r,r' >0$, we have
\begin{align}
\frac{1}{r} \partial_{r} \ln(\tilde F_-^0(t,x,r))  = \frac{1}{r'}\partial_{r'} \ln(\tilde F_-^0(t,x,r')) 
\end{align}
In particular, $\partial_{r}(\frac{1}{r} \partial_{r} \ln(\tilde F_-^0(t,x,r)) ) = 0$. Thus, there exist functions $\beta(t,x)$ and $\psi(t,x)$ such that
\begin{align}
\ln(\tilde F_-^0(t,x,r)) = - \frac{\beta(t,x)}{2} r^2 +\psi(t,x)
\end{align}
In other words, $F_-^0(t,x,\xi )$ is a local Maxwellian centered at $\xi  = 0$, i.e. 
\begin{align}
F_-^0(t,x,\xi ) =e^{-\frac{\beta(t,x)|\xi|^2}{2} + \psi(t,x)}
\end{align}
Observe that $Q(F_-^0,F_-^0) = Q_{+-}^0(F_+^0,F_-^0) = 0$. Therefore,
\begin{equation}
\{ \xi  \cdot \nabla_x  -  E\cdot \nabla_\xi \}   F_-^0 = 0.
\end{equation}
Thus, 
\begin{align}
0 =\{ \xi  \cdot \nabla_x  -  E\cdot \nabla_\xi \}  \ln(F_-^0)  = \xi \cdot \nabla_x \psi + \xi\cdot \nabla_x \beta \frac{|\xi|^2}{2} + \xi \cdot E^0 \beta.
\end{align}
Taking $|\xi| \to \infty$, we deduce that $\nabla_x \beta = 0$, i.e. $\beta(t,x) = \beta(t)$. Finally, we conclude $\nabla_x \psi = -\beta E^0 = \nabla_x (\beta \phi^0)$. Thus, $\psi - \beta \phi^0$ is independent of $x$ and $\xi $. We integrate in $\xi $, use conservation of mass  \eqref{eq:massCons} to determine this function, and add an appropriate constant to $\phi^0$ to deduce 
\begin{align}
F_-^0(t,x,\xi ) = \left(\frac{\beta(t)}{2\pi} \right)^{\frac{3}{2}}e^{-\beta(t) ( \frac{|\xi|^2}{2} - \phi^0(t,x))}
\end{align}
This implies $n_-^0 (t,x)= \frac{e^{\beta(t) \phi^0(t,x)} }{\int_{\mathbf T^3} e^{\beta(t) \phi^0(t,x')} dx' }$, giving the third line of \eqref{eq:VPLIon}. As for the first line, since $\nabla_\xi  F_-^0 = -\beta w F_-^0$, we have the important vanishing property
\begin{align}
Q^0_{-+}(F_-^0,F_+^0) = \beta \nabla_ \xi\cdot \{\int_{\mathbf R^3} \Phi(\xi) \xi F_-^0(t,x,\xi ) d\xi  F_+^0(v)\} = 0,
\end{align}
as $\Phi (\xi)\xi = (0,0,0)^T$. Finally, for the second line, we invoke \eqref{eq:energyCons} and the identity
\begin{align}
 \iint_{\mathbf T^3 \times \mathbf R^3} \frac{|\xi|^2}{2} F_-^0dx d\xi  = \frac{3}{2\beta}.
\end{align}

\subsection{The Boltzmann Case}

We now apply the analogous re-scaling to the Vlasov-Poisson-Boltzmann system with hard spheres as was done in Section \ref{app:non-dim}. Unlike the Vlasov-Poisson-Landau system, the ion-electron collisions do not seem to vanish in the limit $\eps\downarrow 0$. Moreover, we show that the formal expansion in $\eps$ seems to yield a system of equations that seems difficult, if not impossible, to solve. This is because the the $O(1)$ terms in the expansion depend on the $O(\eps)$ terms in a nontrivial way, which in turn depend on the $O(\eps^2)$ terms, and so on.

 For simplicity, we let $ Z = \mathsf e = 1$. Following \cite{zhang2009stability}, the Vlasov-Poisson-Boltzmann system reads
\begin{equation}\label{eq:VPB} 
\begin{aligned}
\{\partial_t  + v \cdot \nabla_x  +\frac{1}{m_-} E\cdot \nabla_v\} F_+ &= Q_{++}(F_+, F_+) + Q_{-+}(F_-, F_+)\}, \\
\{\partial_t  + v \cdot \nabla_x  -\frac{1}{m_+} E\cdot \nabla_v\}   F_- &= Q_{--}(F_-, F_-) + Q_{+-}(F_+, F_-), \\
 - \Delta_x \phi &=  4\pi ( n_+ - n_-) 
 \end{aligned}
 \end{equation}
Once again, $E = -\nabla_x \phi$. The Boltzmann collision operators are given as follows: with $\alpha, \beta \in \{-,+\}$, we have
\begin{align}
Q_{\alpha \beta}(G_\alpha, G_\beta) &= \frac{(\sigma_\alpha + \sigma_\beta)^2}{4}  \iint_{\mathbf R^3 \times \mathbf S^2} |(u-v)\cdot \omega| \{ G_\alpha(u')G_\beta(v') -G_\alpha(u) G_\beta(v)\} du d\omega,
\end{align}
where $\sigma_\pm$ are the diameters of the particles, and
\begin{align}
u' &= u + \frac{2 m_\beta}{m_\alpha + m_\beta} ((v-u)\cdot \omega)\omega,\\
v'&= v - \frac{2 m_\alpha}{m_\alpha + m_\beta} ( (v-u)\cdot \omega) \omega.
\end{align}
For simplicity, we set  $\sigma_{\pm} = 1$ as well.
 We now set $\xi= \eps v$ (and $\zeta = \eps u$), and  $\tilde F_-(\xi) =\eps^{-3} F_-(\xi/\eps)$. We now rewrite the first two lines of \eqref{eq:VPB}, 
\begin{align}
\{\partial_t  + v \cdot \nabla_x  + E\cdot \nabla_v\} F_+ &= Q(F_+, F_+) +\tilde Q_{-+}^\eps(\tilde F_-, F_+) \label{eq:boltz+}\\
\{\eps \partial_t + \xi\cdot \nabla_x - E\cdot \nabla_\xi\} \tilde F_- &= Q(\tilde F_-,\tilde F_-) + \tilde Q_{+-}^\eps (F_+,\tilde F_-) . \label{eq:boltz-}
\end{align}
Here, $Q = Q_{++} = Q_{--}$. Before we define the cross-collision terms, it is worthwhile to define the reflection matrix
\begin{align}
R_\omega z = z - 2(z\cdot \omega) \omega,
\end{align}
where $\omega \in \mathbf S^2$. 
The ion-electron collisions have the following (singular!) effect on the ions:
\begin{align}
&\tilde Q_{-+}^\eps (\tilde F_-, F_+)\\
\quad & = \frac{1}{\eps} \iint_{\mathbf R^3 \times \mathbf S^2} |(\zeta -\eps v)\cdot \omega| \{ \tilde F_-(\zeta')F_+(v') -\tilde F_-(\zeta) \tilde F_+(v)\} d\zeta d\omega ,  
\end{align}
where
\begin{align}
\zeta' &=  \zeta+ \frac{2}{1+\eps^2}((\eps v-\zeta)\cdot \omega)\omega = R_\omega \zeta  + 2\eps (v\cdot \omega)\omega + O(\eps^2), \\
v'& = v - \frac{2 \eps}{1 + \eps^2} ((\eps v-  \zeta)\cdot \omega) \omega = v + 2\eps (\zeta \cdot \omega)\omega +  O(\eps^2).
\end{align}
On the other hand, the effect of ion-electron collisions on the electrons are the following:
\begin{equation}
\tilde Q_{+-}^\eps (F_+,\tilde F_-) = \iint_{\mathbf R^3 \times \mathbf S^2} |(\eps u-\xi)\cdot \omega| \{ F_+(u')\tilde F_-(\xi') -F_+(u) \tilde F_-(\xi)\} dud\omega,
\end{equation}
where in the above,
\begin{align}
u' &= u + \frac{2 \eps}{1 + \eps^2} ((\xi-\eps u)\cdot \omega) \omega = u - 2\eps (\xi\cdot \omega) \omega  + O(\eps^2),\\
\xi'&= \xi - \frac{2   }{1 + \eps^2} ((\xi-\eps u)\cdot \omega) \omega = R_\omega \xi + 2\eps (u \cdot \omega) \omega + O(\eps^2).
\end{align}

The  $\eps^{-1}$ term in $\tilde Q_{+-}^\eps$ is the source of the difficulty in describing the formal analysis as $\eps \downarrow 0$. In particular, we have no reason to expect that the effect of $\tilde Q_{-+}^\eps$ will vanish in this limit. To better understand this limit, we use a formal asymptotic expansion. First, we expand the collisons as follows:
\begin{align}
\tilde Q_{-+}^\eps(G_1,G_2)= \sum_{j = -1}^\infty \eps^j q^j_{-+}(G_1,G_2),\label{eq:boltz-+}\\
\tilde Q_{+-}^\eps (G_1,G_2) = \sum_{j = 0}^\infty \eps^j q^j_{+-}(G_1,G_2).\label{eq:boltz+-}
\end{align}
We now give the first two terms in these expansions. The first two terms in \eqref{eq:boltz-+} are
\begin{align}
q^{-1}_{-+}(G_1,G_2).=  \left(\int_{\mathbf R^3 \times \mathbf S^2} |\zeta \cdot \omega| \{ G_1(R_\omega \zeta)- G_1(\zeta)\} d\zeta d \omega \right) G_2(v)  
\end{align}
and
\begin{align}
q^0_{-+}(G_1,G_2)&=- \left(\int_{\mathbf R^3 \times \mathbf S^2}  \mathrm{sgn}(\zeta \cdot \omega)(v\cdot \omega)\{G_1(R_\omega \zeta)-G_1(\zeta)\}  d\zeta d\omega \right) G_2(v)  \label{eq:boltzElecIonExp0term1}\\
& \ \ \ +2 \left(\int_{\mathbf R^3 \times \mathbf S^2} |\zeta \cdot \omega|(v \cdot \omega) \omega \cdot \nabla_\zeta G_1(R_\omega \zeta) d\zeta d\omega \right)G_2(v)  \label{eq:boltzElecIonExp0term2}\\
& \ \ \ -2 \left(\int_{\mathbf R^3 \times \mathbf S^2} |\zeta \cdot \omega| (\zeta \cdot \omega) G_1(R_\omega \zeta)  \omega_i  d\zeta   d\omega  \right) \partial_{v_i} G_2(v)  . \label{eq:boltzElecIonExp0term3}
\end{align}
On the other hand, the first two terms in \eqref{eq:boltz+-} are 
\begin{align}
q^0_{+-}(G_1,G_2) =\left( \int_{\mathbf R^3} G_1(u) du\right)\int_{ \mathbf  S^2} |\xi \cdot \omega|  \{G_1(R_\omega \xi ) - G_1(\xi)\}d\omega.
\end{align}
and
\begin{align}
q^1_{+-}(G_1,G_2) &= -\left( \int_{\mathbf R^3} u_i G_1(u) du\right) \int_{  \mathbf S^2}\mathrm{sgn}(\xi \cdot \omega)\omega_i  \{G_2(R_\omega \xi)-G_2(\xi)\} d\omega  \\
&\quad +2 \left( \int_{\mathbf R^3} u_i G_1(u) du\right) \int_{  \mathbf S^2}\mathrm{sgn}(\xi \cdot \omega)\omega_i \omega \cdot \nabla_\xi  G_2(R_\omega \xi) d\omega.
\end{align}
In  the above, there is another $O(\eps)$ term arising from the expansion $G_1(u') = G_2(u) + 2 \eps (\xi \cdot \omega) \omega \cdot \nabla_u G_1(u) + O(\eps^2)$, but this integrates to zero.

Next, we expand $F_+$ and $\tilde F_-$ into an formal series in $\eps$:
\begin{align}
F_{+} \sim \sum_{j=0}^\infty \eps^j F_{+,j}, \quad \tilde F_- \sim \sum_{j=0}^\infty \eps^j \tilde F_{-,j}, \label{eq:boltz_expansion}
\end{align}
where each $F_{\pm,j}$ is independent of $\eps$.
Similarly, we expand $\phi \sim \sum_{j=0}^\infty \eps^j \phi_j$ and $E \sim \sum_{j=0}^\infty \eps^j E_j$.
Collecting the $O(1)$ terms of \eqref{eq:boltz-}, we see that $F_{-,j}^0$ solves
\begin{align}
\{\xi\cdot \nabla_x - E^0 \cdot \nabla_\xi\} \tilde F_{-,0} = Q(\tilde F_{-,0},\tilde F_{-,0}) + q^0_{-+}(F_{+,0}, \tilde F_{-,0}).
\end{align}
It is straightforward to check that 
\begin{align}
\tilde F_{-,0} = \left(\frac{\beta(t)}{2\pi}\right)^\frac{3}{2} e^{-\beta(t)(\frac{|v|^2}{2}- \phi^0)} 
\end{align}
solves the above, with $(\beta,\phi^0)$ solving the same equation as in \eqref{eq:VPLRescaled}.  Using a similar entropy identity as was used in Section \ref{sec:formalDeriv}, we expect that this is the unique such solution. With this ansatz, we have $q^j_{-,+}(\tilde F_{-,0},G) = 0$ for both $j = -1$ and $j = 0$, and any $G$, due to the radial symmetry of $\tilde F_{-,0}$. In particular, the $O(\frac{1}{\eps})$ term in \eqref{eq:boltz+} vanishes. 
Collecting all the $O(1)$ terms in \eqref{eq:boltz+}, we then have
\begin{align}
\partial_t F_{+,0} + \{v \cdot \nabla_x + E^0 \cdot \nabla_v\} F_{+,0} = Q(F_{+,0},F_{+,0}) + q^{-1}_{-+}(\tilde F_{-,1},F_{+,0}).
\end{align}
Thus, in order to solve for $F_{+,0}$, we must solve for $\tilde F_{-,1}$ as well. Then, collecting all the $O(\eps)$ terms in \eqref{eq:boltz-}, we have
\begin{align}
&\{\xi\cdot \nabla_x - E_0 \cdot \nabla_\xi\} \tilde F_{-,1} - E_1 \cdot \nabla_\xi F_{-,0}\\
&\quad  - Q(\tilde F_{-,1},\tilde F_{-,0}) - Q(\tilde F_{-,0},\tilde F_{-,1}) -q^0_{+-}( F_{+,0},\tilde F_{-,1})\\
 &= \partial_t \tilde F_{-,0} +   q^1_{-+}(F_{+,0}, \tilde F_{-,0}) 
\end{align}
In the above, we used the fact that  $q^0(F_{+,1},\tilde F_{-,0}) = 0$. Now we arrive at the issue: in order to solve for $F_{+,0}$, we must solve for $\tilde F_{-,1}$ as well. However, the equation for $F_{-,1}$ depends on $E_1$, which in turn depends on $ F_{+,1}$, which depends on $\tilde F_{-,2}$, and so on. In summary, it seems difficult to solve for the solve the terms in the expansion \eqref{eq:boltz_expansion}, in order to get a well-defined limiting equation.

\section{Preliminaries}\label{sec:prelims}

\subsection{Well-posedness}

In this section, we state a prerequisite local well-posedness result that guarantees solutions to \eqref{eq:VPLRescaled} in our context. The proof follows by the methods in this paper, so we omit it.
%

\begin{lemma}\label{lem:LWP}
Fix $M ,T> 0$,  and suppose $(F_+^0,\beta,\phi^0)$ is a solution to \eqref{eq:VPLIon} satisfying the hypothesis of Theorem \ref{thm:derivation}. Fix
$\eps > 0$. Then there exists $T_\eps \in (0,T_0]$ (depending on $(F_+^0,\beta,\phi^0)$, $\eps >0$ and $M$), and $\varrho = \varrho(M)$ (depending only on $M$) such that the following holds. If we take $(F_{+,in}^\eps, F_{+,in}^\eps)$ satisfying 
 \begin{align}
\|\frac{1}{n_{+,in}^\eps}\|_{L^\infty_x} + \| F_{+,in}^\eps\|_{\mathfrak E}&  \leq M,\\
\|F_{+,in}^\eps - F_{+,in}^0\|_{\mathfrak E'} +\sqrt{ \mathscr E_{-,2,in}^\eps} &\leq \varrho.
\end{align}
then there exists a unique weak solution $(F_+^\eps, F_-^\eps) $ with $0\leq F_\pm^\eps \in C([0,T_\eps];\mathfrak E)\cap L^2([0,T_\eps];\mathfrak D)$ to \eqref{eq:VPLRescaled} with the initial data,
satisfying the estimates
\begin{align}
\sup_{t \in [0,T_\eps]} \left(\|\frac{1}{n_{+}^\eps(t)}\|_{L^\infty_x} + \| F_{+}^\eps(t)\|_{\mathfrak E}\right) &\leq 2M, \\
\sup_{t \in [0,T_\eps]}\left(\|F_{+}^\eps(t) - F_{+}^0(t)\|_{\mathfrak E'}+  \sqrt{ \mathscr E_{-,2}^\eps(t)} \right)&\lesssim_M \|F_{+,in}^\eps - F_{+,in}^0\|_{\mathfrak E'} +\sqrt{ \mathscr E_{-,2,in}^\eps}.
\end{align}

\end{lemma}

\subsection{Lower and upper bounds on the diffusion matrix}

In order prove local well-posedness for the system \eqref{eq:VPLRescaled} and \eqref{eq:VPLIon}, we first prove the following lemma, which gives upper and lower bounds on the diffusion matrix appearing in the collision kernels:

\begin{lemma} \label{lem:PhiUpperLower}
Let $G(v) = G :  \mathbf R^3 \to \mathbf R$, $\nu \in \mathbf S^2$. Then for all $v \in \mathbf R^3$,
\begin{align}\label{eq:PhiConvUpper}
| \Phi_{ij} * G(v) \nu_i \nu_j| \lesssim \|\langle v\rangle^{5}G\|_{L^2} \sigma_{ij} (v)\nu_i \nu_j.
\end{align}
Assuming $G \geq 0$,  we have the lower  bound
\begin{align}\label{eq:PhiConvLower}
\frac{\|G\|_{L^1}}{\langle\|\langle v\rangle^{2}G\|_{L^2}/\|G\|_{L^1}\rangle ^{17} } \sigma_{ij}(v) \nu_i \nu_j \lesssim \Phi_{ij} * G(v) \nu_i \nu_j.
\end{align}

\end{lemma}

\begin{proof}
We first prove \eqref{eq:PhiConvUpper}. Observe that $| \Phi_{ij} * G(v) \nu_i \nu_j|\leq \Phi_{ij} *| G|(v) \nu_i \nu_j$. Thus, it suffices to consider the case when $G \geq 0$. First we have uniform boundedness, 
\begin{equation}
 \| \Phi_{ij} * G(v)\nu_i\nu_j\|_{L^\infty} \lesssim \|(-\Delta_v)^{-1} G\|_{L^\infty} \lesssim  \| \langle v\rangle^2 G\|_{L^2} \\
\end{equation}
Alternatively,
\begin{align}
\Phi_{ij} * G(v)\nu_i\nu_j& \leq \int_{2| v'| < |v|} \frac{|v-v'|^2 - \langle v-v',\nu\rangle^2}{|v-v'|^3} G(v')dv' \label{eq:PhiUpperT1}  \\ 
&\quad + \int_{2| v'| \geq |v|} \frac{|v-v'|^2 - \langle v-v',\nu\rangle^2}{|v-v'|^3} G(v')dv'.\label{eq:PhiUpperT2}
\end{align}
Now, 
\begin{equation}
\sqrt{|v- v'|^2 -  \langle v-v',\nu\rangle^2} = |P_{\nu^\perp} (v-v')| \leq |P_{\nu^\perp}v| + |P_{\nu^\perp} v'|,
\end{equation}
so
\begin{align}
\eqref{eq:PhiUpperT1} &\lesssim \int_{2| v'| < |v|} \frac{|P_{\nu^\perp} v|^2 +  |P_{\nu^\perp} v|^2}{|v-v'|^3} G(v')dv' \\
&\lesssim \frac{|P_{\nu^\perp} v|^2}{|v|^3} \|G\|_{L^1} +  \frac{1}{|v|^3} \||v'|^2 G\|_{L^1}\\
&\lesssim ( \frac{|P_{v^\perp} \nu|^2}{|v|}  +  \frac{1}{|v|^3}) \|\langle v'\rangle^5 G(v')\|_{L^2_{v'}}
\end{align}
On the other hand, we have
\begin{align}
\eqref{eq:PhiUpperT2}& \leq  \int_{2| v'| \geq |v|} \frac{1}{|v-v'|} G(v')dv' \\
& \leq \frac{1}{|v|^3} \int_{2| v'| \geq |v|} \frac{1}{|v-v'|} | v'|^3G(v')dv' \\
&\leq \frac{1}{|v|^3} \|(-\Delta_{v'})^{-1}(\langle v'\rangle^3 G(v'))\|_{L^\infty_{v'}} \\
&\leq  \frac{1}{|v|^3} \|\langle v'\rangle^5G(v')\|_{L^2_{v'}}.
\end{align}
Thus,
\begin{align}
\Phi_{ij} * G(v)\nu_i\nu_j \lesssim \min\{1, \frac{|P_{v^\perp} \nu|^2}{|v|}  +  \frac{1}{|v|^3}\}\|\langle v'\rangle^5 G(v')\|_{L^2_{v'}},
\end{align}
which implies \eqref{eq:PhiConvUpper}.

We now prove \eqref{eq:PhiConvLower}. By re-scaling $G \mapsto G/\|G\|_{L^1}$, it suffices to show \eqref{eq:PhiConvLower} in the case $\|G\|_{L^1} = 1$. For convenience, we define
\begin{align}
k_G &= \langle \|\langle v\rangle^2 G\|_{L^2_v}\rangle.
\end{align}
Let $\nu \in \mathbf S^2$,
\begin{equation}
  \Phi_{ij} * G(v) \nu_i\nu_j = \int_{\mathbf R^3} \frac{\sin^2\theta }{|v'|} G(v-v')dv'
\end{equation}
where $\theta = \theta(v')$ is the angle between $v'$ and $\nu$. Taking $\theta_0 >0$ to be a constant to be chosen later, we bound the above from below by
\begin{equation}\label{eq:angleDecomp}
 \Phi_{ij} * G(v) \nu_i\nu_j \geq \sin^2\theta_0 \left( \int_{\mathbf R^3}\frac{1}{|v'|}G(v-v')dv' - \int_{\sin^2\theta <\sin^2\theta_0} \frac{1 }{|v'|} G(v-v')dv'  \right)
\end{equation}
Now,
\begin{align}
 \int_{\mathbf R^3}\frac{1}{|v'|}G(v-v')dv' &= \int_{| v| \geq 10 | v'|} \frac{1}{|v-v'|}G(v') dv' \\
 &\geq  \frac{9}{10|v|} \int_{|v| \geq 10|v'|}G(v') dv'\\
 &\geq \frac{9}{10|v|} (1- \int_{|v| \leq 10|v'|}G(v') dv') \\
 &\geq \frac{9}{10|v|}(1- \|| v'|^{-2} 1_{|v|\leq 10|v'|}\|_{L^2_{v'}} \|| v'|^2 G(v')\|_{L^2_{v'}}) \\
 &\geq\frac{9}{10|v|}(1-\frac{k_G}{| v|} )\label{eq:farFieldBound}
\end{align}
On the other hand,  for any $\lambda >0$,
\begin{align}
 \int_{\mathbf R^3}\frac{1}{|v'|} G(v-v') dv' & \geq \frac{1}{\lambda}\left(1- \int_{|v'|> \lambda}G(v-v') dv' \right)\\
 &\geq \frac{1}{\lambda}\left(1- \|| v'|^{-2} 1_{|v'| >\lambda}\|_{L^2_{v'}}  \||v-v' |^2G(v')\|_{L^2_{v'}}\right)\\
 &\geq \frac{1}{\lambda} (1- \frac{k_G\langle v\rangle^2}{\lambda^{\frac{1}{2}}} )
\end{align}

Taking $\lambda = 4k^2_G\langle v\rangle^4$, and combining with \eqref{eq:farFieldBound},
\begin{equation}
\int_{\mathbf R^3} \frac{1}{|v'|} G(v-v') dv' \gtrsim \min\{ \frac{1}{|v|}(1-\frac{k_G}{| v|}),\frac{1}{k_G^2\langle v\rangle^4}\}\gtrsim \frac{1}{ k_G^5\langle v\rangle} \label{eq:PoissonLBd}
\end{equation}
We now consider the second term in \eqref{eq:angleDecomp}. First,
\begin{align}
 \int_{\sin^2\theta <\sin^2\theta_0} \frac{1 }{|v'|} G(v-v')dv'& \lesssim k_G \left(\int_{\sin^2 \theta < \sin^2\theta_0}  \frac{1}{|v'|^2\langle v- v'\rangle^4}dv'\right)^\frac{1}{2} \\
 &\lesssim k_G \left(\int_{\sin^2 \theta < \sin^2\theta_0}  \frac{1}{|v'|^2(1+ |v|^2 +2\langle v,v'\rangle+ |v'|^2)^2}dv'\right)^\frac{1}{2}.
\end{align}
 We then use spherical coordinates to evaluate the integral above, setting the span of $\nu$ to be the $z$-axis, and the span of $P_{\nu^\perp} v$ to be the $x$ axis. Letting $v_\| = \langle v,\nu\rangle$, and $|P_{\nu^\perp} v| = v_\perp$, we get
\begin{align}\label{eq:intBd}
&\int_{\sin^2 \theta < \sin^2\theta_0}  \frac{1}{|v'|^2(1+ |v|^2 +2\langle v,v'\rangle+ |v'|^2)^2}dv'\\
&\quad =\int_0^\infty \int_0^{2\pi} \int_{[0,\theta_0]\cup [\pi- \theta_0,\pi]}  \frac{\sin\theta}{(1+ v_\|^2 + v_\perp^2 +2r(v_\| \cos \theta + v_\perp \cos \alpha \sin \theta)+ r^2)^2} d\theta d\alpha dr
\end{align}
We now rewrite
\begin{align}
&v_\|^2 + v_\perp^2 +2r(v_\| \cos \theta + v_\perp \cos \alpha \sin\theta)  + r^2 \\
&\quad =v_\|^2  \sin^2 \theta + v_\perp^2 (\cos^2\theta +\sin^2\alpha \sin^2\theta)  + (r +v_\| \cos \theta + v_\perp \cos \alpha \sin\theta )^2 \\
&\quad \geq v_\perp^2 \cos^2\theta + (r +v_\| \cos \theta + v_\perp \cos \alpha \sin\theta )^2
\end{align}
Extending the domain of integration to be $r \in (-\infty,\infty)$, and using shift invariance in $r$, and then the symmetries of $\sin\theta$, we bound the integral by
\begin{align}
 \eqref{eq:intBd} &\lesssim \int_{\mathbf R} \int_0^{\theta_0}  \frac{\sin\theta}{(1 + v_\perp^2 \cos^2 \theta +  r^2)^2} d\theta dr
 \end{align}
 Taking $\theta_0 \leq \frac{\pi}{4}$, we have $\cos^2\theta \geq \frac{1}{2}$. This ensures that 
 \begin{align}
 \eqref{eq:intBd} &\lesssim \int_{\mathbf R} \int_0^{\theta_0}  \frac{\sin\theta}{(1+ v_\perp+|r|)^4} d\theta dr\\
&\lesssim \frac{1}{\langle v_\perp\rangle^3} \int_0^{\theta_0} \sin\theta d\theta \\
&\lesssim  \frac{\theta_0^2}{\langle v_\perp\rangle^3}.
\end{align}
Combining with \eqref{eq:angleDecomp}, we have
\begin{align}
\Phi_{ij} * G(v)\nu_i\nu_j \gtrsim \theta_0^2 \left(\frac{1}{k^4_G\langle v\rangle}-k_G \frac{\theta_0}{\langle v_\perp\rangle^{\frac{3}{2}}}\right)
\end{align}
At this point, we take $\theta_0 = \lambda \frac{\langle v_\perp\rangle}{\langle  v\rangle}$, where $\lambda \in (0,\frac{\pi}{4})$ is to be chosen later (not necessarily the same $\lambda$ from before).  In particular, $v_\| \theta_0 \lesssim \lambda \langle v_\perp\rangle$, so
\begin{align}
\Phi_{ij} * G(v)\nu_i\nu_j \gtrsim \frac{\lambda^2\langle v_\perp\rangle^2}{\langle v\rangle^2} \left(\frac{1}{k^5_G\langle v\rangle}-k_G \frac{\lambda}{ \langle v_\perp \rangle^{\frac{1}{2}}\langle v\rangle}\right).
\end{align}
Taking $\lambda = \frac{1}{10k_G^6}$, we deduce
\begin{align}
\Phi_{ij} * G(v)\nu_i\nu_j \gtrsim \frac{\langle v_\perp\rangle^2}{k_G^{17} \langle v\rangle^3} = \frac{1}{k_G^{17}}( \frac{|P_{v} \nu|^2}{ \langle v\rangle^3} +  \frac{|P_{v^\perp} \nu|^2}{ \langle v\rangle}).
\end{align}
We deduce \eqref{eq:PhiConvLower}.
\end{proof}
\section{Estimates on the ion distribution}
\label{sec:ion}
\subsection{Boundedness of the ion distribution}

We now prove a priori estimates for the ion distribution.

\begin{proposition}\label{prop:apriori}

 Let $M, T, \eps > 0$. 
Suppose $(F_-^\eps,F_+^\eps)$ is a weak solution to \eqref{eq:VPLRescaled} with $0\leq F_{\pm}^\eps  \in C([0,T]; \mathfrak E) \cap L^2([0,T];\mathfrak D)$. Assume the following bootstrap assumptions:
\begin{align}
\sup_{\pm \in \{-,+\}, 0\leq t\leq T}(\|F_\pm^\eps(t)\|_{\mathfrak E}^2 + \|\frac{1}{n_\pm^\eps(t)}\|_{L^\infty_x})\leq M\label{eq:bootstrapF}
\end{align}
In particular, there exists $\mathscr X_+^\eps(t)\sim_{M} \|F_+^\eps\|_{\mathfrak E}^2$ (depending on $M$) such that
\begin{align}\label{eq:F_+apriori}
\frac{d}{dt} \mathscr X_+^\eps +  \|F_+^\eps(t)\|_{\mathfrak D^+_\eps}^2\leq C_{M} \langle \|F_-^\eps\|_{\mathfrak D}\rangle^\frac{3}{2}.
\end{align}
Alternatively,
\begin{align}\label{eq:F_+aprioriAlt}
\frac{d}{dt} \|F_+^\eps\|_{\mathfrak E}^2 +\frac{1}{C_M}  \|F_+^\eps(t)\|_{\mathfrak D^+_\eps}^2\leq C_{M}( \langle \|F_-^\eps\|_{\mathfrak D}\rangle^\frac{3}{2} + \|\langle v\rangle^{m_1}F_+^\eps(t)\|_{L^2_x(\mathcal H_\sigma \cap \mathcal H_{\sigma;\eps}^+)_v}^2).
\end{align}


\end{proposition}
\begin{proof}  It is  convenient to ignore dependence on $M$: we write $C = C_{M}$, and ``$\lesssim$" instead of ``$\lesssim_M$." We also drop the dependence on $\eps$, for instance,  $F_+ =F_+^\eps$. 
Given $m', s' \in \mathbf R$, we define
\begin{align}
 F_+^{(m',s')}&:= \langle v\rangle^{m'} \langle \nabla_x \rangle^{s'}F_+,\\
F_-^{(m',s')} &:= \langle \xi\rangle^{m'}\langle \nabla_x \rangle^{s'}F_-.
\end{align}
It is also convenient to split the collision operators into their ``diffusion" (second order) and ``transport" (first order) parts in divergence form:
\begin{align}
Q(G_1,G_2)&= Q_D(G_1,G_2) + Q_T(G_1,G_2),\\
\end{align}
where
\begin{align}
Q_D(G_1,G_2)&:= \partial_{v_j} ((\Phi_{ij}*_v G_1) \partial_{v_i} G_2),\\
Q_T(G_1,G_2)&:=  \partial_{v_j}(\partial_{v_i}(\Phi_{ij}*_v G_1) G_2).
\end{align}
We split $Q^{\eps}_{\pm \mp} =Q^{\eps}_{\pm \mp,D} + Q^{\eps}_{\pm \mp,T}$ similarly. We now proceed with the proof, broken into several steps.\\

\noindent \textbf{Step 1:} We have the following estimates:
\begin{align}\label{eq:F_+weightapriori}
&\frac{d}{dt} \|F_+^{(m_2,0)}\|_{L^2_{x,v}}^2 + \frac{1}{C}  \| F_+^{(m_2,0)}\|_{L^2_{x}(\dot{\mathcal H}_\sigma\cap \mathcal  H_{\sigma;\eps}^+)_v}^2  \leq C(1+  \|F_-\|_{\mathfrak D}^\frac{3}{2}),
\end{align}
and for all $\kappa >0$,
\begin{equation}\label{eq:F_+regapriori}
\begin{aligned}
&\frac{d}{dt} \|F_+^{(m_1,s)}\|_{L^2_{x,v}}^2 + \frac{1}{C}  \|F_+^{(m_1,s)}\|_{L^2_{x}(\dot{\mathcal H}_\sigma\cap \mathcal H_{\sigma;\eps}^+)_v}^2 \\
&\quad \leq C (1 +  \| F_+^{(m_1,0)}\|_{L^2_x(\dot{\mathcal H}_\sigma\cap \mathcal H_{\sigma;\eps}^+)_v}^2 +  \|F_-\|_{\mathfrak D}^\frac{3}{2}).
\end{aligned}
\end{equation}
The proof of \eqref{eq:F_+weightapriori} follows by a similar method as \eqref{eq:F_+regapriori}, so we only show the latter. To prove \eqref{eq:F_+regapriori}, 
observe that $F_+^{(m_1,s)}$ satisfies
\begin{align}
&(\partial_t  + v\cdot \nabla_x + E\cdot \nabla_v )F_+^{(m_1,s)} + [ \langle v\rangle^{m_1} \langle \nabla_x\rangle^s, E\cdot \nabla_v] F_+ \\
&= \langle v\rangle^{m_1} \langle \nabla_x\rangle^s \{Q(F_+,F_+)  + Q_{-+}^\eps(F_-,F_+)\}.
\end{align}
Multiplying the above by  by $F_+^{(m_1,s)}$ integrating, we have
\begin{align}
&\frac{1}{2}\frac{d}{dt} \|F_+^{(m_1,s)}\|_{L^2_{x,v}} ^2  \\
&\quad = \langle [ E\cdot \nabla_v,\langle v\rangle^m \langle \nabla_x\rangle^s] F_+ ,F_+^{(m_1,s)}\rangle_{L^2_{x,\xi }} \label{eq:F_+mom+derivT1}\\
&\quad \quad + \langle \langle v\rangle^{m_1}\langle \nabla_x\rangle^s Q(F_+,F_+), F_+^{(m_1,s)}\rangle_{L^2_{x,v}} \label{eq:F_+mom+derivT2}\\
&\quad \quad + \langle \langle v\rangle^{m_1}\langle \nabla_x\rangle^s Q_{-+}^\eps(F_-,F_+), F_+^{(m_1,s)}\rangle_{L^2_{x,v}}\label{eq:F_+mom+derivT3}
\end{align}
What follows are  estimates for each term on the right-hand side. \\

\noindent \textit{Step 1.1 (electric field commutator):} We now bound  \eqref{eq:F_+mom+derivT1}:
\begin{align}
\eqref{eq:F_+mom+derivT1} \lesssim 1+\|F_+^{(m_1,s)}\|_{L^2_x( \mathcal H_\sigma)_v} .\label{eq:F_+MDT1Bd}
\end{align}
Observe first that
\begin{align}
 [ E\cdot \nabla_v,\langle v\rangle^m \langle \nabla_x\rangle^s] F_+ &= E\cdot \nabla_v(\langle v\rangle^{m_1}  \langle \nabla_x\rangle^sF_+) -\langle v\rangle^{m_1}   E\cdot \nabla_v\langle \nabla_x\rangle^s F_+  \\
 &\quad \langle v\rangle^{m_1} E\cdot \nabla_v\langle \nabla_x\rangle^s F_+  -\langle v\rangle^{m_1}  \langle \nabla_x\rangle^s (E\cdot \nabla_v F_+) \\
 &= v\cdot E F_+^{(m_1-2,s)} + \langle v\rangle^{m_1} \nabla_v \cdot (E \langle \nabla_x\rangle^s F_+ -  \langle \nabla_x\rangle^s (E F_+))
\end{align}
Thus,
\begin{align}
\eqref{eq:F_+mom+derivT1} &= \langle v\cdot E F_+^{(m-2,s)}, F_+^{(m_1,s)}\rangle_{L^2_{x,v}}\\
&\quad -  \langle   E \langle \nabla_x\rangle^s F_+ -  \langle \nabla_x\rangle^s (E F_+), \nabla_v (\langle v\rangle^{m_1}F^{(m_1,s)}_+)\rangle_{L^2_{x,v}}\\
&\lesssim \|E\|_{L^\infty_x} \|F_+^{(m_1,s)}\|_{L^2_{x,v}}^2 \\
&\quad +( \|\nabla_x E\|_{L^\infty_x}  \| F_+^{(m_1 + \frac{3}{2},s-1)}\|_{L^2_{x,v}}  +\|E\|_{H^s_x}  \| F_+^{(m_1 + \frac{3}{2},0)}\|_{L^2_{x,v}}  )  \|F_+^{(m_1-\frac{3}{2},s)}\|_{L^2_xH^1_v}\\
&\lesssim (\| F_+^{(m_1 + \frac{3}{2},s-1)}\|_{L^2_{x,v}}   + \| F_+^{(m_2,0)}\|_{L^2_{x,v}}  ) ( 1+   \|F_+^{(m_1,s)}\|_{L^2_x(\dot{\mathcal H}_\sigma)_v})
\end{align}
Next, note that we have the interpolation inequality
\begin{equation}
   \| F^{(m_1 + \frac{3}{2},s-1)}\|_{L^2_{x,v}} \lesssim \|F_+^{(m_1,s)}\|_{L^2_{x,v}}^{\frac{s-1}{s}} \| F^{(m_1 + \frac{3s}{2},0)}\|_{L^2_{x,v}}^{\frac{1}{s}} \lesssim \|F_+\|_{\mathfrak E} \lesssim 1,
\end{equation}
since $m_1 + \frac{3s}{2} \leq m_2$. We conclude with \eqref{eq:F_+MDT1Bd}.\\

\noindent \textit{Step 1.2 (self-collisions):} Next, we have the following bound on \eqref{eq:F_+mom+derivT2}:
\begin{equation}
\eqref{eq:F_+mom+derivT2}\leq -\frac{1}{C}\|F_+^{(m_1, s)}\|_{L^2_x (\dot{\mathcal H}_\sigma)_v}^2+ C(1 + \| F_+^{(m_1,0)}\|_{L^2_{x}(\dot{\mathcal H}_\sigma)_v}^2).\label{eq:F_+MDT2Bd}
\end{equation}
For this, we separate
\begin{align}
\eqref{eq:F_+mom+derivT2}& = \langle Q_D(F_+,  F_+^{(m_1,s)}), F_+^{(m_1,s)}\rangle_{L^2_{x,v}} \label{eq:F_+mom+derivT2main}\\
&\quad + \langle  \langle v\rangle^{m_1}\langle \nabla_x\rangle^sQ_D(F_+,F_+) - Q_D(F_+,  F_+^{(m_1,s)}), F_+^{(m_1,s)}\rangle_{L^2_{x,v}}\label{eq:F_+mom+derivT2comm} \\
&\quad +  \langle  \langle v\rangle^{m_1}\langle \nabla_x\rangle^sQ_T(F_+,F_+), F_+^{(m_1,s)}\rangle_{L^2_{x,v}}\label{eq:F_+mom+derivT2trans}.
\end{align}
 Using \eqref{eq:PhiConvLower}, we have the bound
\begin{align}
\eqref{eq:F_+mom+derivT2main}& =-\langle ( \Phi_{ij} *_v F_+) \partial_{v_i} F_+^{(m_1,s)},  \partial_{v_j}F_+^{(m_1,s)}\rangle_{L^2_{x,v}}  \\
 &\leq -  \frac{1}{C_{k_{F_+}}} \|F_+^{(m_1,s)}\|_{L^2_x(\dot{\mathcal H}_\sigma)_v}.
\end{align}
where
\begin{equation}
k_{F_+} = \max\{1,\|\langle v\rangle^{5} F_+\|_{L^\infty_xL^2_v }, \|n_+^{-1}\|_{L^\infty_x}\}.
\end{equation}
Now,
\begin{equation}
\|\langle v\rangle^{5} F_+\|_{L^\infty_xL^2_v } \lesssim \| F_+^{(5,2)}\|_{L^2_{x,v} } \lesssim 1.
\end{equation}
so $k_{F_+} \lesssim 1$.
Using these estimates, we get that
\begin{align}
\eqref{eq:F_+mom+derivT2main}& \leq -\frac{1}{C}\|F_+\|_{L^2_x (\dot{\mathcal H}_\sigma)_v}^2
\end{align}
Next, we have the commutator term \eqref{eq:F_+mom+derivT2comm}. Let us first write
\begin{align}
&\langle v\rangle^{m_1}\langle \nabla_x\rangle^sQ_D(F_+,F_+)- Q_D(F_+,F_+^{(m_1,s)}) \\
&= \langle \nabla_x\rangle^s Q_D(F_+, F_+^{(m_1,0)})- Q_D(F_+,F_+^{(m_1,s)}) \label{eq:MDcommT1}\\
&\quad  +   \langle \nabla_x\rangle^s  \{\langle v\rangle^m Q_D(F_+, F_+) -  Q_D(F_+, F_+^{(m_1,0)})\}. \label{eq:MDcommT2}
\end{align}
Now, let
\begin{equation}
\mathcal C_j := \langle \nabla_x\rangle^s((\Phi_{ij}*_v F_+) \partial_{v_i} F_+^{(m_1,0)})-(\Phi_{ij}*_v F_+) \partial_{v_i} F_+^{(m_1,s)}\end{equation}
Then
\begin{align}
\langle \eqref{eq:MDcommT1}, F^{(m_1,s)}_+\rangle_{L^2_{x,v}} &= \langle \mathcal C_i, \partial_{v_i} F_+\rangle_{L^2_{x,v}}\\
&\quad \leq  \langle (\sigma^{-1})_{ij} \mathcal C_i, \mathcal C_j\rangle_{L^2_{x,v}}^{\frac{1}{2}} \langle \sigma_{ij} \partial_{v_i} F_+^{(m_1,s)},\partial_{v_j} F_+^{(m_1,s)}\rangle_{L^2_{x,v}}^{\frac{1}{2}}
\end{align}
Now, the latter factor is bounded by $\|F_+\|_{L^2_{x}(\mathcal H_\sigma)_v}$. On the other hand, writing
\begin{equation}
A_{G} = \sigma^{-\frac{1}{2}} \Phi*_vG  \sigma^{-\frac{1}{2}}.
\end{equation}
Note that by \eqref{eq:PhiConvUpper},  $\|A_{G}\|_{L^\infty_{v}} \lesssim \|\langle v\rangle^5 G\|_{L^2_v}$ given a function $G = G(v)$. Therefore,
we have
\begin{align}
\|\sigma^{-\frac{1}{2}}\mathcal C\|_{L^2_{x,v}}&=\|\langle \nabla_x\rangle^s(A_{F_+} \sigma^{\frac{1}{2}} \nabla_v F_+^{(m_1,0)})-A_{F_+} \sigma^{\frac{1}{2}} \nabla_v F_+^{(m_1,s)}\|_{L^2_{x,v}}\\
&\lesssim  \|A_{F_+^{(0,s)}}\|_{L^2_xL^\infty_v} \|\sigma^{\frac{1}{2}} \nabla_v F_+^{(m_1,s-1)}\|_{L^2_{x,v}}\\
&\lesssim \| F_+^{(m_1,s-1)}\|_{L^2_{x}(\dot{\mathcal H}_\sigma)_v}\\
&\lesssim \| F_+^{(m_1,0)}\|_{L^2_{x}(\dot{\mathcal H}_\sigma)_v}^{\frac{1}{s}}\| F_+^{(m_1,s)}\|_{L^2_{x}(\dot{\mathcal H}_\sigma)_v}^{\frac{s-1}{s}}.
\end{align}
Thus,
\begin{align}
\langle \eqref{eq:MDcommT1}, F^{(m_1,s)}_+\rangle_{L^2_{x,v}}  \lesssim \| F_+^{(m_1,0)}\|_{L^2_{x}(\dot{\mathcal H}_\sigma)_v}^{\frac{1}{s}}\| F_+^{(m_1,s)}\|_{L^2_{x}(\dot{\mathcal H}_\sigma)_v}^{\frac{2s-1}{s}}.
\end{align}
Next, we have
\begin{align}
\eqref{eq:MDcommT2} &=-m_1 \langle \nabla_x \rangle^s\{ v_j( \Phi_{ij} *_v F_+)\partial_{v_i} F_+^{(m_1-2,0)}  + \partial_{v_j} (v_i( \Phi_{ij} *_v F_+)F_+^{(m_1-2,0)} )\\
&\quad - (m_1-2)v_i v_j( \Phi_{ij} *_v F_+) F_+^{(m_1-4,0)}  \}.
\end{align}
Thus,
\begin{align}
\langle \eqref{eq:MDcommT2} , F^{(m_1,s)}_+ \rangle_{L^2_{x,v}} &= -m_1 \langle \langle \nabla_x \rangle^s\{ v_j( \Phi_{ij} *_v F_+)\partial_{v_i} F_+^{(m_1-2,0)}\},F^{(m_1,s)}_+\rangle_{L^2_{x,v}}   \\
&\quad +m_1\langle \langle \nabla_x\rangle^s \{\langle v\rangle^{m_1-2} v_i( \Phi_{ij} *_v F_+)F_+^{(m_1-2,0)} \}, \partial_{v_j} F^{(m_1,s)}_+ \rangle_{L^2_{x,v}}\\
&\quad +m_1 (m_1-2)\langle \langle \nabla_x\rangle^s \{v_i v_j( \Phi_{ij} *_v F_+) F_+^{(m_1-4,0)}\} ,F^{(m_1,s)}_+\rangle_{L^2_{x,v}}\\
&\lesssim  \Big \||\sigma^{\frac{1}{2}} v| \|A_{F_+^{(0,s)}}\|_{L^2_x} \|\sigma^{\frac{1}{2}} \nabla_v F_+^{(m_1-2,s)}\|_{H^s_x} \Big\|_{L^2_v}\|F_+^{(m_1,s)}\|_{L^2_{x,v}} \\
&\quad + \Big \||\sigma^{\frac{1}{2}} v| \|A_{F_+^{(0,s)}}\|_{L^2_x} \| F_+^{(m_1-2,s)}\|_{H^s_x} \Big\|_{L^2_v} \|\sigma^{\frac{1}{2}}\nabla_v F_+^{(m_1,s)}\|_{L^2_{x,v}}\\
&\quad +  \Big \|(v\cdot (\sigma^{\frac{1}{2}} v) )\|A_{F_+^{(0,s)}}\|_{L^2_x} \| F_+^{(m_1-4,s)}\|_{H^s_x} \Big\|_{L^2_v} \|F_+^{(m_1,s)}\|_{L^2_{x,v}}\\
&\lesssim  \|F_+^{(m_1,s)}\|_{L^2_x(\dot{\mathcal H}_\sigma)_v}.
\end{align}
Therefore,
\begin{align}
\eqref{eq:F_+mom+derivT2comm} \lesssim \|F_+^{(m_1,s)}\|_{L^2_x(\dot{\mathcal H}_\sigma)_v} +\| F_+^{(m_1,0)}\|_{L^2_{x}(\dot{\mathcal H}_\sigma)_v}^{\frac{1}{s}}\| F_+^{(m_1,s)}\|_{L^2_{x}(\dot{\mathcal H}_\sigma)_v}^{\frac{2s-1}{s}}.
\end{align}
Next we estimate \eqref{eq:F_+mom+derivT2trans}. We bound this as follows:
\begin{align}
\eqref{eq:F_+mom+derivT2trans} &= \langle \langle \nabla_x \rangle^s \{\partial_{v_i}(\Phi_{ij} *_vF_+)F_+^{(m_1,0)}\}, \partial_{v_j}F^{(m_1,s)}\rangle_{L^2_{x,v}} \\
&\quad +  \langle \langle \nabla_x \rangle^s \{v_j\partial_{v_i}(\Phi_{ij} *_vF_+)F_+^{(m_1-2,0)}\}, F^{(m_1,s)}\rangle_{L^2_{x,v}} \\
&\lesssim \|\langle \nabla_x\rangle^s\{\sigma^{-\frac{1}{2}} (\nabla_v\cdot \Phi*_v F_+) F_+\}\|_{L^2_{x,v}} \|\sigma^{\frac{1}{2}} \nabla_v F_+^{(m_1,s)}\|_{L^2_{x,v}} \\
&\quad + \|\langle \nabla_x \rangle^s \{v_j\partial_{v_i}(\Phi_{ij} *_vF_+)F_+^{(m_1-2,0)}\} \|_{L^2_{x,v}} \|F_+^{(m_1,s)}\|_{L^2_{x,v}}\\
&\lesssim \|\sigma^{-\frac{1}{2}}\nabla_v (-\Delta_v)^{-1} F_+^{(0,s)}\|_{L^2_xL^\infty_v} \|F_+^{(m_1,s)} \|_{L^2_{x,v}}\|F_+^{(m_1,s)}\|_{L^2_{x}(\dot{\mathcal H}_\sigma)_v}\\
&\quad + \|\nabla_v (-\Delta_v)^{-1} F_+^{(0,s)}\|_{L^2_xL^\infty_v} \|F_+^{(m_1,s)} \|_{L^2_{x,v}} \|F_+^{(m_1,s)}\|_{L^2_{x,v}}\\
&\lesssim  \|\langle v\rangle^{\frac{3}{2}}\nabla_v (-\Delta_v)^{-1} F_+^{(0,s)}\|_{L^2_xL^\infty_v} ( 1+ \|F_+^{(m_1,s)}\|_{L^2_{x}(\dot{\mathcal H}_\sigma)_v}).
\end{align}
Now, by
\begin{align}
\langle v\rangle^\frac{3}{2}|\nabla_v (-\Delta_v)^{-1} F_+^{(0,s)}|&\lesssim \langle v\rangle^\frac{3}{2}\int_{\langle v\rangle > 2|v'|} \frac{|F_+^{(0,s)}(v')|}{|v-v'|^2} dv' +\langle v\rangle^\frac{3}{2} \int_{\langle v\rangle \leq 2|v'|} \frac{|F_+^{(0,s)}(v')|}{|v-v'|^2} dv'\\
&\lesssim  \langle v\rangle^{-\frac{1}{2}}\|F_+^{(0,s)}\|_{L^1_v} + \int_{\mathbf R^3}\frac{|F_+^{(\frac{3}{2},s)}(v')|}{|v-v'|^2} dv'
\end{align}
Thus,
\begin{align}
\|\langle v\rangle^{\frac{3}{2}}\nabla_v (-\Delta_v)^{-1} F_+^{(0,s)}\|_{L^2_xL^\infty_v} &\lesssim \|F^{(0,s)}_+\|_{L^2_xL^1_v} + \||\nabla_v|^{-1} |F_+^{(\frac{3}{2}, s)}|\|_{L^2_xL^\infty_v} \\
&\lesssim \|F^{(2,s)}_+\|_{L^2_{x,v}} +   \|\langle \nabla_v\rangle^{\frac{3}{4}}|F_+^{(\frac{3}{2}, s)}|\|_{L^2_{x,v}}\\
&\lesssim 1 +\||F_+^{(\frac{3}{2}, s)}|\|_{L^2_x H^1_v}^{\frac{3}{4}}.\\
&\lesssim 1 + \|F_+^{(m_1, s)}\|_{L^2_x (\dot{\mathcal H}_\sigma)_v}^{\frac{3}{4}}
\end{align}
In the above, we use the interpolation inequality $\|u\|_{H^{3/4}} \lesssim \|u\|_{L^2}^{\frac{1}{4}} \|u\|_{H^1}^{\frac{3}{4}}$, followed by the fact that $\big |\nabla_v |u|\big| = |\nabla_v u|$ a.e., for any  function $u \in H^1(\mathbf R^3)$.
Thus,
\begin{equation}
\eqref{eq:F_+mom+derivT2trans} \lesssim 1 + \|F_+^{(m_1, s)}\|_{L^2_x (\dot{\mathcal H}_\sigma)_v}^{\frac{7}{4}}.
\end{equation}
Combining the estimates \eqref{eq:F_+mom+derivT2main}, \eqref{eq:F_+mom+derivT2comm}  and \eqref{eq:F_+mom+derivT2trans}, we conclude
\begin{align}
 \eqref{eq:F_+mom+derivT2} & \leq - \frac{1}{C}\|F_+^{(m_1, s)}\|_{L^2_x (\dot{\mathcal H}_\sigma)_v}^2 \\
& \quad + C (1+ \|F_+^{(m_1, s)}\|_{L^2_x (\dot{\mathcal H}_\sigma)_v}^{\frac{7}{4}}  + \| F_+^{(m_1,0)}\|_{L^2_{x}(\dot{\mathcal H}_\sigma)_v}^{\frac{1}{s}}\| F_+^{(m_1,s)}\|_{L^2_{x}(\dot{\mathcal H}_\sigma)_v}^{\frac{2s-1}{s}}).
 \end{align}
We then use H\"older's inequality to get \eqref{eq:F_+MDT2Bd}. \\

\noindent \textit{Step 1.3 (control of ion-electron collisions):} 
Now, we turn to \eqref{eq:F_+mom+derivT3}, for which we have the bound 
\begin{equation} 
\begin{aligned}
\eqref{eq:F_+mom+derivT3} &\leq -\frac{1}{C}\{\|F_+^{(m_1,s)}\|_{L^2_x(\mathcal H_{\sigma;\eps}^+)_v}^2 - \lambda \|F_+^{(m_1,s)}\|_{L^2_x(\dot{\mathcal H}_\sigma)_v}^2\} \\
&\quad + C \|F^{(m_1,0)}_+\|_{L^2_x (\mathcal H^+_{\sigma;\eps})_v}^2\\
&\quad + C_{\lambda} (1 + \|F_-^{(m_1,s)}\|_{L^2_x(\dot{\mathcal H}_\sigma)_\xi }^{\frac{3}{2}}).\end{aligned}\label{eq:F_+MDT3Bd}
\end{equation}
for all $\lambda >0$.
We decompose this term in the same way as \eqref{eq:F_+mom+derivT2},
\begin{align}
\eqref{eq:F_+mom+derivT3}& = \langle Q_{-+,D}^\eps(F_-,  F_+^{(m_1,s)}), F_+^{(m_1,s)}\rangle_{L^2_{x,v}} \label{eq:F_+MDT3main}\\
&\quad + \langle  \langle v\rangle^{m_1}\langle \nabla_x\rangle^sQ_{-+,D}^\eps(F_-,F_+) - Q_{-+,D}^\eps(F_-,  F_+^{(m_1,s)}), F_+^{(m_1,s)}\rangle_{L^2_{x,v}}\label{eq:F_+MDT3comm} \\
&\quad +  \langle  \langle v\rangle^{m_1}\langle \nabla_x\rangle^sQ_{-+,T}^\eps(F_-,F_+), F_+^{(m_1,s)}\rangle_{L^2_{x,v}}\label{eq:F_+MDT3trans}
\end{align}
The terms \eqref{eq:F_+MDT3main} and \eqref{eq:F_+MDT3comm} using the same approach as  \eqref{eq:F_+mom+derivT2main} and \eqref{eq:F_+mom+derivT2comm},
\begin{align}
 \eqref{eq:F_+MDT3main} &\leq -\frac{1}{C} \|F^{(m_1,s)}_+\|_{L^2_x (\mathcal H^+_{\sigma;\eps})_v}^2,  \\
 \eqref{eq:F_+MDT3comm} &\lesssim \|F^{(m_1,s)}_+\|_{L^2_x (\mathcal H^+_{\sigma;\eps})_v} +   \|F^{(m_1,0)}_+\|_{L^2_x (\mathcal H^+_{\sigma;\eps})_v}^\frac{1}{s}\|F^{(m_1,s)}_+\|_{L^2_x (\mathcal H^+_{\sigma;\eps})_v}^\frac{2s-1}{s} .
\end{align}
The final term \eqref{eq:F_+MDT3trans} requires a different approach:  we separate it into a main term and commutators,
\begin{align}
\eqref{eq:F_+MDT3trans}&=\langle Q_{-+,T}^\eps(F_-,F_+^{(m_1,s)}), F^{(m_1,s)}\rangle_{L^2_{x,v}} \label{eq:F_+MDT3transMain} \\
 &\quad  + \langle  \langle \nabla_x\rangle^sQ_{-+,T}^\eps(F_-,F_+^{(m_1,0)}) -Q_{-+,T}^\eps(F_-,F_+^{(m_1,s)}) , F^{(m_1,s)}\rangle_{L^2_{x,v}} \label{eq:F_+MDT3transComm1}\\
&\quad  + \langle  \langle \nabla_x\rangle^s\{\langle v\rangle^{m_1}Q_{-+,T}^\eps(F_-,F_+) -Q_{-+,T}^\eps(F_-,F_+^{(m_1,0)}) \}, F^{(m_1,s)}\rangle_{L^2_{x,v}}. \label{eq:F_+MDT3transComm2}
\end{align}
For the first term, we note that $\partial_{v_i}\partial_{v_j} \Phi_{ij}(v) = -8\pi \delta(v)$ in the sense of distributions, where $\delta$ is the Dirac mass. Hence,
\begin{align}
\eqref{eq:F_+MDT3transMain} &= 4\pi \eps \langle F_-|_{\xi  = \eps v} F_+^{(m_1,s)},  F_+^{(m_1,s)}  \rangle_{L^2_{x,v}} \\
&\lesssim \eps \| \langle \eps v\rangle^{\frac{3}{4}}F_-|_{\xi =\eps v}\|_{L^\infty_xL^6_v} \|\langle \eps v\rangle^{-\frac{3}{4}}F_+^{(m_1,s)}\|_{L^2_xL^3_v} \|F_+^{(m_1,s)}\|_{L^2_{x,v}}\\
&\lesssim \eps^{\frac{1}{2}} \|F_-^{(\frac{3}{4},0)}\|_{L^\infty_xL^6_\xi } \|\langle \eps v\rangle^{-\frac{3}{4}}F_+^{(m_1,s)}\|_{L^2_xL^3_v}
\end{align}
Now, using the generalized Minkowski inequality and Sobolev embedding,
\begin{equation}
\|F_-^{(\frac{3}{4},0)}\|_{L^\infty_xL^6_\xi }   \lesssim  \|\nabla_v F_-^{(\frac{3}{4},0)}\|_{L^\infty_x L^2_\xi}\lesssim  \|\nabla_v F_-^{(\frac{3}{4},s)}\|_{L^2_{x,\xi} } \lesssim 1+  \|F_-^{(\frac{9}{4},s)}\|_{L^2_x({\mathcal H}_\sigma)_\xi}
\end{equation}
 On the other hand,
\begin{align}
\|\langle \eps v\rangle^{-\frac{3}{4}}F_+^{(m_1,s)}\|_{L^2_xL^3_v} & \lesssim  \|F_+^{(m_1,s)}\|_{L^2_{x,v}}^{\frac{1}{2}}\|\langle \eps v\rangle^{\frac{3}{2}}F_+^{(m_1,s)}\|_{L^2_xH^1_v}^{\frac{1}{2}}\\
&\lesssim  \eps^{-\frac{1}{4}} \|F_+^{(m_1,s)}\|_{L^2_x(\mathcal H_{\sigma;\eps}^+)_v}^{\frac{1}{2}}
\end{align}
We conclude that
\begin{equation}
\eqref{eq:F_+MDT3transMain} \leq C \eps^{\frac{1}{4}} \|F_+^{(m_1,s)}\|_{L^2_x(\mathcal H_{\sigma;\eps}^+)_v}^{\frac{1}{2}}\|F_-^{(m_1,s)}\|_{L^2_x(\mathcal H_\sigma)_\xi}.
\end{equation}
As for \eqref{eq:F_+MDT3transComm1}, we have
\begin{align}
\eqref{eq:F_+MDT3transComm1} &= 8\pi \langle \langle \nabla_x\rangle^s \{\nabla_\xi (-\Delta_\xi )^{-1}F_-|_{\xi = \eps v} F_+^{(m_1,0)}\} -\nabla_\xi (-\Delta_\xi )^{-1}F_-|_{\xi = \eps v} F_+^{(m_1,s)}, \nabla_v F_+^{(m_1,s)}\rangle_{L^2_{x,v}}\\
&\lesssim \Big\|\langle v\rangle^{\frac{3}{2}}\||\nabla_\xi |^{-1}|F_-^{(0,s)}|_{\xi =\eps v}|\|_{L^2_x}\|F_+^{(m_1,s-1)}\|_{L^2_x} \Big\|_{L^2_v} \| F_+^{(m_1,s)}\|_{L^2_x(\dot{\mathcal H}_\sigma)_v}\\
&\lesssim  \||\nabla_\xi |^{-1}|F_-^{(0,s)}|\|_{L^2_xL^\infty_\xi }\|F_+^{(m+\frac{3}{2},s-1)}\|_{L^2_{x,v}}\| F_+^{(m_1,s)}\|_{L^2_x(\dot{\mathcal H}_\sigma)_v} \\
&\lesssim \||F_-^{(0,s)}|\|_{L^2_xH^{\frac{3}{4}}_\xi }\|F_+^{(m+\frac{3s}{2},0)}\|_{L^2_{x,v}}^{\frac{1}{s}}\|F_+^{(m_1,s)}\|_{L^2_{x,v}}^{\frac{s-1}{s}}\| F_+^{(m_1,s)}\|_{L^2_x(\dot{\mathcal H}_\sigma)_v}\\
&\lesssim \|F_-^{(m_1,s)}\|_{L^2_x(\mathcal H_\sigma)_\xi }^{\frac{3}{4}}\| F_+^{(m_1,s)}\|_{L^2_x(\dot{\mathcal H}_\sigma)_v}.
\end{align}
The second commutator \eqref{eq:F_+MDT3transComm1} is bounded as follows:
\begin{align}
\eqref{eq:F_+MDT3transComm1}& = 8\pi \langle \langle \nabla_x\rangle^s \{v\cdot \nabla_\xi (-\Delta_\xi )^{-1}F_-|_{\xi = \eps v} F_+^{(m_1-2,0)}\}, F_+^{(m_1,s)}\rangle_{L^2_{x,v}} \\
&\lesssim \||\nabla_v |^{-1}F_-^{(0,s)}\|_{L^2_xL^\infty_\xi } \|F_+^{(m_1,s)}\|_{L^2_{x,v}}^2 \\
&\lesssim \|F_-^{(m_1,s)}\|_{L^2_x (\mathcal H_\sigma)_\xi }.
\end{align}
Thus,
\begin{align}
\eqref{eq:F_+MDT3trans}&\lesssim (\|F_-^{(m_1,s)}\|_{L^2_x (\mathcal H_\sigma)_\xi } + \eps^{\frac{1}{4}}  \|F_+^{(m_1,s)}\|_{L^2_x(\mathcal H_{\sigma;\eps}^+)_v}^{\frac{1}{2}}\|F_-^{(m_1,s)}\|_{L^2_x({\mathcal H}_\sigma)_\xi} \\
&\quad +\| F_+^{(m_1,s)}\|_{L^2_x(\dot{\mathcal H}_\sigma)_v}\|F_-^{(m_1,s)}\|_{L^2_x(\mathcal H_\sigma)_\xi }^{\frac{3}{4}}).
\end{align}
Now, combining the bounds on  \eqref{eq:F_+MDT3main}, \eqref{eq:F_+MDT3comm}, \eqref{eq:F_+MDT3trans}, and Young's inequality, we conclude \eqref{eq:F_+MDT3Bd}.

To conclude step 1, we combine  \eqref{eq:F_+MDT1Bd}, \eqref{eq:F_+MDT2Bd} and \eqref{eq:F_+MDT3Bd}, taking $\lambda$ sufficiently small, to get \eqref{eq:F_+regapriori}.\\

\noindent \textbf{Step 2 (combining the estimates):} Take  $0 < \kappa_1 < \kappa_2$, and set $\mathscr X_+ := \kappa_1 \|F_+^{(m_1,s)}\|_{L^2_{x,v}}^2 + \kappa_2 \|F_+^{(m_2,0)}\|_{L^2_{x,v}}^2$. Then by \eqref{eq:F_+weightapriori} and \eqref{eq:F_+regapriori}, we have
\begin{align}
&\frac{d}{dt} \mathscr X_+ + \frac{\kappa_1}{C}  \|F^{(m_1,s)}\|_{L^2_x(\dot{\mathcal H}_\sigma\cap \mathcal H_{\sigma;\eps}^+)_v}^2 + \left(\frac{\kappa_2}{C} - C\kappa_1\right) \|F^{(m_2,0)}\|_{L^2_x(\dot{\mathcal H}_\sigma\cap \mathcal H_{\sigma;\eps}^+)_v}^2   \\
&\quad \leq C(\kappa_1 + \kappa_2)\langle \|F_-\|_{\mathfrak D}\rangle^{\frac{3}{2}} 
\end{align} 
Taking $\kappa_2$ sufficiently large and $\kappa_1$ sufficiently small depending on $M$, we get \eqref{eq:F_+apriori}.

Alternatively, by simply adding \eqref{eq:F_+weightapriori} and \eqref{eq:F_+regapriori}, we have  \eqref{eq:F_+aprioriAlt}.
\end{proof}

\subsection{Error estimate for the ion distribution}

\begin{proposition}\label{prop:ionError}
Let $M, T>0$, and let $(F^\eps_+,F_-^\eps)$ be a solution to \eqref{eq:VPLRescaled} satisfying the same hypotheses as in \ref{prop:apriori}, and $(F^0_+,\beta,\phi^0)$ be a solution to \eqref{eq:VPLIon}, satisfying the hypothesis of Theorem \ref{thm:derivation}.
Then, there is a function $\mathscr Y_+^\eps:[0,T^*) \to \mathbf R_+$ such that $\mathscr Y_+^\eps \sim_{M} \|F^\eps_+ - F^0_+\|_{\mathfrak E'}$ and
\begin{equation}\label{eq:Gbd}
\begin{aligned}
&\frac{d}{dt} \mathscr Y_+^\eps + \|F_+^\eps - F_+^0\|_{\mathfrak D'}^2 \\
&\quad \lesssim_M \langle \|(F_+^0,F_-^\eps)\|_{\mathfrak D}\rangle^2 ( \eps^2 + \mathscr Y_+^\eps ) +\mathscr D_{-,2}^\eps.
\end{aligned}
\end{equation}

\end{proposition}
\begin{proof}
Let $G = F^\eps_+ - F^0_+$. Then,
\begin{align}
&\partial_t G + \{v\cdot \nabla_x + E^\eps \cdot \nabla_v\}G - (E^\eps - E^0)\cdot \nabla_v F_+^0 - Q(F_+^\eps, G) -Q(G,F_+^0) \\
&\quad = Q_{-+}(F_-^\eps, F^\eps_+) 
\end{align}
Similarly to the proof of the previous proposition, given $m',s' \in\mathbf R$, we denote $G^{(m',s')} = \langle v\rangle^{m'}\langle \nabla_x\rangle^s G$,  $(F_\pm^\eps)^{(m',s')}=\langle v\rangle^{m'}\langle \nabla_x\rangle^s F_\pm^\eps$, etc.
The above gives
\begin{align}
\frac{1}{2}\frac{d}{dt} \|G^{(m_0,s)}\|_{L^2_{x,v}}^2 &=  \langle [\langle v\rangle^{m_0} \langle \nabla_x\rangle^s,E^\eps \cdot \nabla_v]G, G^{(m_0,s)}\rangle_{L^2_x}\label{eq:GT1}\\
&\quad + \langle \langle v\rangle^{m_0} \langle \nabla_x\rangle^s\{(E^0 - E^\eps)\cdot \nabla_v F_+^0\}, G^{(m_0,s)}\rangle_{L^2_{x,\xi }}\label{eq:GT2} \\
&\quad + \langle \langle v\rangle^{m_0} \langle \nabla_x\rangle^sQ(F_+^\eps, G), G^{(m_0,s)}\rangle_{L^2_{x,v}} \label{eq:GT3}\\
&\quad +  \langle \langle v\rangle^{m_0}  \langle \nabla_x\rangle^sQ(G,F_+^0), G^{(m_0,s)}\rangle_{L^2_{x,v}}\label{eq:GT4} \\
& \quad + \langle   \langle v\rangle^{m_0} \langle \nabla_x\rangle^s Q_{-+}^\eps (F_-^\eps, F_+^\eps),G^{(m_0,s)}\rangle_{L^2_{x,v}}\label{eq:GT5}.\end{align}
We modify the estimates from Proposition \ref{prop:apriori} to get 
\begin{align}
|\eqref{eq:GT1}| &\lesssim \|G\|_{\mathfrak E'}( \|G\|_{\mathfrak E'} + \|G^{(m_0,s)}\|_{L^2_x(\mathcal H_\sigma)_\xi }) ,\\
|\eqref{eq:GT3}| &\leq -\frac{1}{C} \|G^{(m_0,s)}\|_{L^2_x(\dot{\mathcal H}_\sigma)_v}^2 + C( \langle \|F_+^\eps \|_{\mathfrak D}\rangle ^2\|G\|_{\mathfrak E'}^2 + \|G^{(m_1,0)}\|_{L^2_x(\dot{\mathcal H}_\sigma)_v}^2), \\
|\eqref{eq:GT4}| &\lesssim     \|G\|_{\mathfrak E'}^2 + \|G\|_{\mathfrak E'}^\frac{1}{4}\|G^{(m_0,s)}\|_{L^2_x(\dot{\mathcal H}_\sigma)_v}^{\frac{7}{4}} +  \langle \|F_+^0\|_{\mathfrak D}\rangle  \|G\|_{\mathfrak E'}\|G^{(m_0,s)}\|_{L^2_x(\dot{\mathcal H}_\sigma)_v}
\end{align}
What's left is to bound \eqref{eq:GT2} and \eqref{eq:GT5}. For the former, we have
\begin{align}
| \eqref{eq:GT2} |&\lesssim \|E^0- E^\eps\|_{H^s_x}( \|(F_+^0)^{(m_0,s)}\|_{L^2_{x,\xi }} + \|(F_+^0)^{(m_0+\frac{3}{2},s)}\|_{L^2_{x}(\dot{\mathcal H}_\sigma)_v})\|G^{(m_0,s)}\|_{L^2_{x,v}} \\
&\lesssim( \|G\|_{L^1_vH^{s-1}_x} + \|F_-^\eps - \mu_\beta e^{\beta \phi^0}\|_{L^1_\xi H^{s-1}_x} ) \langle \|F_+^0 \|_{\mathfrak D}\rangle \\
&\lesssim ( \|G\|_{\mathfrak E'}  +  \|F_-^\eps - \mu_\beta e^{\beta \phi^0}\|_{\mathfrak E'} ) \langle \|F_+^0 \|_{\mathfrak D}\rangle
\end{align}
In the above, we used the fact that $m_0+ \frac{3}{2} \leq m_1 $, and 
\begin{align}
-\Delta_x(\phi^\eps - \phi^0) = \int_{\mathbf R^3} G dv - \int_{\mathbf R^3} F_-^\eps - \mu_\beta e^{\beta \phi^0}d\xi .
\end{align}
As for \eqref{eq:GT5}, we split the term as follows:
\begin{align}
\eqref{eq:GT5} &=  \langle   \langle v\rangle^{m_0} \langle \nabla_x\rangle^s Q_{-+}^\eps (F_-^\eps - \mu_\beta e^{\beta \phi^0}, F_+^\eps),G^{(m_0,s)}\rangle_{L^2_{x,v}} \label{eq:GT5-1}\\
&\quad + \langle   \langle v\rangle^{m_0} \langle \nabla_x\rangle^s Q_{-+}^\eps (\mu_\beta e^{\beta \phi^0}, F_+^\eps),G^{(m_0,s)}\rangle_{L^2_{x,v}} \label{eq:GT5-2}
\end{align}
Now,
\begin{align}
|\eqref{eq:GT5-1}|& \lesssim \eps \| \Delta_\xi ^{-1} (F_-^\eps - \mu_\beta e^{\beta \phi^0})\|_{L^\infty_\xi  H^s_x} \|  (F_+^\eps)^{(m_0 + \frac{3}{2},s)}\|_{L^2_{x}H^1_v}\| G^{(m_0 -\frac{3}{2},s)}\|_{L^2_{x}H^1_v}   \\
&\quad +  \| \nabla_\xi  \Delta_\xi ^{-1} (F_-^\eps - \mu_\beta e^{\beta \phi^0})\|_{L^\infty_\xi  H^s_x} \|  (F_+^\eps)^{(m_0+\frac{3}{2},s)}\|_{L^2_{x,v}}\| G^{(m_0 -\frac{3}{2},s)}\|_{L^2_{x}H^1_v} \\
&\lesssim \eps\langle \|  F_+^\eps\|_{\mathfrak D}\rangle (\|G\|_{\mathfrak E'} + \|G^{(m_0,s)}\|_{L^2_x(\dot{\mathcal H}_\sigma)_v} ) \\
&\quad+   \|F_-^\eps -\mu_\beta e^{\beta \phi^0}\|_{\mathfrak E'}^{\frac{1}{4}}  \|F_-^\eps -\mu_\beta e^{\beta \phi^0}\|_{\mathfrak D'} ^{\frac{3}{4}}(\|G\|_{\mathfrak E'} +  \|G^{(m_0,s)}\|_{L^2_x(\dot{\mathcal H}_\sigma)_v})
\end{align}
In the above, we use $m_0 +3 \leq m_1$. 
Next, in order to bound \eqref{eq:GT5-2}, we compute directly from \eqref{eq:Q_-+} that
\begin{align}
&Q_{-+}^\eps(\mu_\beta e^{\beta \phi^0}, F_+^\eps) \\
\quad &= e^{\beta \phi^0} \nabla_v \cdot \int_{\mathbf R^3} \Phi( \xi - \eps v) \{\eps \mu_\beta (\xi) \nabla_v F_+^\eps( v) + \beta \xi \mu_\beta (\xi) F_+^\eps (v)\} d\xi \\
\quad &= \eps e^{\beta \phi^0} \nabla_v \cdot \{ \Phi*\mu_\beta (\eps v) (\beta v + \nabla_v) F_+^\eps  (v)\}.
\end{align}
In the above, we use $\Phi(\xi - \eps v)\xi = \eps \Phi(\xi - \eps v)v$.
Now, $\|\Phi_{ij}*\mu_\beta(v)\|_{L^\infty}\lesssim 1$, so 
\begin{align}
|\eqref{eq:GT5-2}| &\lesssim \eps ( \|( F_+^\eps)^{(m_0+\frac{5}{2},s)}\|_{L^2_{x,v}} \|G^{(m_0-\frac{3}{2},s)}\|_{L^2_xH^1_v} +  \| (F_+^\eps)^{(m_0+\frac{3}{2},s)}\|_{L^2_{x}H^1_v} \|G^{(m_0-\frac{3}{2},s)}\|_{L^2_xH^1_v}) \\
&\lesssim \eps \langle \| F_+^\eps\|_{\mathfrak D} \rangle (\|G\|_{\mathfrak E'} +\|G^{(m_0,s)}\|_{L^2_x(\mathcal H_\sigma)_v}) .
\end{align}
Thus,
\begin{align}
|\eqref{eq:GT5}|&\lesssim  (\eps \langle \| F_+^\eps \|_{\mathfrak D} \rangle +   \|F_-^\eps -\mu_\beta e^{\beta \phi^0}\|_{\mathfrak E'}^\frac{1}{4} \|F_-^\eps -\mu_\beta e^{\beta \phi^0}\|_{ \mathfrak D'}^{\frac{3}{4}})\\
&\quad \cdot(\|G\|_{\mathfrak E'} + \|G^{(m_0,s)}\|_{L^2_x(\dot{\mathcal H}_\sigma)_v}).
\end{align}
Combining the bounds on \eqref{eq:GT1} through \eqref{eq:GT5}, we have
\begin{align}
&\frac{d}{dt} \|G^{(m_0,s)}\|_{L^2_{x,v}}^2 +  \frac{1}{C} \|G^{(m_0,s)}\|_{L^2_{x}(\dot{\mathcal H}_\sigma)_v}^2 \\
&\quad \leq C\{\langle \|(F_+^\eps,F_+^0)\|_{\mathfrak D}\rangle^2 ( \eps^2 +  \|G\|_{\mathfrak E'}^2) + \|G^{(m_1,0)}\|_{L^2_{x}(\dot{\mathcal H}_\sigma)_v}^2  \\
&\quad \quad + \| F_-^\eps -\mu_\beta e^{\beta \phi^0}\|_{ \mathfrak E'}^\frac{1}{2}\| F_-^\eps -\mu_\beta e^{\beta \phi^0}\|_{ \mathfrak D'}^\frac{3}{2}\}
\end{align}
Following a similar argument, we can show the following upper bound on $G^{(m_1,0)}$:
\begin{align}
&\frac{d}{dt} \|G^{(m_1,0)}\|_{L^2_{x,v}}^2 +  \frac{1}{C} \|G^{(m_1,0)}\|_{L^2_{x}(\dot{\mathcal H}_\sigma)_v}^2 \\
&\quad \leq C(\langle \|(F_+^\eps,F_+^0)\|_{\mathfrak D}\rangle^2 ( \eps^2 +  \|G\|_{\mathfrak E'}^2) + \| F_-^\eps -\mu_\beta e^{\beta \phi^0}\|_{ \mathfrak E'}^\frac{1}{2}\| F_-^\eps -\mu_\beta e^{\beta \phi^0}\|_{ \mathfrak D'}^\frac{3}{2}).
\end{align}
Now, choosing $0<\kappa_1 < \kappa_2 $ appropriately, we have that 
\begin{align}
\mathscr Y_+^\eps = \kappa_2\|G^{(m_1,0)}(t)\|_{L^2_{x,v}}^2 + \kappa_1 \|G^{(m_0,s)}(t)\|_{L^2_{x,v}}^2
\end{align} 
saqtisfies
\begin{align}
&\frac{d}{dt} \mathscr Y_+^\eps + \|G\|_{\mathfrak D'}^2 \\
&\quad \lesssim \langle \|(F_+^\eps,F_+^0)\|_{\mathfrak D}\rangle^2 ( \eps^2 +  \|G\|_{\mathfrak E'}^2 ) + \| F_-^\eps -\mu_\beta e^{\beta \phi^0}\|_{ \mathfrak E'}^\frac{1}{2}\| F_-^\eps -\mu_\beta e^{\beta \phi^0}\|_{ \mathfrak D'}^\frac{3}{2}.
\end{align}
From this, we deduce \eqref{eq:Gbd}.
\end{proof}

%
%
%
%
%
%
%

 \subsection{The intermediary quantities}\label{sec:intermediary}

In preparation for Section \ref{sec:derivation}, we introduce the what we call the intermediary potential $\psi^\eps$ and intermediary inverse temperature $\gamma^\eps$. Fixing an initial inverse temperature $\beta_{in}$,  the pair $(\gamma^\eps,\psi^\eps)$ solve the Poincar\'e-Poisson system for each $\eps > 0$:
\begin{equation}\label{eq:PP}
\begin{aligned}
&\frac{d}{dt} \left\{\frac{3}{2\gamma^\eps} + \iint_{\mathbf T^3 \times \mathbf R^3} \frac{|v|^2}{2} F_+^\eps dx dv + \frac{1}{8\pi}  \int_{\mathbf T^3} |\nabla_x \psi^\eps|^2 dx \right\} = 0,\\
&\gamma^\eps(0) = \beta_{in}\\
&-\Delta_x   \psi^\eps= 4\pi (n_+^\eps - e^{\gamma^\eps \psi^\eps}) 
\end{aligned}
\end{equation}
We call $\psi^\eps$ the intermediary potential because like $\phi^0$, it solves a Poincar\'e-Poisson system; however, unlike $\phi^0$, we use $F_+^\eps$ instead of $F_+^0$. Thus $\psi$ serves as a better approximation to $\phi^\eps$ than $\phi^0$. Similarly, $\frac{3}{2\gamma^\eps}$ serves as a better approximation of $\iint_{\mathbf T^3\times\mathbf R^3} \frac{|\xi|^2}{2} F_-^\eps dx dv$ than $\frac{3}{2\beta}$. The  lemma below gives existence and uniqueness for the pair $(\gamma^\eps,\psi^\eps)$, for a given $F_+^\eps$, along with bounds that will be used in the coming section. 

\begin{lemma}\label{lem:PP}
Let $M,T$, $\eta \in (0,1]$ and $\beta_{in} \in (e^{-M},e^M)$.
Assume also that $(F_+^\eps,F_-^\eps)$ is a weak solution to \eqref{eq:VPLRescaled} with $ 0\leq F_{\pm}^\eps \in C([0,T];\mathfrak E)\cap L^2([0,T];\mathfrak D)$ satisfying
\begin{align}
\sup_{\pm \in \{-,+\}, t \in [0,T]} \|\frac{1}{n_{\pm}^\eps(t)}\|_{L^\infty_x} +  \|F_\pm^\eps(t)\|_{\mathfrak E} \leq M.
\end{align}
Then there exists $0<T^*$  such that there exists a unique solution $(\gamma^\eps,\psi^\eps) \in C([0,T]\cap [0,T^*);\mathbf R_+\times H^{s+2})$ to \eqref{eq:PP}. Moreover, if $T^* \leq T$, then $\beta(t) \uparrow +\infty$ as $t \uparrow T^*$.

Moreover, assume that for some $T' \in(0, T]$, that
\begin{align}
\sup_{t \in [0,T']}|\ln(\frac{\gamma^\eps(t)}{\beta_{in}})| &\leq \eta,\label{eq:smallGamVar}.
\end{align}
Note that $T< T^*$ necessarily. 
Then, we have the estimates
\begin{align}
\sup_{t \in [0,T']}\|\psi^\eps \|_{H^{s+2}_x} & \lesssim_{M} 1,\label{eq:BigPhiBd}\\
\label{eq:dtBigPhiandBetaBd}
\sup_{t \in [0,T']}\|\partial_t \psi^\eps \|_{H^{s+1}_x} + |\dot \gamma^\eps| &\lesssim_{M} \langle \|\langle \xi\rangle^5 F_-^\eps\|_{L^2_x(\dot{\mathcal H}_\sigma)_\xi }\rangle,
\end{align}
Next, assume $(F_+^0,\beta,\phi^0)$ is a weak solution to \eqref{eq:VPLIon} and satisfies the hypotheses of Theorem \ref{thm:derivation} (with $T_0 = T$), with the given $\beta_{in}$. Then, 
we also control the difference of $(\gamma^\eps,\psi^\eps)$ and $(\beta,\phi^0)$:
\begin{align}
\label{eq:gammaBetaDiff}
\sup_{t \in [0,T']} \{|\gamma^\eps(t) - \beta(t)|  + \|\psi^\eps(t)-\phi^0(t)\|_{H^{s+1}_x}\} \lesssim \sup_{t\in [0,T]} \|F_+^\eps(t) - F_+^0(t)\|_{\mathfrak E'}.
\end{align}

\end{lemma}

\begin{proof} We break the proof up into 3 parts: first, the bounds \eqref{eq:BigPhiBd}, and \eqref{eq:dtBigPhiandBetaBd}; second, the bound \eqref{eq:gammaBetaDiff}; and third, we sketch how to construct the solutions. Once again, all bounds involved may depend on $M$. \\

\noindent \textbf{Step 1:} We now show \eqref{eq:BigPhiBd}, and \eqref{eq:dtBigPhiandBetaBd}. For this step, it is convenient to drop the superscripts of $\eps$, i.e. $F_+^\eps  = F_+$, $F_-^\eps = F_-$, etc. 
We now prove \eqref{eq:BigPhiBd}, from the third line of \eqref{eq:PP}, and the maximum principle, we have the estimate
\begin{align}
\|\psi\|_{L^\infty_x} \lesssim \frac{1}{\gamma}\|\ln(n_+)\|_{L^\infty_x} \lesssim 1.
\end{align}
Next, applying $|\nabla_x|^s$ to both sides of the third line of  \eqref{eq:PP}, 
\begin{align}
\frac{1}{4\pi} |\nabla_x|^{s'+ 2}\psi = |\nabla_x|^{s'}\{ n_+ - e^{\gamma \psi}\}.
\end{align}
Using  Theorem 5.2.6 in   \cite{metivier2008differential}, 
\begin{align}
\|e^{\gamma \psi}\|_{H^{s}_x} \leq C_{\|\psi\|_{L^\infty_x}} (\|\psi\|_{H^{s}_x} +1) \lesssim  \|\psi\|_{H^{s}_x} + 1.
\end{align}
Hence,
\begin{align}
\|\psi\|_{H^{s + 2}_x} \lesssim 1+ \|n_+\|_{H^s_x} + \|\psi\|_{H_x^{s}} \lesssim 1 + \|\psi\|_{H^{s+2}_x}^{\frac{s}{s+2}} \|\psi\|_{L^2_x}^{\frac{2}{s+2}}.
\end{align}
Now, using  $\|\psi\|_{L^2_x} \lesssim \|\psi\|_{L^\infty_x} \lesssim 1$ and Young's inequality, we deduce $\|\psi\|_{H^{s + 2}_x} \lesssim 1$. 

Next, we have
\begin{align}\label{eq:dtBigPhi}
(\gamma e^{\gamma \psi} -\frac{1}{4\pi} \Delta_x) \partial_t \psi = \partial_t n_+ - \dot \gamma e^{\gamma \psi} \psi.
\end{align}
The maximum principle implies 
\begin{align}
e^{-C\|\psi\|_{L^\infty_x}}\|\partial_t \psi\|_{L^\infty_x} \leq C \|\partial_t n_+\|_{L^\infty_x} + |\dot \gamma| e^{C\| \psi\|_{L^\infty_x}} \|\psi\|_{L^\infty_x},
\end{align}
From the continuity equation, we know $\|\partial_t n_+ \|_{L^\infty_x} \lesssim \|\partial_t n_+\|_{H^{s-1}_x} \lesssim 1$; in combination with the estimates on $\psi$, this implies
\begin{align}
\|\partial_t \psi\|_{L^\infty_x} \lesssim 1 + |\dot \gamma|.
\end{align}
Similarly as was done to get the bound on  $\|\psi\|_{H^{s+2}_x}$, we deduce
\begin{align}
\|\partial_t \psi\|_{H^{s+1}_x} \lesssim 1 + |\dot \gamma|.
\end{align}
We now estimate $\dot \gamma$: first, 
\begin{align}
\frac{3\dot \gamma}{2\gamma^2} =\frac{1}{2} \iint_{\mathbf T^3\times \mathbf R^3 } |v|^2\partial_t F_+ dx dv  + \frac{1}{4\pi} \int_{\mathbf T^3} \nabla_x \psi\cdot \partial_t \nabla_x \psi dx.
\end{align}
We first analyze the last term. From \eqref{eq:dtBigPhi},
\begin{align}
\frac{1}{4\pi} \int_{\mathbf T^3} \nabla_x \psi\cdot \partial_t \nabla_x \psi dx &= -\frac{1}{4\pi} \langle \Delta_x \psi, \partial_t \psi\rangle_{L^2_x}\\
&= - \frac{1}{4\pi} \langle \Delta_x \psi,(\gamma e^{\gamma \psi} -\frac{1}{4\pi} \Delta_x)^{-1}\partial_t n_+  \rangle_{L^2_x} \\
&\quad+  \frac{\dot \gamma}{4\pi} \langle \Delta_x \psi,(\gamma e^{\gamma \psi} -\frac{1}{4\pi} \Delta_x)^{-1}(e^{\gamma \psi} \psi)  \rangle_{L^2_x} 
\end{align}
Looking more closely at the latter term, we observe
\begin{align}
& \frac{1}{4\pi} \langle \Delta_x \psi,(\gamma e^{\gamma \psi} -\frac{1}{4\pi} \Delta_x)^{-1}(e^{\gamma \psi} \psi)  \rangle_{L^2_x}  \\
 &=  - \langle  (\gamma e^{\gamma \psi} - \frac{1}{4\pi} \Delta_x)\psi,(\gamma e^{\gamma \psi} -\frac{1}{4\pi} \Delta_x)^{-1}(e^{\gamma \psi} \psi)\rangle_{L^2_x} \\
 &\quad  + \gamma \langle e^{\gamma \psi}\psi, (\gamma e^{\gamma \psi} -\frac{1}{4\pi} \Delta_x)^{-1}(e^{\gamma \psi} \psi)\rangle_{L^2_x}\\
 &= -\langle  \psi,e^{\gamma \psi}\psi\rangle_{L^2_x} -  \gamma \langle e^{\gamma \psi}\psi, (\gamma e^{\gamma \psi} -\frac{1}{4\pi} \Delta_x)^{-1}(e^{\gamma \psi} \psi)\rangle_{L^2_x} \leq 0.
\end{align}
Thus, we have that
\begin{align}\label{eq:gammaDot}
\dot \gamma &=\frac{\frac{1}{2} \iint_{\mathbf T^3\times \mathbf R^3 } |v|^2\partial_t F_+ dx dv- \frac{1}{4\pi} \langle \Delta_x \psi,(\gamma e^{\gamma \psi} -\frac{1}{4\pi} \Delta_x)^{-1}\partial_t n_+  \rangle_{L^2_x}  }{\frac{3}{2\gamma^2} -\frac{1}{4\pi} \langle \Delta_x \psi,(\gamma e^{\gamma \psi} -\frac{1}{4\pi} \Delta_x)^{-1}(e^{\gamma \psi} \psi)  \rangle_{L^2_x}},
\end{align}
with the denominator being strictly bounded from below by $\frac{3}{2\gamma^2 } $. On the other hand, by the continuity equation $\partial_t n_+ +\int_{\mathbf R^3} v\cdot \nabla_x F_+ dx dv =0$, we have 
\begin{align}
|\langle \Delta_x \psi,(\gamma e^{\gamma \psi} -\frac{1}{4\pi} \Delta_x)^{-1}\partial_t n_+  \rangle_{L^2_x} |\lesssim 1.
\end{align}
Thus,
\begin{align}
|\dot \gamma| \lesssim 1+ |\iint_{\mathbf T^3\times \mathbf R^3 } |v|^2\partial_t F_+ dx dv|.
\end{align}
Now,
\begin{align}
\iint_{\mathbf T^3\times \mathbf R^3 } |v|^2\partial_t F_+ dx dv = 2\iint_{\mathbf T^3\times \mathbf R^3 }v\cdot E F_+ dx dv +\iint_{\mathbf T^3 \times \mathbf R^3} |v|^2 Q^{\eps}_{-+}(F_-,F_+)dx dv 
\end{align}
The  first term has the bound
\begin{align}
2\iint_{\mathbf T^3\times \mathbf R^3 }v\cdot E F_+ dx dv &\lesssim \|E\|_{L^2_x} \| |v|F_+\|_{L^2_x L^1_v} \\
& \lesssim (\|\langle \xi\rangle^2 F_-\|_{L^2_{x,v}} +  \|\langle v\rangle^2 F_+\|_{L^2_{x,v}}) \|\langle v\rangle^3F_+\|_{L^2_{x,v}}\\
&\lesssim 1.
\end{align}
Next,
\begin{align}
&\iint_{\mathbf T^3 \times \mathbf R^3} |v|^2 Q^{\eps}_{-+}(F_-,F_+)dx dv  \\
&\quad =-2\iint_{\mathbf T^3\times \mathbf R^3}v_j\{\eps  \Phi_{ij}*_\xi  F_-|_{\xi  = \eps v} \partial_{v_i} F_+ - \partial_{\xi _i}  \Phi_{ij}*_\xi  F_-|_{\xi  = \eps v} F_+\}dx dv \\
&\quad = 2\iint_{\mathbf T^3\times \mathbf R^3}\{\eps \ \mathrm {tr}( \Phi*_\xi  F_-|_{\xi  = \eps v}) +(1 + \eps^2)v_j \partial_{\xi _i}  \Phi_{ij}*_\xi  F_-|_{\xi  = \eps v}\}F_+dx dv
\end{align}
Therefore,\begin{align}
&|\iint_{\mathbf T^3 \times \mathbf R^3} |v|^2 Q^{\eps}_{-+}(F_-,F_+)dx dv|\\
& \lesssim (\|(-\Delta_\xi )^{-1} F_-\|_{L^2_xL^\infty_\xi } +  \|\nabla_\xi (-\Delta_\xi )^{-1} F_-\|_{L^2_xL^\infty_\xi })\|v F_+\|_{L^2_{x}L^1_v} \\
&\lesssim \|\langle \xi\rangle^2  \langle \nabla_v \rangle F_- \|_{L^2_{x,v}} \|\langle v\rangle^3 F_+\|_{L^2_{x,v}} \\
&\lesssim\langle \|\langle \xi\rangle^5 F_-\|_{L^2_x(\dot{\mathcal H}_\sigma)_\xi }\rangle \|\langle v\rangle^3 F_+\|_{L^2_{x,v}}\\
&\lesssim   \langle \|\langle \xi\rangle^5 F_-\|_{L^2_x(\dot{\mathcal H}_\sigma)_\xi }\rangle.
\end{align}
From this, we conclude \eqref{eq:dtBigPhiandBetaBd}.
\\

\noindent \textbf{Part 2:} We now control the difference \eqref{eq:gammaBetaDiff}.
 reintroduce the superscripts of $\eps$, so as to distinguish $F_+^\eps$ from $F_+^0$, etc, although we still ignore the dependence of constants on $M$. We first bound the difference $\psi^\eps_{in} - \phi^0_{in}$. 
It suffices to show this in the case
\begin{align}\label{eq:ionSmallDiff}
\sup_{t\in [0,T]} \|F_+^\eps(t) - F_+^0(t)\|_{\mathfrak E'} \leq \alpha,
\end{align}
where $\alpha = \alpha(M)$ is sufficiently small. If this bound does not hold, then we use the bound
\begin{align}
\sup_{t \in [0,T]} |\gamma^\eps(t) - \beta(t)|  + \|\psi^\eps(t)-\phi^0(t)\|_{H^{s+1}_x}&\leq C_M\\
& \leq \frac{C_M}{\alpha}\sup_{t\in [0,T]} \|F_+^\eps(t) - F_+^0(t)\|_{\mathfrak E'}.
\end{align}

We first control the difference at time zero: 
\begin{align}\label{eq:initialEq}
-\frac{1}{4\pi} \Delta_x(\psi^\eps_{in} - \phi^0_{in})& = n_{+,in}^\eps - n_{+,in}^0  + e^{\beta_{in} \phi^0_{in}} - e^{\beta_{in} \psi^\eps_{in}}. 
\end{align}
Using the maximum principle, we have 
\begin{align}
\|e^{\beta_{in}\psi^\eps_{in}} - e^{\beta_{in}\phi^0_{in}}\|_{L^\infty_x} \lesssim  \|n_{+,in}^\eps - n_{+,in}^0 \|_{L^\infty_x} \lesssim \sup_{t\in [0,T]} \|F_+^\eps(t) - F_+^0(t)\|_{\mathfrak E'}.
\end{align}
By the mean-value theorem, this implies $\|\psi^\eps_{in} - \phi^0_{in}\|_{L^\infty_x} \lesssim \sup_{t\in [0,T]} \|F_+^\eps(t) - F_+^0(t)\|_{\mathfrak E'}$. 
In fact, by writing 
\begin{align}
 e^{\beta_{in} \phi^0} - e^{\beta_{in} \psi^\eps_{in}} = \beta_{in} \int_0^1 e^{\lambda \phi^0 +  (1- \lambda)\psi^\eps} (\phi^0 - \psi^\eps)d\lambda,
\end{align}
it is easy to show that after
applying $\langle \nabla_x\rangle^s$ to \eqref{eq:initialEq}, we get
\begin{align}
\|\psi^\eps_{in} - \phi^0_{in}\|_{H^{s+2}_x} \lesssim \|n_{+,in}^\eps - n_{+,in}^0 \|_{H^s_x} + \|\psi^\eps_{in} - \phi^0_{in}\|_{H^{s}_x}.
\end{align}
Thus, by interpolation, we get
\begin{align}\label{eq:initialEstimate}
\|\psi^\eps_{in} - \phi^0_{in}\|_{H^{s+2}_x} \lesssim \sup_{t\in [0,T]} \|F_+^\eps(t) - F_+^0(t)\|_{\mathfrak E'}. 
\end{align}

Now, we show how to bound the differences for positive times.
Take $\kappa \in (0,1]$ to be a constant to be determined later, and let $\tilde T$ be the longest time such that $\tilde T \leq T$ and 
\begin{align}
\sup_{t \in [0,\tilde T]} | \beta - \gamma^\eps| + \|\psi^\eps - \phi^0\|_{H^{s+2}_x} \leq \kappa.
\end{align}
The equation we have for the difference $\psi^\eps - \phi^0$ is
\begin{align}
-\Delta_x(\psi^\eps - \phi^0) &= 4\pi(n_+^\eps - n_+^0 + e^{\beta \phi^0} - e^{\gamma^\eps \psi^\eps}).
\end{align}
Rearranging, we have
\begin{align}
&(\gamma^\eps e^{\gamma^\eps \psi^\eps}  -\frac{1}{4\pi} \Delta_x)(\psi^\eps - \phi^0)\\
 &\quad = n_+^\eps - n_+^0 +  e^{\gamma^\eps \psi^\eps} (e^{\beta \phi^0 - \gamma^\eps \psi^\eps} - 1 - (\beta \phi^0 -  \gamma \psi^\eps)  +  (\beta - \gamma^\eps) (\phi^0 - \psi^\eps))\\
 &\quad \quad +(\beta - \gamma^\eps) e^{\gamma^\eps \psi^\eps} \psi^\eps.
\end{align}
Then, on the interval $t \in [0,\tilde T_{\eps}]$, we have
\begin{align}\label{eq:psiDifftoGammaDiff}
\begin{aligned}
&\|\psi^\eps - \phi^0 -(\beta - \gamma^\eps) (\gamma^\eps e^{\gamma^\eps \psi^\eps}  -\frac{1}{4\pi} \Delta_x)^{-1} (e^{\gamma^\eps \psi^\eps}\psi^\eps)\|_{H^{s+2}_x}\\
&\quad  \lesssim  \|F^\eps_+ - F^0_+\|_{\mathfrak E'}  + \kappa(|\gamma^\eps - \beta| +\|\psi^\eps - \phi^0\|_{H^{s+2}_x} ) . 
\end{aligned}
\end{align}
In particular, taking $\kappa$ small enough,
\begin{align}
\|\psi^\eps - \phi^0\|_{H^{s+2}_x} \lesssim \sup_{t\in [0,T]} \|F_+^\eps(t) - F_+^0(t)\|_{\mathfrak E'}+ |\gamma^\eps - \beta|.
\end{align}
Now,
\begin{align}
&\Big|\frac{1}{8\pi} \int_{\mathbf T^3} \nabla_x (\psi^\eps - \phi^0) \cdot \nabla_x (\psi^\eps + \phi^0) dx \\
&\quad\quad + \frac{1}{4\pi} \int_{\mathbf T^3}(\beta - \gamma^\eps) (\gamma^\eps e^{\gamma^\eps \psi^\eps}  -\frac{1}{4\pi} \Delta_x)^{-1}(e^{\gamma^\eps \psi^\eps}\psi^\eps) \Delta_x \psi^\eps dx\Big| \\
&\lesssim  \left|\int_{\mathbf T^3}\{(\psi^\eps - \phi^0) - (\beta - \gamma^\eps) (\gamma^\eps e^{\gamma^\eps \psi^\eps}  -\frac{1}{4\pi} \Delta_x)^{-1} (e^{\gamma^\eps \psi^\eps}\psi^\eps)\} \Delta_x \psi^\eps dx\right|  \\
&\quad + \int_{\mathbf T^3} |\nabla_x (\psi^\eps - \phi^0)|^2dx \\
& \lesssim \sup_{t\in [0,T]} \|F_+^\eps(t) - F_+^0(t)\|_{\mathfrak E'}+ |\gamma^\eps - \beta|
\end{align}
Next,
\begin{align}
&\frac{d}{dt} \left(\frac{3}{2\beta} -\frac{3}{2\gamma^\eps }\right) \\
&\quad= \frac{d}{dt} \left(\iint_{\mathbf T^3\times \mathbf R^3}\frac{|v|^2}{2} (F_+^\eps - F_+^0) dx dv  + \frac{1}{8\pi} \int_{\mathbf T^3} \nabla_x (\psi^\eps - \phi^0) \cdot \nabla_x (\psi^\eps + \phi^0) dx\right)
\end{align}
Then, using $\beta_{in} = \gamma(t=0)^\eps$ and integrating the above, we have
\begin{align}
\frac{3(\gamma^\eps - \beta)}{2\beta \gamma^\eps}&= \iint_{\mathbf T^3\times \mathbf R^3}\frac{|v|^2}{2} (F_+^\eps - F_+^0) dx dv  \\
&\quad+ \frac{1}{8\pi} \int_{\mathbf T^3} \nabla_x (\psi^\eps - \phi^0) \cdot \nabla_x (\psi^\eps + \phi^0) dx \\
& \quad - \iint_{\mathbf T^3\times \mathbf R^3}\frac{|v|^2}{2} (F_{+,in}^\eps - F_{+,in}^0) dx dv\\
&\quad   - \frac{1}{8\pi} \int_{\mathbf T^3} \nabla_x (\psi_{in}^\eps - \phi_{in}^0) \cdot \nabla_x (\psi_{in}^\eps + \phi_{in}^0) dx.
\end{align}
Using the previous bounds, plus \eqref{eq:initialEstimate}, we have 
\begin{align}
&|\gamma^\eps - \beta|  \left| \frac{3}{2\beta \gamma^\eps} - \frac{1}{4\pi} \int_{\mathbf T^3} (\gamma^\eps e^{\gamma^\eps \psi^\eps}  -\frac{1}{4\pi} \Delta_x)^{-1} (e^{\gamma^\eps \psi^\eps}\psi^\eps )\Delta_x \psi^\eps dx\right| \\
&\quad \lesssim |\gamma^\eps - \beta|^2 + \sup_{t\in [0,T]} \|F_+^\eps(t) - F_+^0(t)\|_{\mathfrak E'}.
\end{align}
By a similar argument as was used to bound $\dot \gamma$, we have that 
\begin{equation}
\int_{\mathbf T^3} (\gamma^\eps e^{\gamma^\eps \psi^\eps}  -\frac{1}{4\pi} \Delta_x)^{-1} (e^{\gamma^\eps \psi^\eps}\psi^\eps )\Delta_x \psi^\eps dx \leq 0.
\end{equation}
We conclude that for all $t \in [0,\tilde T]$, 
\begin{align}
|\gamma^\eps(t) - \beta(t)| \lesssim \kappa |\gamma^\eps(t) - \beta(t)|  +  \sup_{t\in [0,T]} \|F_+^\eps(t) - F_+^0(t)\|_{\mathfrak E'}  
\end{align}
Taking $\kappa$ sufficiently small, we have $|\gamma^\eps(t) - \beta(t)| \leq C_0 \alpha$ for some $C_0$. So, by taking $\alpha = \frac{\kappa}{2C_0}$, we have $\tilde T = T$.\\

\noindent\textbf{Step 3:} Finally, we sketch how to construct solutions  $(\gamma^\eps,\psi^\eps)$ satisfying \eqref{eq:PP}.
We refer the reader (ii) of Theorem 1.4 in \cite{bardos_maxwellboltzmann_2018} and its proof. Using standard elliptic theory, there exists  a unique $\psi_{in}^\eps \in H^{s+2}$ solving the system
\begin{align}
-\Delta_x \psi_{in} ^\eps= 4\pi (n_{+,in} ^\eps- e^{\beta_{in} \psi_{in}^\eps}).
\end{align}
Now, for $t > 0$, define
\begin{align}
E(t) := \frac{3}{2\beta_{in}} + \iint_{\mathbf T^3 \times \mathbf R^3} \frac{|v|^2}{2} ( F_{+,in}^\eps - F_+^\eps(t) )dx dv + \frac{1}{8\pi}  \int_{\mathbf T^3} |\nabla_x \psi_{in}^\eps|^2 dx.
\end{align}
Then, for all $t >0$ such that $E(t)>0$, there exists $(\gamma(t),\psi(t)) \in \mathbf R_+ \times H^{s+2}$ solving third line of \eqref{eq:PP}, and the first line integrated on $[0,t]$.
Since $t \mapsto \iint_{\mathbf T^3\times \mathbf R^3} \frac{|v|^2}{2} F_+^\eps(t)dxdv$ is continuous, and the above condition holds at $t = 0$, we take $T^* >0$ to be the first time less that $T$ such that the above condition does not hold, or take $T^* = T$ if no such time exists. Then there exists a unique $(\gamma,\psi)$ satisfying \eqref{eq:PP}, with $\ln(\gamma)  \in L^\infty_{loc}([0,T^*))$ and $\psi \in L^\infty_{loc}([0,T^*);H^{s+2})$.  Moreover,  $\frac{3}{2\gamma^\eps(t)}   \leq E(t)$, so $\gamma^\eps(t) \to \infty$ as $t\to T^*$ whenever $T^* \leq T$.

Next, we note that $\gamma$ is continuous in time. This follows by the identity \eqref{eq:gammaDot}, and the fact that 
\begin{align}
\|\partial_t n_+^\eps(t)\|_{L^2_x}, \iint_{\mathbf T^3\times \mathbf R^3} \frac{|v|^2}{2}\partial_t  F_+^\eps(t,x,v)dxdv \in L^1([0,T]),
\end{align}
also shown above. It is straightforward to show $\psi \in C([0,T^*);H^{s+2})$ from this, and \eqref{eq:dtBigPhi}.

\end{proof}

\section{Estimates on the electron distribution}
\label{sec:derivation}

\subsection{Stability of the Maxwellian}

The main result of this section is Proposition \ref{prop:derivationRefined}, which, roughly speaking, shows that the Maxwellian $\mu_{\gamma^\eps} e^{\gamma^\eps \psi^\eps}$ is stable; that is, if $F_{-,in}^\eps$ is close to $\mu_{\gamma^\eps} e^{\gamma^\eps \psi^\eps}$, then it will remain close. We also utilize the stretched exponential decay to acquire some form of asymptotic stability.

To state the result,
let $\eta>0$, and recall $q_\alpha  = e^{\alpha \eta} \beta_{in}$ for each $\alpha\in \{1,2,3\}$. Then, similarly to $\mathscr E_{-,\alpha}$ and $\mathscr D_{-,\alpha}$, we define for each
 $\alpha\in \{1,2\}$,
 \begin{equation}
 \begin{aligned}
 \widetilde {\mathscr E}_{-,\alpha}^\eps &:= \|e^{q_\alpha |\xi|^2/4} \langle \nabla_x \rangle^s\{F_{-}^\eps  - \mu_{\gamma^\eps \psi^\eps} e^{\gamma^\eps \psi^\eps}\}\|_{L^2_{x,\xi }}^2 \\
 &\quad + \|e^{q_{\alpha +1}|\xi|^2/4} \{F_{-}^\eps  - \mu_{\gamma^\eps} e^{\gamma^\eps \psi^\eps}\}\|_{L^2_{x,\xi }}^2 \\
 \end{aligned}
 \end{equation}
and
 \begin{equation}
 \begin{aligned}
 \widetilde {  \mathscr D}_{-,\alpha}^\eps& := \|e^{q_\alpha |\xi|^2/4} \langle \nabla_x \rangle^s\{F_{-}^\eps  - \mu_{\gamma^\eps \psi} e^{\gamma^\eps \psi^\eps}\}\|_{L^2_x(\mathcal H_\sigma\cap \mathcal H_{\sigma;\eps}^-)_\xi }^2\\
   &\quad  + \|e^{q_{\alpha +1}|\xi|^2/4} \{F_{-}^\eps  - \mu_{\gamma^\eps} e^{\gamma^\eps \psi^\eps}\}\|_{L^2_x(\mathcal H_\sigma\cap \mathcal H_{\sigma;\eps}^-)_\xi }^2.
 \end{aligned}
 \end{equation}

\begin{proposition}\label{prop:derivationRefined}  
Fix $M >0$. Then there exists positive constants $ \eta, \varsigma$ and $\overline \eps$ such that the following holds. Assume $0<\eps \leq \overline \eps$. For each such $\eps$,  let $(F_+^\eps,F_-^\eps)$ be solution to the system \eqref{eq:VPLRescaled} with $0\leq F_{\pm}^\eps \in C([0,T];\mathfrak E)\cap L^2([0,T];\mathfrak D)$. Fixing $\beta_{in} \in (e^{-M}, e^M)$, 
let $(\psi^\eps,\gamma^\eps)$ be the corresponding solution to the system \eqref{eq:PP}, assuming that
\begin{align}
\sup_{t\in[0,T]}( \|\frac{1}{n_+^\eps(t)}\|_{L^\infty_x}+ \|F_+^\eps(t)\|_{\mathfrak E} )&\leq M, \label{eq:bootstrap1}\\
\sup_{t\in[0,T]} \sqrt{ \widetilde {\mathscr E}_{-,2}^\eps(t)} & \leq \varsigma,\label{eq:bootstrap2}\\
\sup_{t\in[0,T]} |\ln( \frac{ \gamma^\eps(t)}{\beta_{in}}) |&\leq \eta, \label{eq:bootstrap3}\\
\int_0^T \|\partial_t (n_-^\eps - e^{\gamma^\eps \psi^\eps})\|_{L^2_x} dt& \leq 1.
\label{eq:bootstrap4}
\end{align}
We also impose the following assumption on the initial data:
\begin{align}
\label{eq:energyCondition}
\left|\int_{\mathbf R^3 \times \mathbf T^3} \frac{|\xi|^2}{2} (F_{-,in}^\eps - \mu_{\beta_{in}}e^{\beta_{in} \psi^\eps_{in}}) dxd\xi  + \frac{1}{8\pi} \int_{\mathbf T^3}  |\nabla_x \phi_{in}^\eps|^2 - |\nabla_x \psi^\eps_{in}|^2dx \right| \leq M\eps.
\end{align}
Then, we have
\begin{align}\label{eq:stability1}
\eps \sup_{t \in [0,T]}  \widetilde {\mathscr E}^\eps_{-,2}(t) + \int_0^T  \widetilde {\mathscr D}^\eps_{-,2}(t) dt \leq C_M (\eps  \widetilde {\mathscr E}^\eps_{-,2,in}  + \eps^2)
\end{align}
and for all $t \in[0,T]$, we have 
\begin{align}\label{eq:stability2}
\mathscr E^\eps_{-,1}(t) \leq C_M( e^{-\frac{1}{C_M}(\frac{t}{\eps})^{\frac{2}{3}}}\sup_{t'\in[0,T]} {\mathscr E_{-,2}^\eps(t')}  +\eps^{\frac{5}{3}} t^\frac{1}{3})
\end{align}

\end{proposition}
Throughout the rest of this section, we will not make reference to $(F_+^0,\beta,\phi^0)$ and its derived quantities, and instead only work with $(F_-^\eps,F_+^\eps)$. With this understood, throughout this section, we will write $F_+^\eps = F_+$, $F_-^\eps = F_-$, and so on, with the dependence on $\eps$ being implicit. 
\subsection{Setup}
We introduce  $a = n_- - e^{\gamma  \psi}$ and note that
\begin{align}\label{eq:aphiPhi}
4\pi a = -\Delta_x(\phi - \psi),
\end{align}
 Using \eqref{eq:bootstrap3}, we note
\begin{align}
\gamma(t) \leq q_1 < q_2 < q_3.
\end{align}
Next, we separate
\begin{align}
F_- -  \mu_{\gamma} e^{\gamma  \psi} = \sqrt{ \mu_{\gamma}} f = \sqrt{\mu_{q_\alpha } e^{{q_\alpha  \phi}}}g_\alpha 
\end{align}
where $ \alpha\in \{1,2,3\}$. 

The equation for $f$ reads as follows
\begin{equation}
\label{eq:f}\begin{aligned}
&\eps \mu_\gamma^{-\frac{1}{2}}\partial_t \{\mu_\gamma^{\frac{1}{2}}f\} + \{\xi\cdot \nabla_x +E \cdot (\frac{\xi }{2} - \nabla_\xi )\} f +e^{\gamma \psi} \mathcal L_\gamma f + \mathcal M_{\gamma,F_+}f\\
&\quad \quad  -4\pi \gamma \xi\cdot \nabla_x\Delta_x^{-1} a \mu^{\frac{1}{2}}_\gamma e^{\gamma \psi} \\
&\quad =- \eps \mu_\gamma^{\frac{1}{2}} {\partial_t\{\mu_\gamma e^{\gamma \psi}\}}-e^{\gamma \psi}\mathcal M_{\gamma,F_+} \mu_\gamma^{\frac{1}{2}} + \Gamma_\gamma(f,f)
\end{aligned}
\end{equation}
where
\begin{align}
\mathcal L_\gamma h&= \mu_\gamma^{-\frac{1}{2}} \{Q(\mu_{\gamma}^{\frac{1}{2}}h,\mu_{\gamma}) + Q(\mu_{\gamma},\mu_{\gamma}^{\frac{1}{2}}h)\},\\
\mathcal M_{\gamma,F_+}h & =\mu_{\gamma}^{-\frac{1}{2}}Q_{+-}^\eps(F_+, \mu^{\frac{1}{2}}_{\gamma}h), \\
\Gamma_\gamma (h_1,h_2)&= - \mu^{-\frac{1}{2}}_{\gamma} Q(\mu^{\frac{1}{2}}_{\gamma} h_1,\mu^{\frac{1}{2}}_{\gamma} h_2).
\end{align}
On the other hand, $g_\alpha $, $\alpha\in\{1,2,3\}$ solve the equation

\begin{equation}\label{eq:gjEq}
\begin{aligned}
&\eps \{\partial_t  + q\partial_t \phi\}g_\alpha + \{v\cdot \nabla_x -E \cdot \nabla_\xi\} g_\alpha  +e^{\gamma \psi} \widetilde {\mathcal L}_{\gamma,q} g_\alpha + \mathcal M_{q_\alpha ,F_+}g_\alpha \\
&\quad \quad  -4\pi \gamma \xi\cdot \nabla_x\Delta_x^{-1} a  \frac{\mu_\gamma}{\sqrt{\mu_{q_\alpha }}} e^{\gamma \psi - \frac{q_\alpha }{2}\phi}\\
&\quad =- \eps\frac{\partial_t\{ \mu_{\gamma}e^{\gamma\psi}\}}{\sqrt{\mu_{q_\alpha }e^{q_\alpha \psi}}}-e^{\gamma \psi - \frac{q_\alpha }{2} \phi}\mathcal M_{q_\alpha ,F_+}(\mu^{-\frac{1}{2}}_{q_\alpha } \mu_\gamma) + e^{\frac{q_\alpha }{2} \phi} \Gamma_{q_\alpha }(g_\alpha ,g_\alpha )
\end{aligned}
\end{equation}
where $\mathcal M_{q_\alpha ,F_+}$ and $\Gamma_{q_\alpha }$ are defined the same way as $\mathcal M_{\gamma,F_+}$ and $\Gamma_\gamma$ respectively, with  $\gamma$ substituted with $q_\alpha $. The operator $\widetilde {\mathcal L}_{\gamma,q}$ is defined as follows.
\begin{align}
\widetilde{\mathcal L}_{\gamma,q} \psi&= \mu_q^{-\frac{1}{2}} \{Q(\mu_q^{\frac{1}{2}}\psi,\mu_\gamma) + Q(\mu_\gamma,\mu_q^{\frac{1}{2}}\psi)\}
\end{align}

Now, in addition, we define the operator $\mathcal P_\gamma$ to be the $L^2$-projection (in the $\xi $ variable) to the kernel of $\mathcal L_\gamma$, and we take $\mathcal P_\gamma^\perp = \mathrm{Id}_{L^2(\mathbf R^3)} - \mathcal P_\gamma$.  From \cite{guo_landau_2002}, we have that this kernel is
\begin{align}
N(\mathcal L_\gamma) = \mathrm{span} \{ \mu_\gamma^{\frac{1}{2}}(\xi), w_1\mu_\gamma^{\frac{1}{2}},\xi _2\mu_\gamma^{\frac{1}{2}},\xi _3\mu_\gamma^{\frac{1}{2}}, |\xi|^2\mu_\gamma^{\frac{1}{2}}\}.
\end{align}
In addition to the density of the perturbation $a$ already defined, we define the macroscopic variables $c:\mathbf T^3\to\mathbf R$ and $b :\mathbf T^3\to \mathbf R^3$ as the coefficients of this projection on $f$:
\begin{align}
\mathcal P_\gamma f = a \mu_\gamma^{\frac{1}{2}} + \gamma^{\frac{1}{2}}b\cdot w \mu^{\frac{1}{2}}_\gamma + c\frac{\gamma |\xi|^2 - 3}{\sqrt{6}}\mu_\gamma^{\frac{1}{2}} \label{eq:NLprojection}
\end{align}
We note that the representation above is in terms of  an orthonormal basis for $N(\mathcal L_\gamma)$. Moreover,
\begin{align}
a &= \int_{\mathbf R^3} \{F_-- \mu_\gamma e^{\gamma \psi}\} d\xi   = n_- -e^{\gamma \psi} \label{eq:atoDensity}\\
b &= \int_{\mathbf R^3} \gamma^{\frac{1}{2}} \xi \{F_-- \mu_\gamma e^{\gamma \psi}\}d\xi  = \gamma^{\frac{1}{2}}\int_{\mathbf R^3}  \xi F_- d\xi \\
c &= \int_{\mathbf R^3}\frac{( \gamma |\xi|^2 -3)}{\sqrt{6}}\{F_-- \mu_\gamma e^{\gamma \psi}\}d\xi  = \frac{1}{\sqrt{6}}\left(\gamma \int_{\mathbf R^3} |\xi|^2 F_- d\xi  - 3n_-\right). \label{eq:ctoDens+Energy}
\end{align}
Thus, $(a,b,c)$ form a linear transformation of the physical macroscopic variables of density, current and local kinetic energy of the electrons. Given $r\geq 0$, we shall denote $a^{(r)} =\langle \nabla_x\rangle^r a$, $g_\alpha ^{(r)} =\langle \nabla_x\rangle^r g_\alpha $, and so forth. 

Finally, we close this section by mentioning some commonly occurring estimates that will occur throughout this section.
We note that under the bootstrap assumptions, $\|\phi\|_{H^{s+2}_x}   + \|\psi\|_{H^{s+2}_x} \lesssim_M 1$. Therefore,
\begin{align}
&\|f\|_{L^2_{x,\xi }} \lesssim_M \|g_1\|_{L^2_{x,\xi }} \lesssim_M \|g_2\|_{L^2_{x,\xi }} \lesssim_M \|g_3\|_{L^2_{x,\xi }} \\
&\|f^{(s)}\|_{L^2_{x,\xi }} \lesssim_M \|g_1^{(s)}\|_{L^2_{x,\xi }} \lesssim_M \|g_2^{(s)}\|_{L^2_{x,\xi }} \lesssim_M \|g_3^{(s)}\|_{L^2_{x,\xi }} 
\end{align}
The same inequalities hold over the spaces $L^2_x(\mathcal H_\sigma)_\xi $ and $L^2_x(\mathcal H_{\sigma;\eps}^-)_\xi $.
Moreover,
\begin{align}
\widetilde{\mathscr E}_{-,\alpha} &\sim_M  \|g_\alpha ^{(s)}\|_{L^2_{x,\xi }}^2 +  \|g_{\alpha+1}\|_{L^2_{x,\xi }}^2 \\
\widetilde{\mathscr D}_{-,\alpha} &\sim_M  \|g_\alpha ^{(s)}\|_{L^2_{x}(\mathcal H_\sigma\cap \mathcal H_{\sigma;\eps}^-)_{\xi }}^2 +  \|g_{\alpha+1}\|_{L^2_{x}(\mathcal H_\sigma\cap \mathcal H_{\sigma;\eps}^-)_\xi }^2
\end{align}

%
%
%
\subsection{Preliminary estimates}

Next, we have upper and lower bounds on the linearized collision operators $\widetilde{\mathcal L}_{q,\gamma}$.

\begin{lemma}\label{lem:LandGammaBds}
Assume \eqref{eq:bootstrap3}. Let $h_1,h_2,h_3 : \mathbf R^3 \to \mathbf R$.
We have the following:
\begin{enumerate}
\item[(i)]  Then
\begin{align}
\|\mathcal P_\gamma^\perp h_1\|_{\mathcal H_\sigma}^2 \lesssim  \langle \mathcal L_\gamma h_1,h_1\rangle_{L^2} \lesssim \|\mathcal P^\perp_\gamma h_1\|_{\mathcal H_\sigma}^2. \label{eq:Lbeta}
\end{align}
\item [(ii)] Taking $q \in \{q_1,q_2,q_3\}$, and $h_1,h_2 :\mathbf R^3\to \mathbf R$,
\begin{align}
\langle \widetilde{ \mathcal L}_{\gamma,q}h_1,h_2\rangle_{L^2} &\lesssim  \|\mathcal P_\gamma^\perp h_1\|_{\mathcal H_{\sigma}}  \|\mathcal P_\gamma^\perp h_2\|_{\mathcal H_{\sigma}} + \eta  \| h_1\|_{\mathcal H_{\sigma}}  \| h_2\|_{\mathcal H_{\sigma}},  \label{eq:Lqbeta1} \\
\|\mathcal P_\gamma^\perp h_1\|_{\mathcal H_{\sigma}}^2  &\lesssim \langle \widetilde{ \mathcal L}_{\gamma,q}h_1,h_1\rangle_{L^2} +\eta^2  \|\mathcal P_\gamma h_1\|^2_{L^2}.\label{eq:Lqbeta2}
\end{align}

\item[(iii)] Take $\eta > 0$ sufficiently small. Suppose $h_1 =\sqrt{ \frac{{\mu_q}}{{\mu_\gamma}}}h_2$ where $q \in\{q_1,q_2,q_3\}$. Then,
\begin{align}\label{eq:projClose}
 \|\mathcal P_\gamma^\perp h_1 \|_{\mathcal H_\sigma}^2 \lesssim \eta\|\mathcal P_\gamma h_1\|_{L^2}^2 + \|\mathcal P_\gamma^\perp h_2\|_{\mathcal H_\sigma}^2.
\end{align}
Also,
\begin{align}\label{eq:projSplit}
\|h_1\|_{\mathcal H_\sigma} \lesssim \|\mathcal P_\gamma h_1\|_{L^2} + \|\mathcal P_\gamma^\perp h_1\|_{\mathcal H_\sigma}.
\end{align}
\item[(iv)]  If $q\in\{q_1,q_2\}$, then, taking $q' \in [0,q)$, we have
\begin{align}
\langle \Gamma_q(h_1,h_2),h_3\rangle_{L^2} \lesssim_{q'}( \|e^{\frac{q'|\xi|^2}{4} } h_1\|_{L^2} \|h_2\|_{\mathcal H_{\sigma}} + \|e^{\frac{q'|\xi|^2}{4} } h_1\|_{\mathcal H_{\sigma}} \|h_2\|_{L^2})\|h_3\|_{\mathcal H_\sigma}.\label{eq:Gammaq}
 \end{align}
\end{enumerate}
\end{lemma}

\begin{proof}
For (i),
the equivalence \eqref{eq:Lbeta} in the case $\gamma = 2$ can be found in \cite{guo_landau_2002}. The case of general $\gamma$ follows by re-scaling in $\xi $.

For (ii), we compute
\begin{align}
\langle \widetilde{\mathcal L}_{\gamma,q} h_1,h_2\rangle_{L^2}  &= \langle \mathcal L_\gamma h_1, h_2\rangle_{L^2} + \frac{\gamma - q}{2} \{ \langle \sigma^\gamma_{ij} \xi_i h_1, \partial_{\xi _j}  h_2\rangle_{L^2} - \langle\sigma^\gamma_{ij}  \xi_j \partial_{\xi _i}  h_1,h_2\rangle_{L^2} \} \\
& \quad -\frac{(\gamma - q)^2}{2}\langle \sigma^\gamma_{ij} \xi_i \xi_j h_1,h_2\rangle_{L^2}
\end{align}
It is clear from the above that we have \eqref{eq:Lqbeta1}.
In the case $h_2 = h_1$, we have by \eqref{eq:Lbeta} and the computation above that
\begin{align}
\langle \widetilde{\mathcal L}_{\gamma,q} h_1,h_1\rangle_{L^2}&  \geq \langle \mathcal L_\gamma h_1,h_1\rangle_{L^2_\xi } - C\eta^2 \|h_1\|_{\mathcal H_\sigma}^2 \\
&\geq (\frac{1}{C}- C\eta^2) \|\mathcal P_\gamma^\perp h_1\|_{\mathcal H_\sigma}^2 - C\eta^2 |\mathcal P_\gamma h_1|_{\mathcal H_\sigma}^2.
\end{align}
Taking $\eta$ sufficiently small, we get \eqref{eq:Lqbeta2}.
%
%

We now prove (iii).
We write $ \zeta =\sqrt{ \frac{\mu_q}{\mu_{\gamma}}} $. Now,
\begin{align}
\|\mathcal P_\gamma^\perp h_1\|_{\mathcal H_\sigma} = \|\mathcal P_\gamma^\perp (\zeta h_2)\|_{\mathcal H_\sigma} &\leq \|\mathcal P_\gamma^\perp (\zeta \mathcal P_\gamma h_2)\|_{\mathcal H_\sigma} + \|\mathcal P_\gamma^\perp(\zeta \mathcal P_\gamma^\perp h_2)\|_{\mathcal H_\sigma} \\
&\lesssim      \|\mathcal P_\gamma^\perp ((\zeta-1) \mathcal P_\gamma h_2)\|_{\mathcal H_\sigma} + \|\mathcal P_\gamma^\perp h_2\|_{\mathcal H_\sigma} \\
&\lesssim     \|(\zeta-1) \mathcal P_\gamma h_2\|_{\mathcal H_\sigma} + \|\mathcal P_\gamma^\perp h_2\|_{\mathcal H_\sigma}
\end{align}
Now, $|\zeta- 1|\lesssim \eta |\xi|^2$, and $|\mathcal P_\gamma h_2(\xi)|\lesssim \langle \xi\rangle^2 e^{-\frac{\gamma}{2} |\xi|^2 }\|\mathcal P_\gamma h_2\|_{L^2}$, 
$\|(\zeta - 1) \mathcal P_\gamma h_2\|_{\mathcal H_\sigma}  \lesssim \eta \|\mathcal P_\gamma h_2\|_{L^2}$. Now,
 \begin{align}
 \|\mathcal P_\gamma h_2\|_{L^2}& \leq \|\mathcal P_\gamma (\zeta^{-1} \mathcal P_\gamma h_1)\|_{L^2}  + \|\mathcal P_\gamma ((\zeta ^{-1}-1)\mathcal P_\gamma^\perp h_1)\|_{L^2} \\
 & \lesssim \|\mathcal P_\gamma h_1\|_{L^2} + \eta \|\mathcal P_\gamma^\perp h_1\|_{L^2}. 
 \end{align}
Combining these bounds, we deduce \eqref{eq:projClose} after taking $\eta$ sufficiently small. From the latter estimate, we deduce \eqref{eq:projSplit}.


Finally, we prove (iv).  The bound \eqref{eq:Gammaq} in the case $q = 2$ follows from the proof of Theorem 3 in \cite{guo_landau_2002}. The case of general $q$ follows by re-scaling in $\xi $.
\end{proof}

\begin{lemma}
Suppose $q \in \{q_1,q_2,q_3\}$ and assume \eqref{eq:bootstrap3} with $\eta$ taken sufficiently small. Let $G,h_1,h_2: \mathbf R^3 \to \mathbf R$, $G = G(v)$, $h_1  =  h_1(\xi)$, $h_2=  h_2(\xi)$.
\begin{enumerate}\item[(i)]  We have the upper bound
\begin{equation}
\begin{aligned}
\langle \mathcal M_{q,G}  h_1, h_2\rangle_{L^2} &\lesssim \|\langle v\rangle^3 G\|_{ L^2} (\| h_1\|_{\mathcal H_{\sigma;\eps}^-} +\eps^{\frac{1}{2}}\|\langle \frac{\xi }{\eps}\rangle^{-\frac{1}{2}} h_1\|_{L^2}) ( \| h_2\|_{\mathcal H_{\sigma;\eps}^-}+\eps^{\frac{1}{2}} \|\langle \frac{\xi }{\eps}\rangle^{-\frac{1}{2}} h_2\|_{L^2}  ) \\
&\quad + \eps^{\frac{1}{2}}\|\langle v\rangle^{\frac{7}{2}} G\|_{\mathcal H_\sigma}  \|\langle \frac{\xi }{\eps}\rangle^{-\frac{1}{2}} h_1\|_{L^2}\| h_2\|_{\mathcal H_{\sigma;\eps}^-}.\label{eq:M_qGupper}
\end{aligned}
\end{equation}
\item[(ii)] Assume $G\geq0$,  $\|\langle v\rangle^3 G\|_{L^2},  \|\langle v\rangle^{\frac{7}{2}} G\|_{\mathcal H_\sigma} < \infty$ and $\|G\|_{L^1_v}>0$. Let
\begin{align}
k_G := \sup \{1, \|\langle v\rangle^3 G\|_{L^2_v}, \frac{1}{\|G\|_{L^1_v}}\},
\end{align}
Then,
\begin{align}
\langle \mathcal M_{q,G}  h_1, h_1\rangle_{L^2_\xi } \geq \frac{1}{C_{k_G}}\| h_1\|_{\mathcal H^-_{\sigma;\eps}}^2 - C_{k_G}\eps \langle \|\langle v\rangle^{\frac{7}{2}} G\|_{\mathcal H_\sigma}\rangle^2 \| h_1\|_{L^2}^2.\label{eq:M_qGlower}
\end{align}
\end{enumerate}
\end{lemma}

\begin{proof}

Recall that $ \mathcal M_{q,G}h = \frac{1}{\sqrt{\mu_q}}Q_{+-}^\eps (\mu_q,\sqrt{\mu_q} h)$, with $Q_{+-}^\eps$ defined in \eqref{eq:Q_+-}. Then,
\begin{align}
&\langle \mathcal M_{q,G} h_1,  h_2 \rangle_{L^2_\xi } \\
&\quad = \frac{1}{\eps} \langle( \Phi_{ij} * G)|_{v = \frac{\xi }{\eps}}(\partial_{\xi _i}- \frac{q \xi_i}{2})  h_1 , (\partial_{\xi _j} + \frac{q \xi_i}{2})  h_1\rangle_{L^2_\xi } \label{eq:M_qGT1}\\
&\quad \quad -\langle(\partial_{v_i} \Phi_{ij} * G)|_{v = \frac{\xi }{\eps}} h_1 , (\partial_{\xi _j} + \frac{q \xi_i}{2})  h_1\rangle_{L^2_\xi }.\label{eq:M_qGT2}
\end{align}

For the first term, we have
\begin{align}
\eqref{eq:M_qGT1}& \lesssim  \frac{1}{\eps} \langle( \Phi_{ij} * |G|)|_{v = \frac{\xi }{\eps}}(\partial_{  \xi_i}- \frac{q \xi_i}{2})  h_1 , (\partial_{\xi _j} - \frac{q \xi_i}{2})  h_1\rangle_{L^2_\xi }^{\frac{1}{2}} \\
& \cdot \langle( \Phi_{ij} * |G|)|_{v = \frac{\xi }{\eps}}(\partial_{\xi _i}+\frac{q \xi_i}{2})  h_2 , (\partial_{\xi _j} + \frac{q \xi_i}{2})  h_2\rangle_{L^2_\xi }^{\frac{1}{2}} \\
&\lesssim \|\langle v\rangle^3 G\|_{L^2_v}(\| h_1\|_{\mathcal H^-_{\sigma;\eps}}^2 + \frac{1}{\eps}\langle \sigma_{ij}|_{v= \frac{\xi }{\eps}}\xi_i\xi_j h_1 ,  h_1\rangle_{L^2_\xi })^{\frac{1}{2}}  \\
&\quad \cdot (\| h_2\|_{\mathcal H^-_{\sigma;\eps}}^2 +\frac{1}{\eps} \langle \sigma_{ij}|_{v= \frac{\xi }{\eps}}\xi_i\xi_j  h_2 ,  h_2\rangle_{L^2_\xi })^{\frac{1}{2}}.
 \end{align}
Above we used the upper bound \eqref{eq:PhiConvUpper}. Now, recall that by \eqref{eq:sigmaComp}, $\sigma_{ij}(z)z_i z_j \sim \frac{|z|^2}{\langle z\rangle^3}$. Hence, 
\begin{equation}
\frac{1}{\eps} \sigma_{ij}(\frac{\xi }{\eps})\xi_i\xi_j \sim \frac{|\xi|^2}{\eps\langle \xi/\eps\rangle^3 } = \frac{\eps|\xi/\eps |^2}{\langle \xi/\eps\rangle^3 }  \leq \frac{\eps}{\langle \xi/\eps\rangle}\leq \min\{\eps,\frac{1}{|\xi|}\}.
\end{equation}
Thus,
\begin{equation}
\eqref{eq:M_qGT1} \lesssim \|\langle v\rangle^3 G\|_{ L^2_v} (\| h_1\|_{\mathcal H_{\sigma;\eps}^-} +\eps^{\frac{1}{2}}\| h_1\|_{L^2_\xi }) ( \| h_2\|_{\mathcal H_{\sigma;\eps}^-}+\eps^{\frac{1}{2}} \| h_2\|_{L^2_\xi }  ).
\end{equation}
Next, we consider the second term: using \eqref{eq:sigmaComp} again, and recalling the definition of $\mathcal H_{\sigma,-}^\eps$ in  \eqref{eq:H-}, we have
\begin{align}
&|\langle(\partial_{v_i} \Phi_{ij} * G)|_{v = \frac{\xi }{\eps}} h_1 , (\partial_{\xi _j} + \frac{q \xi_i}{2})  h_2\rangle_{L^2_\xi }|\\
& \lesssim \left( \int_{\mathbf R^3} ( \sigma^{-1})_{ij}(\frac{\xi }{\eps})\partial_{v_k} \Phi_{ki} * G|_{v= \frac{\xi }{\eps}} \partial_{v_l} \Phi_{lj} * G|_{v= \frac{\xi }{\eps}}  h_1^2(\xi) d\xi \right)^{\frac{1}{2}}\\
&\quad \cdot  \langle \sigma_{ij}( \frac{\xi }{\eps})(\partial_{\xi _i}+\frac{q \xi_i}{2})  h_2 , (\partial_{\xi _j} + \frac{q \xi_i}{2})  h_2\rangle_{L^2_\xi }^{\frac{1}{2}}\\
&\lesssim  \left( \int_{\mathbf R^3} \langle \frac{\xi }{\eps}\rangle^{3} |(\nabla_v (-\Delta_v)^{-1} G)(\frac{\xi }{\eps}) |^2 h_1^2(\xi) d\xi \right)^{\frac{1}{2}} \\
&\lesssim \eps^{\frac{1}{2}} \sup_{v \in \mathbf R^3, i \in \{1,2,3\}} \{ \langle v\rangle^{2} |\nabla_v(-\Delta_v)^{-1}G(v)|\} \|\langle \frac{\xi }{\eps}\rangle^{-\frac{1}{2}} h_1\|_{L^2_\xi } (\| h_2\|_{\mathcal H^-_{\sigma;\eps}} + \eps^{\frac{1}{2}} \|\langle \frac{\xi }{\eps}\rangle^{-\frac{1}{2}}  h_2\|_{L^2_\xi })
\end{align}
Now, we have
\begin{equation}
\sup_{v \in \mathbf R^3} \|\langle v\rangle^2 |\partial_{v_j}\Phi_{ij}* G(v)\| \lesssim \|\langle v\rangle^{2} G\|_{H^1_{v}} \lesssim \|\langle v\rangle^{\frac{7}{2}}G \|_{\mathcal H_\sigma}.
\end{equation}\\

\noindent\textbf{Proof of (ii):} We write
\begin{align}
\langle \mathcal M_{q,G} h_1,  h_1 \rangle_{L^2_\xi } &= \frac{1}{\eps} \langle( \Phi_{ij} * G)|_{v = \frac{\xi }{\eps}}(\partial_{\xi _i}- \frac{q \xi_i}{2})  h_1 , (\partial_{\xi _j} + \frac{q \xi_i}{2})  h_1\rangle_{L^2_\xi } \\
&\quad -\langle(\partial_{v_i} \Phi_{ij} * G)|_{v = \frac{\xi }{\eps}} h_1 , (\partial_{\xi _j} + \frac{q \xi_i}{2})  h_1\rangle_{L^2_\xi }\\
&= \frac{1}{\eps} \langle( \Phi_{ij} * G)|_{v = \frac{\xi }{\eps}}\partial_{\xi _i}  h_1 , \partial_{\xi _j}  h_1\rangle_{L^2_\xi } \label{eq:M_qGlowerT1} \\
&\quad -\frac{1}{\eps} \langle( \Phi_{ij} * G)|_{v = \frac{\xi }{\eps}}\xi_i \xi_j  h_1 ,  h_1\rangle_{L^2_\xi } \label{eq:M_qGlowerT2} \\
&\quad -\langle(\partial_{v_i} \Phi_{ij} * G)|_{v = \frac{\xi }{\eps}} h_1 , (\partial_{\xi _j} + \frac{q \xi_i}{2})  h_1\rangle_{L^2_\xi } \label{eq:M_qGlowerT3}
\end{align}
Now, applying \eqref{eq:PhiConvLower}, we have
\begin{equation}
\eqref{eq:M_qGlowerT1} \geq  \frac{1}{ C_{k_G}\eps} \langle \sigma_{ij}|_{v=\frac{\xi }{\eps}} \partial_{\xi _i}  h_1, \partial_{\xi _j}  h_1\rangle_{L^2_\xi }.
\end{equation}
Once again, we apply \eqref{eq:PhiConvUpper} to get
\begin{equation}\label{eq:PhiGwwBd}
\frac{1}{\eps}\sup_{\xi  \in \mathbf R^3}\Big|(\Phi_{ij}*G)(\frac{\xi }{\eps}) \xi_i \xi_j\Big| \lesssim \frac{k_G}{\eps} \sigma_{ij}(\frac{\xi}{\eps}) \xi_i\xi_j  \lesssim \frac{k_G|\xi|^2}{\eps\langle \xi/\eps\rangle^3} \leq \frac{k_G\eps}{\langle \xi/\eps\rangle}.
\end{equation}
Hence,
\begin{align}
|\eqref{eq:M_qGlowerT2}|\leq C_{k_G}\eps  \| \langle \frac{\xi }{\eps}\rangle^{-\frac{1}{2}} h_1\|_{L^2}^2
\end{align}
Finally, we reuse the bound from part (i) to get for all $\lambda >0$,
\begin{align}
|\eqref{eq:M_qGlowerT3}|& \lesssim \eps^{\frac{1}{2}} \|\langle v\rangle^{\frac{7}{2}} G\|_{\mathcal H_\sigma} \|\langle \frac{\xi }{\eps}\rangle^{-\frac{1}{2}} h_1\|_{L^2_\xi } \| h_1\|_{\mathcal H_{\sigma;\eps}^-}\\
&\leq \frac{\lambda}{C_{k_G}}  \| h_1\|_{\mathcal H_{\sigma;\eps}^-}^2 +\frac{ C_{k_G}\eps }{\lambda}  \|\langle v\rangle^{\frac{7}{2}} G\|_{\mathcal H_\sigma}^2 \|\langle \frac{\xi }{\eps}\rangle^{-\frac{1}{2}} h_1\|_{L^2}^2.
\end{align}
Taking $\lambda$ sufficiently small so that the first term in the above is dominated by \eqref{eq:M_qGlowerT1}, we deduce \eqref{eq:M_qGlower}.
\end{proof}
%

\subsection{Energy estimates}

\begin{proposition}

We assume that  \eqref{eq:bootstrap1}, \eqref{eq:bootstrap2} and \eqref{eq:bootstrap3} hold.  Then, for all $\kappa > 0$, $t\in[0,T]$:

\begin{enumerate}\label{prop:energyEstimates}
\item[(i)] for $\alpha\in \{1,2,3\}$, we have
\begin{equation}\label{eq:gjEnergy}
\begin{aligned}
&\eps\frac{d}{dt} \{\| g_\alpha \|_{L^2_{x,\xi }}^2 + 4\pi \sqrt{\gamma}\||\Delta_x|^{-1} a\|_{L^2_x}^2\}+ \frac{1}{C}\{ \|\mathcal P_\gamma^\perp g_\alpha \|_{L^2_x(\mathcal H_\sigma)_\xi }^2 + \|g_\alpha \|^2_{L^2_x(\mathcal H_{\sigma;\eps}^-)}\}\\
&\quad  \leq  C( \eps +  \varsigma + \eta +  \kappa) \|\mathcal P_\gamma f\|_{L^2_x(\mathcal H_\sigma)_\xi }^2 \\
&\quad \quad  +C  \eps \langle  \|\partial_t a\|_{L^2_x}  + \|F_+\|_{\mathfrak D} ^2\rangle\|g_\alpha \|_{L^2_{x,\xi }}^2\\
&\quad \quad + C_{\kappa}\eps^2 \langle \|F_+\|_{\mathfrak D}  \rangle^2
\end{aligned}
\end{equation}

\item[(ii)] For $\alpha\in \{1,2\}$,
\begin{equation}\label{eq:g1sEnergy}
\begin{aligned}
&\frac{\eps}{2}\frac{d}{dt} \| g_\alpha ^{(s)}\|_{L^2_{x,\xi }}^2+ \frac{1}{C}\{ \|\mathcal P_\gamma^\perp g_\alpha ^{(s)}\|_{L^2_x(\mathcal H_\sigma)_\xi }^2 + \|g_\alpha ^{(s)}\|^2_{L^2_x(\mathcal H_{\sigma;\eps}^-)}\}\\
&\quad  \leq  C( \eps +  \varsigma + \eta +  \kappa) \|\mathcal P_\gamma f^{(s)}\|_{L^2_x(\mathcal H_\sigma)_\xi }^2 + C_{\eta,\kappa}\|g_{\alpha+1}\|_{L^2_x(\mathcal H_\sigma\cap \mathcal H_{\sigma;\eps}^-)_\xi }^2  \\
&\quad \quad  +C  \eps \langle  \|\partial_t a\|_{L^2_x}  + \|F_+\|_{\mathfrak D}  ^2\rangle\|g_\alpha ^{(s)}\|_{L^2_{x,\xi }}^2\\
&\quad  \quad+ C_{\kappa}\eps^2 \langle  \|F_+\|_{\mathfrak D}  \rangle^2 
\end{aligned}
\end{equation}
\end{enumerate}

\end{proposition}
\begin{proof} We separate the proof into each part:

\noindent\textbf{Proof of (i):} We  multiply \eqref{eq:gjEq} by $g_\alpha $ and integrate:
\begin{align}
&\frac{\eps}{2}\frac{d}{dt} \| g_\alpha \|_{L^2_{x,\xi }}^2 \\
&\quad + \frac{q_\alpha \eps}{2} \langle\partial_t \phi g_\alpha , g_\alpha \rangle_{L^2_{x,\xi }}\label{eq:g_jT1}\\
&\quad  +\langle e^{\gamma \psi} \widetilde {\mathcal L}_{\gamma,q_\alpha } g_\alpha , g_\alpha \rangle_{L^2_{x,\xi }} \label{eq:g_jT2} \\
&\quad +\langle  \mathcal M_{q_\alpha ,F_+}g_\alpha , g_\alpha \rangle_{L^2_{x,\xi }}\label{eq:g_jT3}\\
&\quad -4\pi \gamma \langle \xi\cdot \nabla_x\Delta_x^{-1}a  \frac{\mu_\gamma}{\sqrt{\mu_{q_\alpha }}} e^{\gamma \psi - \frac{q_\alpha }{2}\phi},g_\alpha \rangle_{L^2_{x,\xi }} \label{eq:g_jT3.5} \\
& =- \eps \langle\frac{\partial_t\{ \mu_{\gamma}e^{\gamma\psi}\}}{\sqrt{\mu_{q_\alpha }e^{q_\alpha \phi}}},  g_\alpha \rangle_{L^2_{x,\xi }}\label{eq:g_jT4}\\
&\quad -\langle e^{\gamma \psi- \frac{q_\alpha }{2} \phi}\mathcal M_{q_\alpha ,F_+}(\frac{ \mu_\gamma}{\sqrt{\mu_{q_\alpha }}}),  g_\alpha \rangle_{L^2_{x,\xi }}\label{eq:g_jT5}\\
&\quad  +\langle  e^{\frac{q_\alpha }{2} \phi} \Gamma_{q_\alpha }(g_\alpha ,g_\alpha ),g_\alpha \rangle_{L^2_{x,\xi }}\label{eq:g_jT6}.
\end{align}
By \eqref{eq:dtBigPhiandBetaBd}, 
\begin{align}
|\eqref{eq:g_jT1}|&\lesssim \eps (\|\partial_t \psi\|_{L^\infty_x}+ \|\Delta_x^{-1} \partial_t a\|_{L^\infty_x})\|g_\alpha \|_{L^2_{x,\xi }}^2 \\
&\lesssim  \eps  \langle \|g_\alpha \|_{L^2_x(\mathcal H_\sigma)_\xi } +  \| \partial_t a\|_{L^2_x}\rangle\|g_\alpha \|_{L^2_{x,\xi }}^2  \\
&\lesssim \eps^2+ \eps  \langle \| \partial_t a\|_{L^2_x}\rangle \|g_\alpha \|_{L^2_{x,\xi }}^2+ \varsigma \|g_\alpha \|_{L^2_x(\mathcal H_\sigma)_\xi }^2
\end{align}
Next, we apply \eqref{eq:Lqbeta2} to get
\begin{align}
 \eqref{eq:g_jT2}& \geq \frac{1}{C} e^{-C\|\psi\|_{L^\infty_x}} \|\mathcal P_\gamma^{\perp} g_\alpha \|_{L^2_{x,\xi }}^2 - Ce^{C\|\psi\|_{L^\infty_x}} \eta^2 \|\mathcal P_\gamma g_\alpha \|_{L^2_{x,\xi }}^2 \\
 &\geq \frac{1}{C}  \|\mathcal P_\gamma^{\perp} g_\alpha \|_{L^2_{x,\xi }}^2 - C\eta^2 \|g_\alpha \|_{L^2_x(\mathcal H_\sigma)_\xi }^2.
 \end{align}
 In the second line, we use
 $\|\phi\|_{L^\infty_x} \lesssim M + \varsigma \lesssim 1$.
  Next, we  apply \eqref{eq:M_qGlower} to get
\begin{align}
\eqref{eq:g_jT3} \geq  \frac{1}{C_{k_{F_+}}} \|g_\alpha \|_{L^2_x(\mathcal H^-_{\sigma;\eps})_\xi }^2 -C_{K_{F_+}}\eps \langle \|\langle v\rangle^{\frac{7}{2}} F_+\|_{L^\infty_x(\mathcal H_\sigma)_\xi }\rangle ^2 \|g_\alpha \|_{L^2_{x,\xi }}^2
\end{align}
where
\begin{align}
k_{F_+} = \max\{1, \|\langle v\rangle^3 F_+\|_{L^\infty_x L^2_v}, \|n_+^{-1}\|_{L^\infty_x}\}  \lesssim 1.
\end{align}
 Using this, we have
\begin{align}
\eqref{eq:g_jT3}\geq \frac{1}{C}\|g_\alpha \|_{L^2_x(\mathcal H^-_{\sigma;\eps})_\xi }^2 -C\eps \langle  \|F_+\|_{\mathfrak D}  \rangle^2 \|g_\alpha \|_{L^2_{x,\xi }}^2
\end{align}

For the next term, we write
\begin{align}
\frac{e^{\frac{q_\alpha }{2} \phi}g_\alpha }{\sqrt{\mu_{q_\alpha }}} = \frac{\sqrt{\mu_\gamma}f }{\mu_{q_\alpha }} = \frac{f}{\sqrt{\mu_\gamma}} + \left(\frac{1}{\sqrt{\mu_\gamma}} -\frac{\sqrt{\mu_\gamma} }{\mu_{q_\alpha }} \right) f = \frac{f}{\sqrt{\mu_\gamma}} + \left(\frac{\mu_{q_\alpha }}{\mu_\gamma} -1 \right)\frac{e^{\frac{q_\alpha }{2}\phi} g_\alpha }{\sqrt{\mu_{q_\alpha }}}
\end{align}
Then,
\begin{align}
\eqref{eq:g_jT3.5} &=  -4\pi \gamma \langle \xi\cdot \nabla_x\Delta_x^{-1}a  \sqrt{\mu_\gamma} e^{\gamma \psi - q_\alpha \phi},f \rangle_{L^2_{x,\xi }} \\
&\quad -4\pi \gamma \langle \xi\cdot \nabla_x\Delta_x^{-1}a \frac{\mu_{q_\alpha } - \mu_\gamma}{\sqrt{\mu_{q_\alpha }}}e^{\gamma \psi - \frac{q_\alpha }{2}\phi}, g_\alpha  \rangle_{L^2_{x,\xi }} \\
&=-4\pi \gamma \langle \nabla_x\Delta_x^{-1}a  ,b \rangle_{L^2_{x,\xi }} \label{eq:g_jT3.5-1}\\
&\quad -4\pi \gamma \langle (e^{\gamma \psi - q_\alpha \phi}-1)\nabla_x\Delta_x^{-1}a ,b \rangle_{L^2_{x,\xi }}\label{eq:g_jT3.5-2} \\
&\quad -4\pi \gamma \langle \xi\cdot \nabla_x\Delta_x^{-1}a \frac{\mu_{q_\alpha } - \mu_\gamma}{\sqrt{\mu_{q_\alpha }}}e^{\gamma \psi - \frac{q_\alpha }{2}\phi}, g_\alpha  \rangle_{L^2_{x,\xi }}\label{eq:g_jT3.5-3}
\end{align}
For term \eqref{eq:g_jT3.5-1}, we use the continuity equation (see \eqref{eq:massConsPert} below)
\begin{align}\label{eq:cont}
\eps \partial_t a + \sqrt{\gamma} \nabla_x \cdot b = -\eps \partial_t (e^{\gamma \psi}),
\end{align}
which gives 
\begin{align}
\eqref{eq:g_jT3.5-1} &= 2\pi  \sqrt{\gamma}\eps\frac{d}{dt} \||\nabla_x|^{-1} a\|_{L^2_x}^2 -4\pi \sqrt{\gamma} \eps \langle \partial_t (e^{\gamma \psi}),a\rangle_{L^2_x} \\
&= 2\pi \eps\frac{d}{dt} \{ \sqrt{\gamma} \||\nabla_x|^{-1} a\|_{L^2_x}^2\} \\
&\quad -\frac{\eps \pi \dot \gamma}{\sqrt{\gamma}}\||\nabla_x|^{-1} a\|_{L^2_x}^2 -4\pi \sqrt{\gamma} \eps \langle \partial_t (e^{\gamma \psi}),a\rangle_{L^2_x}
\end{align}
Now, by \eqref{eq:dtBigPhiandBetaBd},
\begin{align}
|\frac{\eps \pi \dot \gamma}{\sqrt{\gamma}}\||\nabla_x|^{-1} a\|_{L^2_x}^2 +4\pi \sqrt{\gamma} \eps \langle \partial_t (e^{\gamma \psi}),a\rangle_{L^2_x}|&\lesssim \eps (1+ |\dot \gamma | + \|\partial_t \psi\|_{L^\infty_x})\|a\|_{L^2_x}\\
& \lesssim \eps( 1 + \|g_\alpha \|_{L^2_x(\mathcal H_\sigma)_\xi }) \|g_\alpha \|_{L^2_x(\mathcal H_\sigma)_\xi } \\
&\lesssim \eps^2 + \eps \|g_\alpha \|_{L^2_x(\mathcal H_\sigma)_\xi }^2.
\end{align}
Hence,
\begin{align}
|\eqref{eq:g_jT3.5-1} - 2\pi \eps\frac{d}{dt} \{ \sqrt{\gamma} \||\nabla_x|^{-1} a\|_{L^2_x}^2\}| \lesssim \eps^2 +\eps \|g_\alpha \|_{L^2_x(\mathcal H_\sigma)_\xi }^2.
\end{align}
For \eqref{eq:g_jT3.5-2}, we have
\begin{align}
\|e^{\gamma \psi - q_\alpha  \phi} -1\|_{L^\infty_x} = \|e^{(\gamma -q_\alpha )\psi - 4\pi q_\alpha  \Delta_x^{-1} a} -1\|_{L^\infty_x} \lesssim \varsigma + \eta .
\end{align}
Thus, 
\begin{align}
|\eqref{eq:g_jT3.5-2}| &\lesssim (\varsigma + \eta) \|a\|_{L^2_x} \|b\|_{L^2_x} \\
& \lesssim (\varsigma + \eta)\|g_\alpha \|_{L^2_x(\mathcal H_\sigma)_\xi }^2. 
\end{align}
Third, we have \eqref{eq:g_jT3.5-3}. For this, we note that when $\eta$ is taken sufficiently small, we can guarantee that
\begin{align}
|\frac{\mu_{q_\alpha } - \mu_\gamma}{\sqrt{\mu_{q_\alpha }}}| \lesssim \eta e^{-\frac{\beta_{in}}{8}|\xi|^2}
\end{align}
Hence,
\begin{align}
|\eqref{eq:g_jT3.5-3}| \lesssim \eta \|g_\alpha \|_{L^2_x(\mathcal H_\sigma)_\xi }^2.
\end{align}
Combining these bounds, we find
\begin{align}
|\eqref{eq:g_jT3.5} -2\pi \eps\frac{d}{dt} \{ \sqrt{\gamma} \||\nabla_x|^{-1} a\|_{L^2_x}\}| \lesssim \eps^2 +(\eps + \varsigma + \eta) \|g_\alpha \|_{L^2_{x}(\mathcal H_\sigma)_\xi }^2 
\end{align}
We move on to our next term:
\begin{align}
\eqref{eq:g_jT4} & \lesssim \eps \|\langle \xi\rangle^{\frac{1}{2}} \frac{\partial_t\{ \mu_{\gamma}e^{\gamma\psi}\}}{\sqrt{\mu_{q_\alpha }e^{q_\alpha \phi}}}  \|_{L^2_{x,\xi }} \|\langle \xi\rangle^{-\frac{1}{2}}g_\alpha \|_{L^2_{x,\xi }} \\
&\lesssim\eps \|\langle \xi\rangle^{\frac{1}{2}} \frac{\partial_t\{ \mu_{\gamma}e^{\gamma\psi}\}}{\sqrt{\mu_{q_\alpha }e^{q_\alpha \phi}}} \|_{L^2_{x,\xi }} \|g_\alpha \|_{L^2_{x}(\mathcal H_\sigma)_\xi } .
\end{align}
Now,
\begin{align}
 \frac{\partial_t\{ \mu_{\gamma}e^{\gamma\psi}\}}{\sqrt{\mu_{q_\alpha }e^{q_\alpha \phi}}} = (\frac{\dot \gamma }{\gamma} (\frac{3}{2} -\gamma  (|\xi|^2 - \psi)) + \gamma \partial_t \psi) \frac{\mu_\gamma}{\sqrt{\mu_{q_\alpha }}} e^{\gamma \psi - \frac{q_\alpha }{2} \phi}.
\end{align}
Taking $\eta$ sufficiently small, we ensure that $\langle \xi\rangle^{\frac{1}{2}} \mu_{q_\alpha }^{-\frac{1}{2}}\mu_\gamma \lesssim e^{-\frac{\beta_{in}|\xi|^2}{8}}$. Thus, combining this with \eqref{eq:dtBigPhiandBetaBd},
\begin{align}
\eqref{eq:g_jT4}& \lesssim \eps ( |\dot \gamma| + \| \partial_t \psi\|_{L^\infty_x}) \|g_\alpha \|_{L^2_{x}(\mathcal H_\sigma)_\xi }\\
&\leq  C_\kappa \eps^2 + (\eps + \kappa) \|g_\alpha \|_{L^2_{x}(\mathcal H_\sigma)_\xi }^2.
\end{align}
Next, we bound \eqref{eq:g_jT5} using \eqref{eq:M_qGupper}:
\begin{align}
\eqref{eq:g_jT5} &\lesssim (\|\mu_{q_\alpha }^{-\frac{1}{2}}\mu_{\gamma}\|_{\mathcal H_{\sigma;\eps}^-} +\eps^{\frac{1}{2}}\|\langle \frac{\xi }{\eps}\rangle^{-\frac{1}{2}}\mu_{q_\alpha }^{-\frac{1}{2}}\mu_{\gamma}\|_{ L^2_\xi }) ( \|g_\alpha \|_{L^2_x(\mathcal H_{\sigma;\eps}^-)_\xi}+\eps^{\frac{1}{2}} \|\langle \frac{\xi }{\eps}\rangle^{-\frac{1}{2}}g_\alpha \|_{L^2_{x,\xi }}  ) \\
&\quad +\eps^{\frac{1}{2}}  \|F_+\|_{\mathfrak D}   \|\langle \frac{\xi }{\eps} \rangle^{-\frac{1}{2}}\mu_{q_\alpha }^{-\frac{1}{2}}\mu_{\gamma}\|_{ L^2_\xi } \|g_\alpha \|_{L^2_x (\mathcal H_{\sigma;\eps}^-)_\xi }
\end{align}
Now, by \eqref{eq:sigmaComp} again, we have
\begin{align}
\|\mu_{q_\alpha }^{-\frac{1}{2}}\mu_{\gamma}\|_{\mathcal H_{\sigma;\eps}^-}  &\lesssim  \frac{1}{\eps}\int_{\mathbf R^3} \sigma_{ij}(\frac{\xi }{\eps})\xi_i\xi_j e^{-\frac{\beta_{in}}{10}|\xi|^2} d\xi  \\
&\lesssim  \frac{1}{\eps}\int_{\mathbf R^3} \frac{|\xi|^2}{\langle \xi/\eps\rangle^3 }e^{-\frac{\beta_{in}}{10}|\xi|^2} d\xi  \\
&\leq  \eps^2 \int_{\mathbf R^3}  \frac{1}{|\xi| }e^{-\frac{\beta_{in}}{10}|\xi|^2}  d\xi\\
&\lesssim \eps^2.
\end{align}
since $\frac{1}{|\xi|}$ is locally integrable.
Similarly,
\begin{align}
\|\langle \frac{\xi }{\eps}\rangle^{-\frac{1}{2}}\mu_{q_\alpha }^{-\frac{1}{2}}\mu_{\gamma}\|_{L^2_{x,\xi }} \lesssim \eps.
\end{align}
Moreover, by Sobolev embedding
\begin{align}
\|\langle \frac{\xi }{\eps}\rangle^{-\frac{1}{2}}g_\alpha \|_{L^2_{x,\xi }} \lesssim \eps^{\frac{1}{2}} \|1_{B_1}(\xi)|\xi|^{-\frac{1}{2}}g_\alpha \|_{L^2_{x,\xi }} + \|\langle \xi\rangle^{-\frac{1}{2}} g_\alpha \|_{L^2_{x,\xi }} \lesssim \|g_\alpha \|_{L^2_x (\mathcal H_\sigma)_\xi }.
\end{align}
Hence,
\begin{align}
|\eqref{eq:g_jT5} |&\lesssim \eps\langle  \|F_+\|_{\mathfrak D}   \rangle\|g_\alpha \|_{L^2_x(\mathcal H_\sigma \cap \mathcal H_{\sigma;\eps}^-)_\xi }\end{align}
Thus, taking $\lambda >0$ to be chosen later,
\begin{align}
|\eqref{eq:g_jT5}| &\leq \kappa \|g_\alpha \|_{L^2_x(\mathcal H_\sigma \cap \mathcal H_{\sigma;\eps}^-)_\xi }^2 +C_{\kappa}\eps^2 \langle  \|F_+\|_{\mathfrak D}   \rangle^2.
\end{align}
Finally, applying \eqref{eq:Gammaq}, we have
\begin{align}
\eqref{eq:g_jT6}& \lesssim \|g_1\|_{L^\infty_xL^2_\xi }  \|g_\alpha \|_{L^2_x(\mathcal H_\sigma)_\xi }^2 \\
&\lesssim \|g_1^{(s)}\|_{L^2_{x,\xi }} \|g_\alpha \|_{L^2_x(\mathcal H_\sigma)_\xi }^2 \\
& \lesssim \varsigma  \|g_\alpha \|_{L^2_x(\mathcal H_\sigma)_\xi }^2.
\end{align}
Combining the upper bounds for \eqref{eq:g_jT1} through  \eqref{eq:g_jT6}, we conclude
\begin{align}
&\frac{\eps}{2}\frac{d}{dt} \{\| g_\alpha \|_{L^2_{x,\xi }}^2 + 4\pi  \sqrt{\gamma}\||\Delta_x|^{-1} a\|_{L^2_x}^2\}+ \frac{1}{C}\{ \|\mathcal P_\gamma^\perp g_\alpha \|_{L^2_x(\mathcal H_\sigma)_\xi }^2 + \|g_\alpha \|^2_{L^2_x(\mathcal H_{\sigma;\eps}^-)}\}\\
&\quad  \leq  C( \eps +  \varsigma + \eta +  \kappa) \|g_\alpha \|_{L^2_x(\mathcal H_\sigma \cap \mathcal H_{\sigma;\eps}^-)_\xi }^2 \\
&\quad \quad  +C  \eps \langle  \|\partial_t a\|_{L^2_x}  + \|F_+\|_{\mathfrak D}  ^2\rangle\|g_\alpha \|_{L^2_{x,\xi }}^2\\
&\quad \quad + C_{\kappa}\eps^2 \langle  \|F_+\|_{\mathfrak D}  \rangle^2.
\end{align}
Using \eqref{eq:projSplit} and taking $\eps, \varsigma,\eta$ and $\kappa$ sufficiently small, conclude \eqref{eq:gjEnergy}.

We now prove (ii). 
Let $g^{(s)}_j := \langle \nabla_x\rangle^s g_\alpha $. Then,
\begin{align}
&\frac{\eps}{2}\frac{d}{dt} \| g_\alpha ^{(s)}\|_{L^2_{x,\xi }}^2 \\
&\quad +q_\alpha \eps \langle\langle \nabla_x\rangle^s\{\partial_t \phi g_\alpha \}, g_\alpha ^{(s)}\rangle_{L^2_{x,\xi }} \label{eq:g_jsT1}\\
&\quad + \langle \langle \nabla_x\rangle^s \{E g_\alpha \}  - E g_\alpha ^{(s)}  ,\nabla_v g_\alpha ^{(s)}\rangle_{L^2_{x,\xi }} \label{eq:g_jsT2} \\
&\quad  +\langle\langle\nabla_x\rangle^{s}( e^{\gamma \psi} \widetilde {\mathcal L}_{\gamma,q_\alpha } g_\alpha ),g _j^{(s)}\rangle_{L^2_{x,\xi }} \label{eq:g_jsT3}\\
&\quad -4\pi \gamma \langle \frac{\mu_\gamma}{\sqrt{\mu_{q_\alpha }}} \langle \nabla_x\rangle^s \{\xi\cdot \nabla_x\Delta_x^{-1}a  e^{\gamma \psi - \frac{q_\alpha }{2}\phi}\},g_\alpha ^{(s)}\rangle_{L^2_{x,\xi }} \label{eq:g_jsT3.5}\\
&\quad +\langle \langle \nabla_x\rangle^s(  \mathcal M_{q_\alpha ,F_+}g_\alpha ), g_\alpha ^{(s)}\rangle_{L^2_{x,\xi }} \label{eq:g_jsT4}\\
& =- \eps \langle\langle \nabla_x\rangle^s\left( \frac{\partial_t\{ \mu_{\gamma}e^{\gamma\psi}\}}{\sqrt{\mu_{q_\alpha }e^{q_\alpha \phi}}}\right),  g_\alpha ^{(s)}\rangle_{L^2_{x,\xi }} \label{eq:g_jsT5}\\
&\quad -\langle  \langle \nabla_x\rangle^s\{e^{\gamma \psi - \frac{q_\alpha }{2}\phi}\mathcal M_{q_\alpha ,F_+}(\frac{\mu_\gamma}{\sqrt{\mu_{q_\alpha }}} )\},  g_\alpha ^{(s)}\rangle_{L^2_{x,\xi }} \label{eq:g_jsT6}\\
&\quad  +\langle \langle \nabla_x\rangle^s \{ e^{\frac{q_\alpha }{2} \phi} \Gamma_{q_\alpha }(g_\alpha ,g_\alpha )\},g_\alpha ^{(s)}\rangle_{L^2_{x,\xi }} \label{eq:g_jsT7}
\end{align}
First,
\begin{align}
|\eqref{eq:g_jsT1} | &= q_\alpha \eps |\langle \langle \nabla_x \rangle^s\{(\partial_t \psi + 4\pi \Delta_x^{-1}\partial_t a)g_\alpha \},g_\alpha ^{(s)}\rangle_{L^2_{x,\xi }} \\
&\lesssim \eps \int_{\mathbf R^3}   \big( \|\partial_t \psi\|_{H^s_x} \|g_\alpha ^{(s)}\|_{L^2_{x}}^2 + \|\partial_t a\|_{H^{s-2}_x} \|g_\alpha \|_{L^\infty_x}\|g_\alpha ^{(s)}\|_{L^2_{x}} \\
&\quad\quad +\|\Delta_x^{-1} \partial_t a\|_{L^\infty_x} \|g_\alpha ^{(s)}\|_{L^2_{x}} \big)d\xi 
\end{align}
Applying \eqref{eq:cont} to the second term, we get
\begin{align}
&\eps \int_{\mathbf R^3} \|\partial_t a\|_{H^{s-2}_x} \|g_\alpha \|_{L^\infty_x}\|g_\alpha ^{(s)}\|_{L^2_{x}} d\xi \\
&\quad  \lesssim  \int_{\mathbf R^3} \|b\|_{H^{s-1}_x} \|g_\alpha ^{(s-1)}\|_{L^2_x} \|g_\alpha ^{(s)}\|_{L^2_x} d\xi  + \eps (1+|\dot \gamma|  + \|\partial_t \psi\|_{H^s_x})\|g_\alpha ^{(s)}\|_{L^2_{x,\xi }}^2 \\
&\lesssim \|b\|_{H^{s-1}_x}  \|\langle \xi\rangle^{\frac{1}{2}}g_\alpha ^{(s-1)}\|_{L^2_{x,\xi }} \|\langle \xi\rangle^{-\frac{1}{2}}g_\alpha ^{(s)}\|_{L^2_{x,\xi }} + \eps (1+|\dot \gamma|  + \|\partial_t \psi\|_{H^s_x})\|g_\alpha ^{(s)}\|_{L^2_{x,\xi }}^2
\end{align}
Collecting terms, 
\begin{align}
|\eqref{eq:g_jsT1} | &\lesssim \|b\|_{H^{s-1}_x}  \|\langle \xi\rangle^{\frac{1}{2}}g_\alpha ^{(s-1)}\|_{L^2_{x,\xi }} \\
&\quad +  \eps (1+ \|\partial_t a\|_{L^\infty_x}  + |\dot \gamma|  + \|\partial_t \psi\|_{H^s_x})\|g_\alpha ^{(s)}\|_{L^2_{x,\xi }}^2
\end{align}
Applying \eqref{eq:dtBigPhiandBetaBd}, we get
\begin{align}
\eps (|\dot \gamma|  + \|\partial_t \psi\|_{H^s_x})\|g_\alpha ^{(s)}\|_{L^2_{x,\xi }}^2 \lesssim \eps^2 + \eps \|g_\alpha ^{(s)}\|_{L^2_{x,\xi }}^2 + \varsigma  \|g_\alpha ^{(s)}\|_{L^2_x(\mathcal H_\sigma)_\xi }^2.
\end{align}
On the other hand, 
using interpolation, we have
\begin{align}
\|\langle \xi\rangle^{\frac{1}{2}} g_\alpha ^{(s-1)}\|_{L^2_{x,\xi }}\|\langle \xi\rangle^{-\frac{1}{2}}g_\alpha ^{(s)}\|_{L^2_{x,\xi }} \leq  (C_{\eta}\|g_{\alpha+1}\|_{L^2_x(\mathcal H_\sigma)_\xi }^2  + \|g_\alpha ^{(s)}\|_{L^2_x(\mathcal H_\sigma)_\xi }^2).
\end{align}
Thus,
\begin{align}
|\eqref{eq:g_jsT1} | \lesssim \eps \langle \|\partial_t a\|_{L^2_x}\rangle \|g_\alpha ^{(s)}\|_{L^2_{x,\xi }}^2 +C\varsigma\|g_\alpha ^{(s)}\|_{L^2_x(\mathcal H_\sigma)_\xi }^2  +  C_{\eta}\|g_{\alpha+1}\|_{L^2_x(\mathcal H_\sigma)_\xi }^2 
\end{align}
Next,
\begin{align}
|\eqref{eq:g_jsT2}| &\lesssim  \|\langle \xi\rangle^{\frac{1}{2}} (\langle \nabla_x\rangle^s \{E  g_\alpha \} - E  g_\alpha ^{(s)})\|_{L^2_{x,\xi }} \|g_\alpha ^{(s)}\|_{L^2_x(\mathcal H_\sigma)_\xi } \\
&\lesssim  (\|E \|_{H^s_x} \|\langle \xi\rangle^{\frac{1}{2}}g_\alpha \|_{L^\infty_{x}L^2_\xi } +\|\nabla_x E \|_{L^\infty_{x}} \|\langle \xi\rangle^{\frac{1}{2}}g_\alpha ^{(s-1)}\|_{L^2_{x,\xi }} )  \|g_\alpha ^{(s)}\|_{L^2_x(\mathcal H_\sigma)_\xi }
\end{align}
Above we use the usual commutator estimate for $[\langle \nabla_x\rangle^s, E _j]$. Next we interpolate
\begin{align}
\|\langle \xi\rangle^{\frac{1}{2}}g_\alpha ^{(s-1)}\|_{L^2_{x,\xi }} &\lesssim \|\langle \xi\rangle^{s-\frac{1}{2}}  g_\alpha \|_{L^2_{x,\xi }}^{\frac{1}{s}}\|\langle \xi\rangle^{-\frac{1}{2}}g_\alpha ^{(s)}\|_{L^2_{x,\xi }}^{\frac{s-1}{s}}\\
&\lesssim C_\eta \|\langle \xi\rangle^{-\frac{1}{2}}g_{\alpha+1}\|_{L^2_{x,\xi }}^{\frac{1}{s}}\|\langle \xi\rangle^{-\frac{1}{2}}g_\alpha ^{(s)}\|_{L^2_{x,\xi }}^{\frac{s-1}{s}} \\
&\lesssim C_\eta \|\langle \xi\rangle^{-\frac{1}{2}}g_{\alpha+1}\|_{L^2_x(\mathcal H_\sigma)_\xi }^{\frac{1}{s}}\|\langle \xi\rangle^{-\frac{1}{2}}g_\alpha ^{(s)}\|_{L^2_{x}(\mathcal H_\sigma)_\xi }^{\frac{s-1}{s}}
\end{align}
This implies 
\begin{align}
|\eqref{eq:g_jsT2}|  \leq \kappa \|g_\alpha ^{(s)}\|_{L^2_{x}}^2 + C_{\eta,\kappa} \|\langle \xi\rangle^{-\frac{1}{2}}g_{\alpha+1}\|_{L^2_x(\mathcal H_\sigma)_\xi }^2.
\end{align}
Next, applying \eqref{eq:Lqbeta1} and \eqref{eq:Lqbeta2}, we have
\begin{align}
\eqref{eq:g_jsT3}&= \langle e^{\gamma \psi} \widetilde{\mathcal L}_{\gamma,q} g_\alpha ^{(s)},g _\alpha^{(s)}\rangle_{L^2_{x,\xi }} \\
&\quad + \langle \widetilde {\mathcal L}_{\gamma,q} \{\langle\nabla_x\rangle^{s}( e^{\gamma \psi} g_\alpha ) -e^{\gamma \psi} g_\alpha ^{(s)}\},g _1^{(s)}\rangle_{L^2_{x,\xi }}\\
&\geq \frac{1}{C}\|\mathcal P_\gamma^\perp g_\alpha ^{(s)}\|_{L^2_x(\mathcal H_\sigma)_\xi }^2 - C\varsigma^2 \| g_\alpha ^{(s)}\|_{L^2_x(\mathcal H_\sigma)_\xi }^2 \\
&\quad - C\|\langle\nabla_x\rangle^{s}( e^{\gamma \psi} g_\alpha ) -e^{\gamma \psi} g_\alpha ^{(s)}\|_{L^2_x(\mathcal H_\sigma)_\xi } \|g_\alpha ^{(s)}\|_{L^2_x(\mathcal H_\sigma)_\xi }
\end{align}
Now,
\begin{align}
 \|\langle\nabla_x\rangle^{s}( e^{\gamma \psi} g_\alpha ) -e^{\gamma \psi} g_\alpha ^{(s)}\|_{L^2_x(\mathcal H_\sigma)_\xi }&\lesssim \|\langle \nabla_x\rangle^{s} e^{\gamma \psi}\|_{L^2_x} \|g_\alpha ^{(s-1)}\|_{L^2_x(\mathcal H_\sigma)_\xi }  \\
 &\lesssim  \|g_\alpha ^{(s-1)}\|_{L^2_x(\mathcal H_\sigma)_\xi }.
  \end{align}
Hence, we can interpolate to get
\begin{align}
\eqref{eq:g_jsT3}
&\geq \frac{1}{C}\|\mathcal P_\gamma^\perp g_\alpha ^{(s)}\|_{L^2_x(\mathcal H_\sigma)_\xi }^2 - \kappa \| g_\alpha ^{(s)}\|_{L^2_{x}(\mathcal H_\sigma)_\xi }^2 \\
&\quad - C_\kappa\|g_{\alpha+1}\|_{L^2_x(\mathcal H_\sigma)_\xi }^2
\end{align}
For the next term, we have
\begin{align}
|\eqref{eq:g_jsT3.5}| &\lesssim  \|e^{\gamma \psi - \frac{q_\alpha }{2} \phi}\nabla_x \Delta_x^{-1}a\|_{H^s}\|g_\alpha ^{(s)}\|_{L^2_x(\mathcal H_\sigma)_\xi }  \\
&\lesssim \|g_\alpha ^{(s-1)}\|_{L^2_x(\mathcal H_\sigma)_\xi }  \|g_\alpha ^{(s)}\|_{L^2_x(\mathcal H_\sigma)_\xi }  \\
&\leq \kappa   \|g_\alpha ^{(s)}\|_{L^2_x(\mathcal H_\sigma)_\xi }^2 +   C_\kappa  \|g_{\alpha+1}\|_{L^2_x(\mathcal H_\sigma)_\xi }^2.
\end{align}
For our next term, we have
\begin{align}
\eqref{eq:g_jsT4} &=\langle  \mathcal M_{q_1,F_+}g_\alpha ^{(s)}, g_\alpha ^{(s)}\rangle_{L^2_{x,\xi }} \label{eq:g_jT4T1} \\
&\quad +\langle \langle \nabla_x\rangle^s(  \mathcal M_{q_1,F_+}g_\alpha ) - \mathcal M_{q_1,F_+}g_\alpha ^{(s)} , g_\alpha ^{(s)}\rangle_{L^2_{x,\xi }}\label{eq:g_jT4T2}
\end{align}
For the first term, similarly to \eqref{eq:g_jsT3}, we have
\begin{align}
\eqref{eq:g_jT4T1} \geq \frac{1}{C}  \|g_\alpha ^{(s)}\|_{L^2_x(\mathcal H^-_{\sigma;\eps})_\xi }^2- C \eps \langle\|F_+^{(s,m)}\|_{L^2_x(\dot{\mathcal H}_\sigma)_v} \rangle^2 \|g_\alpha ^{(s)}\|_{L^2_{x,\xi }}^2
\end{align}
Next, combining the commutator estimate with the proof of \eqref{eq:M_qGupper}, we have
\begin{align}
|\eqref{eq:g_jT4T2}|& \lesssim  \|F^{(m_1,s)}_+\|_{L^2_{x,v}} ( \|g^{(s-1)}_1\|_{L^2_x(\mathcal H_{\sigma;\eps}^-)_\xi }+\eps^{\frac{1}{2}} \|\langle \frac{\xi }{\eps}\rangle^{-\frac{1}{2}}g^{(s-1)}_1\|_{L^2_{x,\xi }}  ) \\
&\quad \cdot (\|g^{(s)}_1\|_{L^2_x(\mathcal H_{\sigma;\eps}^-)_\xi } +\eps^{\frac{1}{2}}\|\langle \frac{\xi }{\eps}\rangle^{-\frac{1}{2}}g^{(s)}_1\|_{L^2_{x,\xi }})\\
&\quad + \eps^{\frac{1}{2}} \|F_+\|_{\mathfrak D}   \|g^{(s-1)}_1\|_{L^2_{x,\xi }}\|g^{(s)}_1\|_{L^2_x(\mathcal H_{\sigma;\eps}^-)_\xi }\\
&\lesssim  (\|g_\alpha \|_{L^2_x(\mathcal H_{\sigma;\eps}^-)_\xi } +\eps \|g_\alpha \|_{L^2_x(\mathcal H_{\sigma})_\xi })^{\frac{1}{s}}  (\|g_\alpha \|_{L^2_x(\mathcal H_{\sigma;\eps}^-)_\xi }+\eps \|g_\alpha \|_{L^2_x(\mathcal H_{\sigma})_\xi })^{\frac{2s-1}{s}} \\
&\quad + \eps^{\frac{1}{2}} \|F_+\|_{\mathfrak D}    \|g^{(s)}_1\|_{L^2_{x,\xi }}\|g^{(s)}_1\|_{L^2_x(\mathcal H_{\sigma;\eps}^-)_\xi }.
\end{align}
Thus, we get
\begin{align}
\eqref{eq:g_jsT4}& \geq \frac{1}{C} \|g_\alpha ^{(s)}\|_{L^2_x(\mathcal H^-_{\sigma;\eps})_\xi }^2  -C\{\|g_{\alpha+1}\|_{L^2_x(\mathcal H_{\sigma;\eps}^-)_\xi }^2 + \|g_{\alpha+1}\|_{L^2_x(\mathcal H_{\sigma})_\xi }^2\}\\
&\quad  -C \eps\langle \|F_+\|_{\mathfrak D}  \rangle^2\|g_\alpha ^{(s)}\|_{L^2_{x,\xi }}^2
 \end{align}
Next, similarly to \eqref{eq:g_jT4},
\begin{align}
\eqref{eq:g_jsT5}&\leq C_{\kappa}\eps  + \kappa \|g_\alpha ^{(s)}\|_{L^2_x (\mathcal H_\sigma)_\xi }
\end{align}
and,  as with \eqref{eq:g_jT5}, we use \eqref{eq:M_qGupper} to  get
\begin{align}
\eqref{eq:g_jsT6}&\leq \kappa \|g_\alpha ^{(s)}\|_{L^2_{x}(\mathcal H_\sigma)_\xi }^2 +  C_\kappa \eps^2 \langle \|F_+\|_{\mathfrak D}  \rangle^2 .
\end{align}
Finally, using \eqref{eq:Gammaq}, 
\begin{align}
\eqref{eq:g_jsT7} &\lesssim  \|g_\alpha ^{(s)}\|_{L^2_{x,\xi }}\|g_\alpha ^{(s)}\|_{L^2_x(\mathcal H_\sigma)_\xi }^2\\
& \lesssim \varsigma\|g_\alpha ^{(s)}\|_{L^2_x(\mathcal H_\sigma)_\xi }^2
\end{align}
We now combine the estimates for the terms \eqref{eq:g_jsT1} through \eqref{eq:g_jsT7}:
\begin{align}
&\frac{\eps}{2}\frac{d}{dt} \| g_\alpha ^{(s)}\|_{L^2_{x,\xi }}^2+ \frac{1}{C}\{ \|\mathcal P_\gamma^\perp g_\alpha ^{(s)}\|_{L^2_x(\mathcal H_\sigma)_\xi }^2 + \|g_\alpha ^{(s)}\|^2_{L^2_x(\mathcal H_{\sigma;\eps}^-)_\xi}\}\\
&\quad  \leq  C( \eps +  \varsigma + \eta +  \kappa) \|g_\alpha ^{(s)}\|_{L^2_x(\mathcal H_\sigma)_\xi }^2 + C_{\eta,\kappa}\{\|g_{\alpha+1}\|_{L^2_x(\mathcal H_{\sigma;\eps}^-)_\xi }^2 + \|g_{\alpha+1}\|_{L^2_x(\mathcal H_{\sigma})_\xi }^2\} \\
&\quad \quad  +C  \eps \langle  \|\partial_t a\|_{L^2_x}  + \|F_+\|_{\mathfrak D}  ^2\rangle\|g_\alpha ^{(s)}\|_{L^2_{x,\xi }}^2\\
&\quad \quad + C_{\kappa}\eps^2 \langle  \|F_+\|_{\mathfrak D}  \rangle^2.
\end{align}
Using \eqref{eq:projSplit} again, we conclude with \eqref{eq:g1sEnergy}.
\end{proof}
%


\subsection{Macroscopic estimates}

What remains is to get bounds on $\mathcal P_\gamma f$. This requires analysis of the local conservation laws. In order to derive these equations efficiently, and, in particular, see how the transport part of \eqref{eq:f} couples $(a,b,c)$, we introduce the ladder operators: the lowering and raising operators are respectively
\begin{align}
  \mathcal A_j  = \gamma^{\frac{1}{2}} \frac{  \xi_j}{2} + \gamma^{-\frac{1}{2}}  \partial_{\xi _j},\ \ \  \mathcal A_j^*  =\gamma^{\frac{1}{2}} \frac{ \xi_j}{2} -  \gamma^{-\frac{1}{2}} \partial_{\xi _j}, \ \ \ j \in \{1,2,3\}.
\end{align}
We recall the identities
\begin{equation}
\mathcal A_j \mu^{\frac{1}{2}} = 0, \ \ \ [\mathcal A_j,\mathcal A^*_k] = \varsigma_{jk}, \ \ \ j, k \in \{1,2,3\}.\label{eq:ladderCommutator}
\end{equation}
Recall also that the family of Hermite functions
\begin{equation}
\{\frac{1}{\sqrt{n_1!n_2!n_3!}} (\mathcal A_{1}^*)^{n_1} (\mathcal A_{2}^*)^{n_2} (\mathcal A_3^*)^{n_3} \mu_\gamma^{\frac{1}{2}}\}_{n_1,n_2,n_3 \in \mathbf N_0^3}
\end{equation}
gives a complete orthonormal basis for $L^2_\xi (\mathbf R^3)$.
We shall denote the (unnormalized) Hermite functions
\begin{equation}
\mathfrak h = \mu_\gamma^{\frac{1}{2}}, \ \ \ \mathfrak h_{j_1\ldots j_m} = \mathcal A_{j_1} \cdots \mathcal A_{j_n} \mu_\gamma^{\frac{1}{2}},  \ \ \ j_1,\ldots,j_m  \in\{1,2,3\}.
\end{equation}
We can represent the kernel of $\mathcal  L_\gamma$ using the hermite functions as follows:
\begin{equation}
N(  \mathcal L) = \mathrm{span}\{\mathfrak h ,\mathfrak h_1, \mathfrak h_2,\mathfrak h_3, \frac{1}{\sqrt{6}}\mathfrak h_{jj}\}.
\end{equation}
We can thus rewrite \eqref{eq:NLprojection} as
\begin{align}
\mathcal P_\gamma f &= a \mathfrak h+ b_j  \mathfrak h_j + c \frac{1}{\sqrt{6}}  \mathfrak h_{jj}
\end{align}
It is also convenient to define the following projection of $h$, involving third-order hermite functions:
\begin{equation}
d_j= \frac{1}{\sqrt{10}} \langle \mathfrak h_{jkk},f\rangle_{L^2_\xi }, \ \ \ j \in \{1,2,3\}
\end{equation}
We now rewrite the transport part of \eqref{eq:f} in terms of $\mathcal  A,\mathcal  A^*$
\begin{align}\label{eq:transportLadder}
 \xi\cdot \nabla_x + E \cdot (\nabla_\xi  - \frac{\xi }{2})= (\gamma^{-\frac{1}{2}}\partial_{x_j} + \gamma^{\frac{1}{2}} E _{j}) \mathcal A^*_j  + \gamma^{-\frac{1}{2}}\partial_{x_j}  \mathcal A_j.
\end{align}
We also will need to evaluate projections of $\mu^{-\frac{1}{2}} \partial_t (\mu^{\frac{1}{2}}f)$. Let $\zeta(\xi)$ be some linear combination of the hermite functions in $\gamma^{\frac{1}{2}}\xi$, i.e. there exists a polynomial $p:\mathbf R^3 \to \mathbf R$ (with coefficients independent of $\gamma$) such that
\begin{equation}
\zeta(\xi) = p(\gamma^{\frac{1}{2}} \xi) \mu(\xi)^{\frac{1}{2}}.
\end{equation}
Then,
\begin{equation}
\langle \zeta, \mu^{-\frac{1}{2}}\partial_t (\mu^{\frac{1}{2}} f)\rangle_{L^2_\xi } = \partial_t \langle \zeta, f\rangle_{L^2_\xi } - \langle \partial_t( \zeta  \mu^{-\frac{1}{2}})\mu^{\frac{1}{2}},  f\rangle_{L^2_\xi }
\end{equation}
Now,
\begin{equation}
\partial_t\{ \zeta  \mu^{-\frac{1}{2}}\} = \partial_t \{p (\gamma^{\frac{1}{2}} \xi)\} = \frac{\dot \gamma}{2\gamma} \xi\cdot \nabla_\xi  \{p(\gamma^{\frac{1}{2}} \xi)\} = \frac{\dot \gamma}{2\gamma} \xi\cdot \nabla_\xi  \{\zeta \mu^{-\frac{1}{2}}\}.
\end{equation}
Therefore,
\begin{align}
\partial_t\{\zeta  \mu^{-\frac{1}{2}}\} \mu^{\frac{1}{2}} &= \frac{\dot \gamma}{2\gamma} (\xi\cdot \nabla_\xi    + \frac{\gamma |\xi|^2}{2})\zeta \\
&=\frac{\dot \gamma}{2\gamma} \xi\cdot (\nabla_\xi  + \frac{\gamma \xi }{2} )\zeta \\
&=\frac{\dot \gamma}{2\gamma} (\mathcal A_j+ \mathcal A^*_j) \mathcal A_j\zeta
\end{align}
In summary,
\begin{equation} \label{eq:dtLadder}
\langle \zeta, \mu^{-\frac{1}{2}}\partial_t (\mu^{\frac{1}{2}} f)\rangle_{L^2_\xi } = \partial_t \langle \zeta, f\rangle_{L^2_\xi }  - \frac{\dot \gamma}{2\gamma} \langle  (\mathcal A_j+ \mathcal A_j^*)\mathcal A_j\zeta,f\rangle_{L^2_\xi }.
\end{equation}
The following lemma contains formulas for relevant projections of the equation \eqref{eq:f}.

\begin{lemma}\label{lemma:projections}
The following formulas hold:
\begin{equation}
\eps \partial_t a + \gamma^{-\frac{1}{2}} \partial_{x_j}  b_j =-\eps \partial_t (e^{\gamma \psi} ) , \label{eq:massConsPert}
\end{equation}

\begin{equation}
\begin{aligned}
&\eps (\partial_t - \frac{\dot \gamma}{2\gamma} )b_j + (\gamma^{-\frac{1}{2}} \partial_{x_j} +\gamma^{\frac{1}{2}} E _{j})a  -4\pi \gamma^{\frac{1}{2}} e^{\gamma \psi}\partial_{x_j} \Delta_x^{-1} a +  \sqrt{\frac{2}{3}}  \partial_{x_j}c \\
& \ \ \ + \partial_{x_k}  \langle \mathfrak h_{jk},\mathcal P^\perp_\gamma f \rangle_{L^2_\xi }  +    \langle \mathfrak h_j,\mathcal M_{\gamma, F_+} f\rangle_{L^2_\xi }  = -e^{\gamma \psi} \langle  \mathfrak h_j, \mathcal M_{\gamma, F_+}\mathfrak h\rangle_{L^2_\xi }, \label{eq:momConsPert}
\end{aligned}
\end{equation}

\begin{equation}
\begin{aligned}
& \eps ( \partial_t c - \frac{\dot \gamma}{\gamma} (c + \sqrt{\frac{3}{2}} a)) +\sqrt{ \frac{2}{3}} (\gamma^{-\frac{1}{2}} \partial_{x_j} + \gamma^{\frac{1}{2}} E _{j}) b_j + \sqrt{\frac{5}{3}} \gamma^{-\frac{1}{2}} \partial_{x_j} d_j \\
& \ \ \ + \frac{1 }{\sqrt{6}}\langle \mathfrak h_{jj}, \mathcal M_{\gamma, F_+} f \rangle_{L^2_\xi } =\eps \sqrt{\frac{3}{2}}\frac{ \dot \gamma}{\gamma} e^{\gamma \psi} +  \frac{e^{\gamma \psi} }{\sqrt{6}}\langle \mathfrak h_{jj},\mathcal M_{\gamma,F_+} \mathfrak h \rangle_{L^2_\xi } ,\label{eq:energyConsPert}
\end{aligned}
\end{equation}

\begin{equation}
\begin{aligned}
& \eps (\partial_td_j  - \frac{\dot \gamma}{2\gamma}(3d_j + \frac{3 \sqrt{2}}{\sqrt{5}} b_j)) +  \frac{3\sqrt{3}}{\sqrt{5}} (\gamma^{-\frac{1}{2}} \partial_{x^j}  + \gamma^{\frac{1}{2}}E _{j})  c\\
& \ \ \ +\sqrt{\frac{2}{5}}(\gamma^{-\frac{1}{2}} \partial_{x_k}  + \gamma^{\frac{1}{2}} E _{k}) \langle \mathfrak h_{jk}, \mathcal P_\gamma^\perp f\rangle_{L^2_\xi }  +\gamma^{-\frac{1}{2}}\frac{1}{\sqrt{10}}\partial_{x_k} \langle   \mathfrak h_{jkll}, f\rangle_{L^2_\xi } \\
& \ \ \ + \frac{1}{\sqrt{10}}\langle \mathfrak h_{jll}, \mathcal L_\gamma f+ \mathcal M_{\gamma, F_+} f\rangle_{L^2_\xi }=\langle\mathfrak h_{jll}, - e^{\gamma \psi} \mathcal M_{\gamma,F_+}\mathfrak h+  \Gamma_\gamma (f,f)\rangle.\label{eq:3rdMomConsPert}
\end{aligned}
\end{equation}

\end{lemma}
\begin{proof}
These formulas follow from direct computation of the $\xi $ integral of \eqref{eq:f} multiplied by $\mathfrak h,\mathfrak h_j,\frac{1}{\sqrt{6}}\mathfrak h_{jj}$ and $\frac{1}{\sqrt{10}}\mathfrak h_{jll}$. We only give the details for
\eqref{eq:momConsPert}. The proof of the other formulas are similar.
Using \eqref{eq:transportLadder} and \eqref{eq:dtLadder}, and the fact that $\mathcal L_\gamma f$,  $\mu^{-\frac{1}{2}} \partial_t \nu_\gamma$ and $\Gamma_\gamma (f,f)$ are orthogonal to $\mathfrak h_j$, we have
\begin{align}
& \eps (\partial_t b_j  - \frac{\dot \gamma}{2\gamma} \langle  (\mathcal A_k\mathcal A_k+ \mathcal A_k^*\mathcal A_k)\mathfrak h_j,f\rangle_{L^2_\xi }).\\
&\quad+ \langle \mathfrak h_j, \{(\gamma^{-\frac{1}{2}}\partial_{x_k} + \gamma^{\frac{1}{2}} E _{k}) \mathcal A^*_k  + \gamma^{-\frac{1}{2}}\partial_{x_k}  \mathcal A_k\}f\rangle_{L^2_\xi } \\
&\quad +\langle \mathfrak h_j, \mathcal M_{\gamma,F_+}f\rangle_{L^2_\xi }\\
& = -\langle \mathfrak h_j ,e^{\gamma \psi}\mathcal M_{\gamma,F_+} \mathfrak h\rangle_{L^2_\xi }.
\end{align}
It remains to evaluate terms involving the ladder operators. Using \eqref{eq:ladderCommutator}, one computes
\begin{align}
(\mathcal A_k+ \mathcal A_k^*)\mathcal A_k\mathfrak h_j &= (\mathcal A^*_k+ \mathcal A_k)\mathcal A_k \mathcal A_j^* \mathfrak h \\
& = (\mathcal A^*_j+ \mathcal A_j) \mathfrak h \\
&=\mathfrak h_j
\end{align}
and
\begin{align}
&\langle \mathfrak h_j, \{(\gamma^{-\frac{1}{2}}\partial_{x_k} + \gamma^{\frac{1}{2}} E _{k}) \mathcal A^*_k  + \gamma^{-\frac{1}{2}}\partial_{x_k}  \mathcal A_k\}f\rangle_{L^2_\xi } \\
&\quad= \langle \mathcal A_k \mathcal A_j ^*\mathfrak h, (\gamma^{-\frac{1}{2}}\partial_{x_k} + \gamma^{\frac{1}{2}} E _{k})f\rangle_{L^2_\xi } \\
&\quad \quad +  \gamma^{-\frac{1}{2}}\langle \mathcal A_k^*\mathcal A^*_j\mathfrak h,\partial_{x_k}f\rangle_{L^2_\xi }\\
&\quad = (\gamma^{-\frac{1}{2}} \partial_{x_j} + \gamma^{\frac{1}{2}} E _k )a \\
&\quad \quad+ \gamma^{-\frac{1}{2}} \partial_{x_k}\langle \mathfrak h_{jk}, \mathcal P_\gamma^{\perp} f\rangle_{L^2_\xi } \\
&\quad\quad + \frac{1}{\sqrt{6}}\gamma^{-\frac{1}{2}} \langle \mathfrak h_{jk}, \mathfrak h_{ll}\rangle_{L^2_\xi }\partial_{x_k}c.
\end{align}
Finally, noting that in the last term,
\begin{align}
\langle \mathfrak h_{jk}, \mathfrak h_{ll}\rangle_{L^2_\xi } = \sqrt{2}\delta_{jk},
\end{align}
we conclude with \eqref{eq:momConsPert}.
\end{proof}

\begin{lemma}\label{lem:bBd}
Assume the hypotheses of Proposition \ref{prop:derivationRefined}. Then, for all $r \geq 0$,
\begin{equation}
\|b\|_{H^{r}_x} \lesssim_M \|\mathcal P_\gamma^\perp f^{(r)}\|_{L^2_x (\mathcal H_\sigma)_\xi } + \| f^{(r)}\|_{L^2_x (\mathcal H_{\sigma;\eps}^-)_\xi } + \eps^2 \|(a,c)\|_{H^r_x}^2 .\label{eq:bBdMain}
\end{equation}
\end{lemma}
\begin{proof}
It suffices to show the case when $r = 0$. Because of radial symmetry of $\mathfrak h$ and $\mathfrak h_{jj}$, we have
\begin{align}
\|(\mathfrak h,\mathfrak h_{jj})\|_{\mathcal H_{\sigma;\eps}^-} \lesssim \eps.
\end{align}
Hence,
\begin{align}
\| f\|_{L^2_x (\mathcal H_{\sigma;\eps}^-)_\xi }^2 \geq \frac{1}{2}\|\mathcal P_\gamma^\perp f + b_j \mathfrak h_j\|_{L^2_x (\mathcal H_{\sigma;\eps}^-)_\xi }^2  - C\eps^2 \|(a,c)\|_{L^2_x}^2
\end{align}
Now, writing $u = \mathcal P_\gamma^\perp f + b_j \mathfrak h_j$, we have
\begin{align}
\|u\|_{L^2_x (\mathcal H_{\sigma;\eps}^-)_\xi }^2  =\frac{1}{\eps} \langle \sigma_{ij}|_{\frac{\xi }{\eps}} \partial_{\xi _i} u,\partial_{\xi _j} u\rangle_{L^2_{x,\xi }} \geq \frac{1}{\eps} \langle  \sigma_{ij}|_{\frac{\xi }{\eps}} \mathbf 1_{B_1^c}(\xi) \partial_{\xi _i} u,\partial_{\xi _j} u\rangle_{L^2_{x,\xi }}.
\end{align}
On the other hand, for all $\nu\in\mathbf S^2$, and $\xi  \in \mathbf R^3\setminus B_1$, we have
\begin{align}
 \sigma_{ij}(\frac{\xi }{\eps}) \nu_i\nu_j \lesssim \sigma_{ij}(\xi)\nu_i\nu_j 
\end{align}
Now, we expand $u$ as follows:
\begin{align}
&{\frac{1}{\eps} \langle  \sigma_{ij}|_{\frac{\xi }{\eps}} \mathbf 1_{B_1^c}(\xi) \partial_{\xi _i} u,\partial_{\xi _j} u\rangle_{L^2_{x,\xi }}} \\
&=\frac{1}{\eps} \langle  \sigma_{ij}|_{\frac{\xi }{\eps}} \mathbf 1_{B_1^c}(\xi) \partial_{\xi _i} (b_k\mathfrak h_k) ,\partial_{\xi _j} (b_l\mathfrak h_l)\rangle_{L^2_{x,\xi }} \\
&\quad  + \frac{2}{\eps}\langle  \sigma_{ij}|_{\frac{\xi }{\eps}} \mathbf 1_{B_1^c}(\xi) \partial_{\xi _i} (b_k\mathfrak h_k) ,\partial_{\xi _j} \mathcal P_\gamma^\perp f \rangle_{L^2_{x,\xi }} \\
&\quad + \frac{1}{\eps}\langle  \sigma_{ij}|_{\frac{\xi }{\eps}} \mathbf 1_{B_1^c}(\xi) \partial_{\xi _i}\mathcal P_\gamma^\perp f,\partial_{\xi _j} \mathcal P_\gamma^\perp f\rangle_{L^2_{x,\xi }}.
\end{align}
Now $\frac{1}{\eps}\langle  \sigma_{ij}|_{\frac{\xi }{\eps}} \mathbf 1_{B_1^c}(\xi) \partial_{\xi _i} \cdot ,\partial_{\xi _j}  \cdot \rangle_{L^2_{x,\xi }} $ defines a  semi-inner product. So, we 
apply Cauchy-Schwarz and Young's inequality to the cross term to get
\begin{align}
&{\frac{1}{\eps} \langle  \sigma_{ij}|_{\frac{\xi }{\eps}} \mathbf 1_{B_1^c}(\xi) \partial_{\xi _i} u,\partial_{\xi _j} u\rangle_{L^2_{x,\xi }}}  \\
&\quad \geq {\frac{1}{2\eps} \langle  \sigma_{ij}|_{\frac{\xi }{\eps}} \mathbf 1_{B_1^c}(\xi) \partial_{\xi _i} (b_k\mathfrak h_k) ,\partial_{\xi _j} (b_l\mathfrak h_l)\rangle_{L^2_{x,\xi }}}  - {\frac{2}{\eps} \langle  \sigma_{ij}|_{\frac{\xi }{\eps}} \mathbf 1_{B_1^c}(\xi) \partial_{\xi _i}\mathcal P_\gamma^\perp f,\partial_{\xi _j} \mathcal P_\gamma^\perp f\rangle_{L^2_{x,\xi }}}  \\
&\quad \geq \frac{1}{C} {W _{kl}\langle b_l ,b_k\rangle_{L^2_x}}-C\|\mathcal P_\gamma^\perp f\|_{L^2_x(\mathcal H_\sigma)_\xi }^2,
\end{align}
where
\begin{align}
W_{kl} := \frac{1}{\eps}\langle  \sigma_{ij}|_{\frac{\xi }{\eps}} \mathbf 1_{B_1^c}(\xi) \partial_{\xi _i} \mathfrak h_k,\partial_{\xi _j} \mathfrak h_l\rangle_{L^2_{\xi }}.
\end{align}
We now show that
\begin{align}\label{eq:WLowerBd}
W_{kl} \nu_k \nu_l \geq\frac{1}{C}
\end{align}
 for all $\nu\in \mathbf S^2$. Now,
\begin{align}
\partial_{\xi _i} \mathfrak h_k =\frac{\gamma^{\frac{1}{2}}}{2} (\delta_{ik} - \frac{\gamma}{2} \xi_i \xi_k ) \mu^{\frac{1}{2}}_\gamma
\end{align}
Then, by the reverse triangle inequality, 
\begin{align}
\sigma_{ij}|_{\frac{\xi }{\eps}}  \partial_{\xi _i} \mathfrak h_k,\partial_{\xi _j} \mathfrak h_l&=\frac{\gamma}{4\eps } \sigma_{ij}|_{\frac{\xi }{\eps}} (\nu_i - \frac{\gamma}{2}(\nu\cdot \xi) \xi_i)(\nu_j - \frac{\gamma}{2}(\nu\cdot \xi) \xi_j)\mu_\gamma \\
&\geq\{\frac{\gamma}{8\eps} \sigma_{ij}|_{\frac{\xi }{\eps}}\nu_i \nu_j - C\frac{ (\nu \cdot \xi)^2}{\eps}  \sigma_{ij}|_{\frac{\xi }{\eps}} \xi_i\xi_j\} \mu_\gamma
\end{align}
Now, for $|\xi| \geq 1$, we have
\begin{align}
\frac{1}{\eps}\sigma_{ij}|_{\frac{\xi }{\eps}} \xi_i \xi_j \sim \frac{|\xi|^2}{\langle \xi/\eps\rangle^3} \sim \eps^2\frac{1}{|\xi|}.
\end{align}
On the other hand, 
\begin{align}
\frac{1}{\eps} \sigma_{ij}|_{\frac{\xi }{\eps}} \nu_i \nu_j \gtrsim \frac{1}{\eps} \frac{|P_{\xi ^\perp} \nu|^2 }{\langle \xi/\eps\rangle} \gtrsim \frac{|P_{\xi ^\perp} \nu|^2}{|\xi|}.
\end{align}
Hence,
\begin{align}
W_{kl} \nu_k \nu_l &\geq \frac{1}{C} \int_{\mathbf R^3\setminus B_1} \frac{|P_{\xi ^\perp} \nu|^2}{|\xi|} \mu_\gamma d\xi   - \eps^2 \int_{\mathbf R^3\setminus B_1} |\xi|\mu_\gamma d\xi  \\
&\geq \frac{1}{C} - C\eps^2.
\end{align}
Taking $\eps$ small enough, we deduce \eqref{eq:WLowerBd}, and thus
\begin{align}
&{\frac{1}{\eps} \langle  \sigma_{ij}|_{\frac{\xi }{\eps}} \mathbf 1_{B_1^c}(\xi) \partial_{\xi _i} u,\partial_{\xi _j} u\rangle_{L^2_{x,\xi }}}\geq  \frac{1}{C} {\|b\|_{L^2_x}^2}-C\|\mathcal P_\gamma^\perp f\|_{L^2_x(\mathcal H_\sigma)_\xi }^2.
\end{align}
From this, we conclude \eqref{eq:bBdMain} in the case $r = 0$.
\end{proof}

The next lemma allows us to control $(a,b,c)$. 
\begin{proposition}\label{prop:macroscopicEstimates} 
Let $M \geq 1$, and let $\eta,   \varsigma$ and $\eps$ be sufficiently small depending on $M$. Assume the bootstrap assumptions \eqref{eq:bootstrap1} through \eqref{eq:bootstrap4}.
 Then,  exists two real valued functions $\mathscr G = \mathscr G^{(M)}$ and $\mathscr G_{reg} = \mathscr G^{(M)}_{reg}$ on $[0,T]$ satisfying, for all $t \in [0,T]$
\begin{align}
|\mathscr G(t)| \lesssim_M \|f(t)\|_{L^2_{x,\xi }}^2, \quad | \mathscr G_{reg}(t)| \lesssim_M \|f^{(s)}(t)\|_{L^2_{x,\xi }}^2
\end{align}
such that 
\begin{equation}\begin{aligned}\label{eq:macroBd0}
&\eps \frac{d}{dt} \mathscr G(t) +  \|\mathcal P_\gamma f\|_{L^2_{x}(\mathcal H_\sigma)_\xi }^2\\
&\quad  \lesssim_M \eps^2 \langle \|F_+\|_{\mathfrak D}  \rangle^2 + \min_{\alpha\in\{1,2,3\}}\{ \|\mathcal P_\gamma^\perp g_\alpha \|_{L^2_x(\mathcal H_\sigma)_\xi }^2+  \|g_\alpha \|^2_{L^2_x (\mathcal H_{\sigma;\eps}^-)_\xi }\},
\end{aligned}
\end{equation}
and
\begin{equation}\begin{aligned}
\label{eq:macroBd1}
&\eps \frac{d}{dt} \mathscr G_{reg}(t) +  \|\mathcal P_\gamma f^{(s)}\|_{L^2_{x}(\mathcal H_\sigma)_\xi }^2\\
& \quad \lesssim_M \eps^2 \langle  \|F_+\|_{\mathfrak D}  \rangle^2 + \min_{\alpha\in \{1,2\}} \{\|\mathcal P_\gamma^\perp g_\alpha ^{(s)}\|_{L^2_x(\mathcal H_\sigma)_\xi }^2+  \|g_\alpha ^{(s)}\|^2_{L^2_x (\mathcal H_{\sigma;\eps}^-)_\xi }\}.
\end{aligned}
\end{equation}
\end{proposition}

\begin{proof} We break the proof into a number of steps. Given $u: \mathbf T^3 \to \mathbf R$, we split $u = u^\times + \overline u$, where $\overline u(t) = \int_{\mathbf T^3} u(t,x) dx$. The first step concerns the bounds on $(\overline a,\overline c)$. In step 2, we bound $\|a\|_{L^2_x}$. In step 3, we bound $(e^{\gamma \psi} c)^\times$. In step 4, we  bound $\|(a,c)\|_{H^s_x}$. In step 5, we synthesize these bounds.\\

\noindent \textbf{Step 1 (estimate on $\overline a$ and $\overline c$):} Regarding $\overline a$, 
from \eqref{eq:atoDensity}, we simply have
\begin{align}
\int_{\mathbf T^3} a(t,x) dx \equiv 0
\end{align}
for all times. 
In particular, $\Delta_x^{-1} a$ is well defined.

We now discuss the zero mode of $c$.  Subtracting the first line of \eqref{eq:PP} from \eqref{eq:energyCons}, 
we have
\begin{align}
\sqrt{\frac{3}{2}}\frac{d}{dt}\left(\frac{\overline c(t)}{\gamma}  \right) &= \frac{d}{dt} \int_{\mathbf T^3 \times \mathbf R^3} \frac{|\xi|^2}{2} (F_- - \mu_\gamma e^{\gamma \psi}) dxd\xi   
\end{align}
Integrating in time, we have 
\begin{align}
\sqrt{\frac{3}{2}}\frac{\overline c(t)}{\gamma(t)}  &= \sqrt{\frac{3}{2}}\frac{\overline c_{in}}{\beta_{in}} + \frac{1}{8\pi} \int_{\mathbf T^3} |\nabla_x \phi_{in}(x)|^2- |\nabla_x \psi _{in}(x)|^2dx\\
& \quad +  \frac{1}{8\pi} \int_{\mathbf T^3} |\nabla_x \psi(t,x)|^2- |\nabla_x \phi(t,x)|^2dx 
\end{align}
On the other hand, from \eqref{eq:aphiPhi},
\begin{align}
\int_{\mathbf T^3} |\nabla_x \psi(t,x)|^2- |\nabla_x \phi(t,x)|^2dx&  =-2 \langle \nabla_x \psi,\nabla_x \Delta_x^{-1} a\rangle_{L^2_x} -\|\nabla_x\Delta_x^{-1} a\|_{L^2_x}\\
& =2 \langle  \psi ,a\rangle_{L^2_x} - \|\Delta_x^{-\frac{1}{2}} a\|_{L^2_x}^2. 
\end{align}
Combining these with \eqref{eq:energyCondition}, we get
\begin{align}\label{eq:barc-aBd}
\left | \sqrt{\frac{3}{2}}\frac{\overline c(t)}{\gamma(t)} -\frac{1}{4\pi} \langle \psi, a\rangle_{L^2_x}  \right| \lesssim \eps+  \|a\|_{L^2_x}^2.
\end{align}
In what follows, it will become clear that we cannot control $c^\times$ directly. Instead, we can only get a bound on $(ce^{\gamma \psi})^\times$. It is then necessary to compute $c^\times$ in terms of $(ce^{\gamma \psi})^\times$ and $\overline c$. Observe that 
\begin{align}
e^{\gamma \psi} (e^{-\gamma \psi}c)^\times +e^{\gamma \psi} \int_{\mathbf R^3}  e^{-\gamma \psi} cdx' = c^\times +  \overline c.
\end{align}
Since $\int_{\mathbf T^3} e^{\gamma \psi} dx = 1$, the above integrated in $x$ yields
\begin{align}
 \int_{\mathbf R^3}  e^{-\gamma \psi} cdx = \overline c- \int_{\mathbf T^3} e^{\gamma \psi} (e^{-\gamma \psi}c)^\times dx .
\end{align}
Thus, we have the identity
\begin{align}\label{eq:ctimesidentity}
c^\times  = (e^{\gamma \psi} - 1)\overline c + e^{\gamma \psi} (e^{-\gamma \psi}c)^\times - e^{\gamma \phi_0}  \int_{\mathbf T^3} e^{\gamma \psi} (e^{-\gamma \psi}c)^\times dx
\end{align}
In particular, $\|c \|_{L^2_x} \lesssim |\overline c| + \|(e^{-\gamma \psi}c)^\times\|_{L^2_x}$. \\

\noindent \textbf{Step 2 (estimate for $a$):} Now, we turn to bounding $a$. We claim that the following estimate holds:
\begin{equation}
\begin{aligned}
&-\eps\frac{d}{dt}  \langle  e^{-\gamma \psi} \nabla_x \Delta_x^{-1} a, b \rangle_{L^2_x} + \frac{1}{C}\|a\|_{L^2_x}^2 \label{eq:aBdMain} \\
&\quad \leq C\{ \eps^2 \langle \|F_+\|_{\mathfrak D}  \rangle^2+\|b\|_{L^2_x}^2 + \|(e^{-\gamma \psi}c)^\times \|_{L^2_x}^2 +\|\mathcal P_\gamma^\perp f\|_{L^2_x(\mathcal H_\sigma)_\xi }^2 + \| f\|_{L^2_x(\mathcal H_{\sigma;\eps}^-)_\xi }^2\}.
\end{aligned}
\end{equation}
To prove this, first
observe
\begin{align}
(\gamma^{-\frac{1}{2}} \nabla_x+\gamma^{\frac{1}{2}} E )a = \gamma^{-\frac{1}{2}} e^{\gamma\psi} \nabla_x (e^{-\gamma\psi} a) -4\pi a \nabla_x \Delta_x^{-1} a ,
\end{align}
Therefore,
\begin{align}
&-\eps\frac{d}{dt}  \langle  e^{-\gamma \psi} \nabla_x \Delta_x^{-1} a, b \rangle_{L^2_x} + \gamma^{-\frac{1}{2}}\|e^{-\frac{\gamma \psi}{2}} a\|_{L^2_x}^2 +4\pi \|e^{\frac{\gamma \psi}{2}} \nabla_x \Delta_x^{-1}a\|_{L^2_x}^2   +4\pi \overline c^2 \\
 &\quad = -\langle  e^{-\gamma \psi} \nabla_x \Delta_x^{-1} a, \eps \partial_t b  + (\gamma^{-\frac{1}{2}} \nabla_x+\gamma^{\frac{1}{2}} E  -4\pi \gamma^{\frac{1}{2}}e^{\gamma \psi} \nabla_x\Delta_x^{-1})a + \sqrt{\frac{2}{3}}\nabla_x c \rangle_{L^2_x} \label{eq:aBdT1}\\
&\quad \quad -\eps \langle e^{-\gamma \psi} \nabla_x \Delta_x^{-1}\partial_t  a, b \rangle_{L^2_x}  \label{eq:aBdT2}\\
&\quad \quad +\eps \langle \partial_t( \gamma \psi) e^{-\gamma \psi}  \nabla_x \Delta_x^{-1}  a, b\rangle_{L^2_x} \label{eq:aBdT3}\\
&\quad \quad +4\pi \eps \langle  e^{-\gamma \psi}  \nabla_x \Delta_x^{-1}  a, a \nabla_x \Delta_x^{-1}a  \rangle_{L^2_x} \label{eq:aBdT4} \\
&\quad \quad + 4\pi \overline c^2+\sqrt{\frac{2}{3}} \langle e^{-\gamma \psi} \nabla_x \Delta_x^{-1} a, \nabla_x c\rangle_{L^2_x}. \label{eq:aBdT5}
\end{align}
From \eqref{eq:momConsPert}, we have 
\begin{align}
&\|\eps \partial_t b_j + (\gamma^{-\frac{1}{2}} \partial_{x_j} +\gamma^{\frac{1}{2}} E _{j})a-4\pi \gamma^{\frac{1}{2}}e^{\gamma \psi} \nabla_x\Delta_x^{-1}a \|_{ H^{-1}_x}\\
& = \Big\| \eps \frac{\dot \gamma}{2\gamma} b_j + \sqrt{\frac{2}{3}}  \partial_{x_j}c  + \partial_{x_k}  \langle \mathfrak h_{jk},\mathcal P^\perp_\gamma f \rangle_{L^2_\xi }  + \langle \mathfrak h_j,\mathcal M_{\gamma, F_+} f\rangle_{L^2_\xi } \\
&\quad   +e^{\gamma \psi} \langle  \mathfrak h_j, \mathcal M_{\gamma, F_+}\mathfrak h\rangle_{L^2_\xi } \Big\|_{ H^{-1}_x}\\
&\lesssim \eps  |\dot \gamma| \|b\|_{L^2_x} + \|c\|_{L^2_x} + \|\mathcal P_\gamma^\perp f\|_{L^2_x(\mathcal H_\sigma)_\xi } \\
&\quad +  \|\langle \mathfrak h_j,\mathcal M_{\gamma, F_+} f\rangle_{L^2_{\xi }}\|_{H^{-1}_x} \label{eq:hjMf} \\
&\quad +\| e^{\gamma \psi} \langle  \mathfrak h_j, \mathcal M_{\gamma, F_+}\mathfrak h\rangle_{L^2_\xi } \|_{ H^{-1}_x} \label{eq:hjMh}
\end{align}
Applying \eqref{eq:M_qGupper} to the latter two terms, we have
\begin{align}
\eqref{eq:hjMf}& \lesssim \left\| \|F_+^{(3,0)}\|_{L^2_{v}}(\|f\|_{(\mathcal H_{\sigma;\eps}^-)_\xi }+ \eps \|f\|_{(\mathcal H_\sigma)_\xi })  \right\|_{H^{-1}_x} \\
&\quad +\eps \left\| \|F_+^{(\frac{7}{2},0)}\|_{(\dot{\mathcal H}_\sigma)_v} \|f\|_{(\mathcal H_\sigma)_\xi }  \right\|_{H^{-1}_x} \\
&\lesssim   \|f\|_{L^2_x(\mathcal H_{\sigma;\eps}^-)_\xi } +\eps \langle   \|F_+\|_{\mathfrak D}   \rangle\|f\|_{L^2_x(\mathcal H_\sigma)_\xi }. 
\end{align}
and
\begin{align}
\eqref{eq:hjMh}& \lesssim \eps \|F_+^{(3,0)}\|_{H^{-1}_xL^2_{v}}  +\eps  \|F_+^{(\frac{7}{2},0)}\|_{H^{-1}_x(\dot{\mathcal H}_\sigma)_\xi } \\
&\lesssim  \eps \langle \|F_+\|_{\mathfrak D}  \rangle. 
\end{align}
Then,
\begin{align}
|\eqref{eq:aBdT1} |& \lesssim  \|e^{\gamma \psi}\nabla_x (-\Delta_x)^{-1} a\|_{H^1_x} \\
&\quad \cdot \|\eps \partial_t b + (\gamma^{-\frac{1}{2}} \nabla_x+\gamma^{\frac{1}{2}} E )a-4\pi \gamma^{\frac{1}{2}}e^{\gamma \psi} \nabla_x\Delta_x^{-1}a \|_{ H^{-1}_x} \\
&\lesssim \|a\|_{L^2_x} \{ \eps  \|b\|_{L^2_x}  + \|\mathcal P_\gamma^\perp f\|_{L^2_x (\mathcal H_\sigma)_\xi } \\
&\quad + \|f\|_{L^2_x(\mathcal H_{\sigma;\eps}^-)_\xi } +\eps \langle \|F_+\|_{\mathfrak D}   \rangle\langle\|f\|_{L^2_x(\mathcal H_\sigma)_\xi }\rangle\}.
\end{align}
Now, note that 
\begin{align} \label{eq:fSplit}
 \|f\|_{L^2_x(\mathcal H_\sigma)_\xi } \lesssim \|a\|_{L^2_x} + \|b\|_{L^2_x}  + |\overline c| +  \|(e^{\gamma \psi} c)^\times\|_{L^2_x}+ \|\mathcal P_\gamma^\perp f\|_{L^2_x(\mathcal H_\sigma)_\xi }.
\end{align}
Then, taking $\lambda >0$, and using Young's inequality,  and the bootstrap assumptions, we have
\begin{align}
|\eqref{eq:aBdT1}| &\leq C(\eps + \varsigma + \lambda) (\|a\|_{L^2_x}^2 + \overline c^2)\\
&\quad  + C_{\lambda} \{\eps^2\langle  \|F_+\|_{\mathfrak D}  \rangle^2  + \|(b,(e^{\gamma \psi}c)^\times)\|_{L^2_x}^2 + \|\mathcal P_\gamma^\perp f\|_{L^2_x(\mathcal H_\sigma)_\xi }^2 +\| f\|_{L^2_x(\mathcal H_{\sigma;\eps}^-)_\xi }^2\}
\end{align}
We now turn to bounding \eqref{eq:aBdT2} and \eqref{eq:aBdT3}.
Using \eqref{eq:dtBigPhiandBetaBd} and \eqref{eq:massConsPert}, we have
\begin{align}
\eps\|\partial_t a\|_{H^{-1}_x} &\lesssim \| b\|_{L^2_x} + \eps\|\partial_t e^{\gamma \psi}\|_{H^{-1}_x}  \lesssim \|b\|_{L^2_x} + \eps\langle \|f\|_{L^2(\mathcal H_\sigma)_\xi }\rangle.\label{eq:dtaBd}
\end{align}
Hence,
\begin{align}
|\eqref{eq:aBdT2}|&\lesssim (\eps \langle \|f\|_{L^2_{x}(\mathcal H_\sigma)_\xi }\rangle + \|b\|_{L^2_x})\|b\|_{L^2_x}\\
&\lesssim \eps^2 \langle \|f\|_{L^2_{x}(\mathcal H_\sigma)_\xi }\rangle^2 + \|b\|_{L^2_x}^2\\
&\lesssim \eps^2 (\|a\|_{L^2_x} + |\overline c|^2) + \|(b,(e^{\gamma \psi}c)^\times)\|_{L^2_x}^2 + \|\mathcal P_\gamma^\perp f\|_{L^2_x(\mathcal H_\sigma)_\xi }^2
\end{align}
Next,
\begin{align}
| \eqref{eq:aBdT3} |&\lesssim  \eps ( \|\partial_t \psi\|_{L^\infty_x}  + |\dot \gamma| ) \|a\|_{L^2_x} \|b\|_{L^2_x}.
\end{align}
Applying \eqref{eq:dtBigPhiandBetaBd}  again, the above reduces to
\begin{align}
 |\eqref{eq:aBdT3}| &\lesssim \eps \langle \|f\|_{L^2_x(\mathcal H_\sigma)_\xi }\rangle\|a\|_{L^2_x} \|b\|_{L^2_x} \\
 &\lesssim \eps^2 (\|a\|_{L^2_x} + |\overline c|^2) + \|(b,(e^{\gamma \psi}c)^\times)\|_{L^2_x}^2 + \|\mathcal P_\gamma^\perp f\|_{L^2_x(\mathcal H_\sigma)_\xi }^2.
 \end{align}
Thus,
\begin{align}
|\eqref{eq:aBdT4}| \lesssim \eps \|a\|_{L^\infty} \|a\|_{H^{-1}}^2 \lesssim \eps \|a\|_{L^2_x}^2. 
\end{align}
Now to bound \eqref{eq:aBdT5}, we use \eqref{eq:ctimesidentity} to get
\begin{align}
\|\nabla_x c - \gamma \nabla_x \psi e^{\gamma \psi} \overline c\|_{H^{-1}_x} \lesssim \|(e^{-\gamma \psi} c)^\times\|_{L^2_x}.
\end{align}
Hence,
\begin{align}
\left| \langle e^{-\gamma \psi}\nabla_x \Delta_x^{-1} a, \nabla_xc \rangle_{L^2_x} - \gamma \langle \nabla_x \Delta_x^{-1} a, \nabla_x \phi_0\rangle_{L^2_x}\overline c\right| \lesssim \|a\|_{L^2_x}\|(e^{-\gamma \psi} c)^\times\|_{L^2_x}.
\end{align}
Combining the above with  \eqref{eq:barc-aBd}, we have
\begin{align}
\eqref{eq:aBdT5}& \leq \sqrt{\frac{2}{3}}\left| \langle e^{-\gamma \psi}\nabla_x \Delta_x^{-1} a, \nabla_xc \rangle_{L^2_x} - \gamma \langle \nabla_x \Delta_x^{-1} a, \nabla_x \phi_0\rangle_{L^2_x}\overline c\right| \\
&\quad +\left|4\pi \overline c - \sqrt{\frac{2}{3}}\gamma \langle    \phi_0,a\rangle_{L^2_x}\right| |\overline c|\\
&\lesssim \varsigma\|a\|_{L^2_x}^2 + \|a\|_{L^2_x}\|(e^{-\gamma \psi} c)^\times\|_{L^2_x} +\eps |\overline c|
\end{align}
Now, combining the bounds on \eqref{eq:aBdT1} through \eqref{eq:aBdT5}, we now have 
\begin{align}
&-\eps\frac{d}{dt}  \langle  e^{-\gamma \psi} \nabla_x \Delta_x^{-1} a, b \rangle_{L^2_x} + \gamma^{-\frac{1}{2}}\|e^{-\frac{\gamma \psi}{2}} a\|_{L^2_x}^2 + 4\pi \overline c^2\\
&\quad \leq C(\eps + \varsigma +  \lambda) (\|a\|_{L^2_x}^2 + \overline c^2)\\
&\quad \quad + C_{\lambda} \{  \eps^2\langle \|F_+\|_{\mathfrak D}   \rangle^2 + \|b\|_{L^2_x}^2 + \|(e^{-\gamma \psi}c)^\times \|_{L^2_x}^2  + \|P_\gamma^\perp f\|_{L^2_x(\mathcal H_\sigma)_\xi } ^2+\| f\|_{L^2_x(\mathcal H_{\sigma;\eps}^-)_\xi }^2\}.
\end{align}
Thus, taking $\eps,\varsigma$ and $\lambda$ sufficiently small, we have \eqref{eq:aBdMain}.\\

%

\noindent \textbf{Step 3 (estimate for $a^{(s)}$):}
Next, we have the following estimate for $a^{(s)}$:
\begin{align}
&-\eps\frac{d}{dt}  \langle   \nabla_x \Delta_x^{-1} a^{(s)}, b^{(s)} \rangle_{L^2_x} +\frac{1}{C} \| a\|_{H^s}^2 \label{eq:asBdMain}\\
&\quad  \lesssim  \eps^2 \langle  \|F_+\|_{\mathfrak D}  \rangle^2 +  \|a\|_{L^2_x}^2  + \|(b,c)\|_{H^s_x}^2 + \|\mathcal P^\perp_\gamma f^{(s)}\|_{L^2_x(\mathcal H_\sigma)_\xi }
\end{align}
To show this, we have
\begin{align}
&-\eps\frac{d}{dt}  \langle   \nabla_x \Delta_x^{-1} a^{(s)}, b^{(s)} \rangle_{L^2_x} + \gamma^{-\frac{1}{2}}\| a^{(s)}\|_{L^2_x}^2\\
 &\quad = -\langle   \nabla_x \Delta_x^{-1} a^{(s)}, \eps \partial_t b^{(s)}  + \gamma^{-\frac{1}{2}} \nabla_xa^{(s)} \rangle_{L^2_x} \label{eq:asBdT1}\\
&\quad \quad -\eps \langle \nabla_x \Delta_x^{-1}\partial_t  a^{(s)}, b^{(s)} \rangle_{L^2_x}.  \label{eq:asBdT2}
\end{align}
From \eqref{eq:momConsPert}, we have
\begin{align}
&\| \eps \partial_t b_j^{(s)}  + \gamma^{-\frac{1}{2}} \partial_{x_j} a^{(s)} \|_{H^{-1}_x}\\
& = \Big\| \eps \frac{\dot \gamma}{2\gamma} b_j + \gamma^{\frac{1}{2}}E _j a -4\pi \gamma^{\frac{1}{2}}e^{\gamma \psi} \partial_{x_j}\Delta_x^{-1}a + \sqrt{\frac{2}{3}}  \partial_{x_j}c  + \partial_{x_k}  \langle \mathfrak h_{jk},\mathcal P^\perp_\gamma f \rangle_{L^2_\xi }   \\
&\quad + \langle \mathfrak h_j,\mathcal M_{\gamma, F_+} f\rangle_{L^2_\xi }  +e^{\gamma \psi} \langle  \mathfrak h_j, \mathcal M_{\gamma, F_+}\mathfrak h\rangle_{L^2_\xi } \Big\|_{ H^{s-1}_x}\\
&\lesssim \eps  |\dot \gamma| \|b\|_{H^{s-1}_x} +\langle \|E \|_{H^{s-1}_x}\rangle \|a\|_{H^{s-1}_x} + \|c\|_{H^s} + \|\mathcal P_\gamma^\perp f\|_{L^2_x(\mathcal H_\sigma)_\xi } \\
&\quad +  \|\langle \mathfrak h_j,\mathcal M_{\gamma, F_+} f\rangle_{L^2_{\xi }}\|_{H^{s-1}_x} +\| e^{\gamma \psi} \langle  \mathfrak h_j, \mathcal M_{\gamma, F_+}\mathfrak h\rangle_{L^2_\xi } \|_{ H^{s-1}_x}  \\
&\lesssim \eps \langle \|f\|_{L^2_x(\mathcal H_\sigma)_\xi }\rangle\|b\|_{H^{s-1}_x} +  \|a\|_{H^{s-1}_x}+ \|c\|_{H^s_x} + \|\mathcal P_\gamma^\perp f^{(s-1)}\|_{L^2_x(\mathcal H_\sigma)_\xi } \\
&\quad +  \|f^{(s-1)}\|_{L^2_x(\mathcal H_{\sigma;\eps}^-)_\xi }+ \eps\langle  \|F_+\|_{\mathfrak D}  \rangle\langle \|f^{(s-1)}\|_{L^2_x(\mathcal H_\sigma)_\xi }\rangle
\end{align}
In the final line, we use the bootstrap assumptions, plus \eqref{eq:dtBigPhiandBetaBd} and \eqref{eq:M_qGupper} as before, combined with algebra estimates for $H^{s-1}$. 
Hence, with the bootstrap assumptions, Young's inequality and interpolation, for any $\lambda >0$ we have
\begin{align}
|\eqref{eq:asBdT1}| &\leq C(\eps + \varsigma + \lambda) \|a^{(s)}\|_{L^2_x}^2\\
&\quad  + C_{\lambda} \{\eps^2\langle  \|F_+\|_{\mathfrak D}  \rangle +  \|a\|_{L^2_x}^2 + \|(b,c)\|_{H^s_x}+ \|P_\gamma^\perp f^{(s)}\|_{L^2_x(\mathcal H_\sigma)_\xi } +\| f^{(s)}\|_{L^2_x(\mathcal H_{\sigma;\eps}^-)_\xi }\}
\end{align}
Next, from \eqref{eq:dtBigPhiandBetaBd} and \eqref{eq:massConsPert}, 
\begin{align}
|\eqref{eq:asBdT2}|&\lesssim \|b^{(s)}\|_{L^2_x}^2 + \eps \|\partial_t (e^{\gamma \psi})\|_{H^{s-1}_x}\|b^{(s)}\|_{L^2_x}\\
& \lesssim (\eps\langle \|f^{(s)}\|_{L^2_x(\mathcal H_\sigma)_\xi }\rangle + \|b^{(s)}\|_{L^2_x})\|b^{(s)}\|_{L^2_x} \\
&\lesssim \eps^2\|a\|_{H^2_x} + \|(b,c)\|_{H^s_x}^2+ \|P_\gamma^\perp f^{(s)}\|_{L^2_x(\mathcal H_\sigma)_\xi } .
\end{align}
Taking $\eps,\varsigma$ and $\lambda$ sufficiently small, we conclude \eqref{eq:asBdMain}.\\

\noindent \textbf{Step 4  (estimate on $(e^{\gamma \psi} c)^\times$):} 
We now show the following bound,
\begin{align}
&-\eps\frac{d}{dt}  \langle  e^{-\gamma \psi} \nabla_x \Delta_x^{-1} (e^{-\gamma \psi} c)^\times, d \rangle_{L^2_x} +\frac{1}{C}\|(e^{-\gamma \psi} c)^\times\|_{L^2_x}^2 \label{eq:cxBdMain} \\
&\quad \leq C\{\eps^2 \langle  \|F_+\|_{\mathfrak D}  \rangle^2+ (\eps + \varsigma)( \|a\|_{L^2_x}^2 + |\overline c|^2)\\
&\quad \quad  + \|b\|_{L^2_x}^2 +  \|\mathcal P_\gamma^\perp f\|_{L^2_x(\mathcal H_\sigma)_\xi }^2+  \| f\|^2_{L^2_x (\mathcal H_{\sigma;\eps}^-)_\xi } \}.
\end{align}
By writing
\begin{align}
(\gamma^{-\frac{1}{2}} \nabla_x+\gamma^{\frac{1}{2}} E )c = \gamma^{-\frac{1}{2}}e^{\gamma \psi} \nabla_x (e^{-\gamma \psi} c) -4\pi  \gamma^{\frac{1}{2}}  c \nabla_x \Delta_x^{-1}a,
\end{align}
we have
\begin{align}
&-\eps\frac{d}{dt}  \langle  e^{-\gamma \psi} \nabla_x \Delta_x^{-1} (e^{-\gamma \psi} c)^\times, d \rangle_{L^2_x} +\frac{3\sqrt{3}}{5 \sqrt{\gamma}}\|(e^{-\gamma \psi} c)^\times\|_{L^2_x}^2 \\
 &\quad = -\langle  e^{-\gamma \psi} \nabla_x \Delta_x^{-1} (e^{-\gamma \psi} c)^\times, \eps \partial_t d + \frac{3\sqrt{3}}{5 }(\gamma^{-\frac{1}{2}} \nabla_x+\gamma^{\frac{1}{2}} E )c \rangle_{L^2_x} \label{eq:cBdT1}\\
&\quad \quad -\eps \langle e^{-\gamma \psi} \nabla_x \Delta_x^{-1} (e^{-\gamma \psi} \partial_t c)^\times, d \rangle_{L^2_x}  \label{eq:cBdT2}\\
&\quad \quad -\eps \langle \partial_t( e^{-\gamma \psi})  \nabla_x \Delta_x^{-1}  (e^{-\gamma \psi} c)^\times, d\rangle_{L^2_x} \label{eq:cBdT3}\\
&\quad \quad -\eps \langle e^{-\gamma \psi}  \nabla_x \Delta_x^{-1}  ( \partial_t(e^{-\gamma \psi}) c)^\times, d\rangle_{L^2_x} \label{eq:cBdT4}\\
&\quad \quad -4\pi \frac{3\sqrt{3}}{5} \gamma^{\frac{1}{2}}  \langle  e^{-\gamma \psi} \nabla_x \Delta_x^{-1} (e^{-\gamma \psi} c)^\times, c \nabla_x \Delta_x^{-1} a\rangle_{L^2_x} \label{eq:cBdT5}.
\end{align}
First, from \eqref{eq:3rdMomConsPert}, we have 
\begin{align}
&\| \eps \partial_t d + \frac{3\sqrt{3}}{5 }(\gamma^{-\frac{1}{2}} \nabla_x+\gamma^{\frac{1}{2}} E )c\|_{H^{-1}_x}\label{eq:dtDetc}\\
&\quad \lesssim   \eps |\dot \gamma |(\|b\|_{H^{-1}_x} + \|d\|_{H^{-1}_x}) \label{eq:dtDetcT1}\\
& \quad\quad +\sup_{j,k \in \{1,2,3\}} \{\|(\gamma^{-\frac{1}{2}} \partial_{x_k}  + \gamma^{\frac{1}{2}} E _{k}) \langle \mathfrak h_{jk}, \mathcal P_\gamma^\perp f\rangle_{L^2_\xi } \|_{H^{-1}_x} \label{eq:dtDetcT2} \\
&\quad\quad\quad  +\|\partial_{x_k} \langle   \mathfrak h_{jkll}, f\rangle_{L^2_\xi }\|_{H^{-1}_x}  \label{eq:dtDetcT3}\\
& \quad\quad\quad + \|\langle \mathcal L_\gamma\mathfrak h_{jll}, f\rangle_{L^2_\xi }\|_{H^{-1}_x}  \label{eq:dtDetcT4}\\
&\quad \quad\quad +\|\langle \mathfrak h_{jll}, \mathcal M_{\gamma, F_+} f\rangle_{L^2_\xi }\|_{H^{-1}_x}  \label{eq:dtDetcT5}\\
&\quad \quad\quad+\|\langle\mathfrak h_{jll},  e^{\gamma \psi} \mathcal M_{\gamma,F_+}\mathfrak h \rangle\|_{H^{-1}_x}  \label{eq:dtDetcT6}\\
&\quad \quad\quad +\|\langle\mathfrak h_{jll},  \Gamma_\gamma (f,f) \rangle\|_{H^{-1}_x} \} \label{eq:dtDetcT7}
\end{align}
Going line by line, we see that
\begin{align}
\eqref{eq:dtDetcT1} &\lesssim \eps (\|b\|_{L^2_x} +  \|\mathcal P_\gamma^\perp f\|_{L^2_x (\mathcal H_\sigma)_\xi })
\end{align}
Next up, we have
\begin{align}
\eqref{eq:dtDetcT2} + \eqref{eq:dtDetcT3} + \eqref{eq:dtDetcT4}\leq  \|\mathcal P_\gamma^\perp f\|_{L^2_x(\mathcal H_{\sigma})_\xi }.
\end{align}
Next, by \eqref{eq:M_qGupper}, we have
\begin{align}
\eqref{eq:dtDetcT5} &\lesssim \|f\|_{L^2_x (\mathcal H_{\sigma;\eps}^-)_\xi }+ \eps \langle \|F_+\|_{\mathfrak D}  \rangle\|f\|_{L^2_x(\mathcal H_\sigma)_\xi }\\
\eqref{eq:dtDetcT6}&\lesssim \eps \langle  \|F_+\|_{\mathfrak D}  \rangle
\end{align}
Finally, from \eqref{eq:Gammaq}, we have
\begin{align}
\eqref{eq:dtDetcT7} \lesssim \varsigma\|f\|_{L^2_x(\mathcal H_\sigma)_\xi }
\end{align}
Hence, 
\begin{align}
|\eqref{eq:cBdT1}| &\lesssim \|(ce^{\gamma \psi})^\times\|_{L^2_x} \{ \|b\|_{L^2_x} +\|\mathcal P_\gamma^\perp f\|_{L^2_x(\mathcal H_\sigma)_\xi }+\|f\|_{L^2_x (\mathcal H_{\sigma;\eps}^-)_\xi }\\
&\quad + \eps \langle  \|F_+\|_{\mathfrak D}  \rangle \langle\|f\|_{L^2_x(\mathcal H_\sigma)_\xi }\rangle+\varsigma \|f\|_{L^2_x(\mathcal H_\sigma)_\xi }\}.
\end{align}
Using \eqref{eq:fSplit}, we deduce that for  any $\lambda > 0$, 
\begin{align}
|\eqref{eq:cBdT1}| &\leq C(\eps + \varsigma +\lambda )\|(e^{\gamma \psi} c)^\times\|_{L^2_x}^2 \\
&\quad + C_{\lambda}\{\eps^2 \langle \|F_+\|_{\mathfrak D}  \rangle^2+ (\eps + \varsigma) ( \|a\|_{L^2_x}^2  + |\overline c|^2)\\
&\quad + \|b\|_{L^2_x}^2 +  \|\mathcal P_\gamma^\perp f\|_{L^2_x(\mathcal H_\sigma)_\xi }^2+  \| f\|^2_{L^2_x (\mathcal H_{\sigma;\eps}^-)_\xi }\}
\end{align}
Next, we have the term \eqref{eq:cBdT2}. From \eqref{eq:energyConsPert}, and \eqref{eq:dtBigPhiandBetaBd},
\begin{align}
\eps \|\partial_t c\|_{H^{-1}_x} &\lesssim \eps |\dot \gamma| \langle\|(a,c)\|_{L^2_x}\rangle + \|(b,d)\|_{L^2_x}  + \eps \langle  \|F_+\|_{\mathfrak D}  \rangle \langle \|f\|_{L^2_x(\mathcal H_\sigma)_\xi }\rangle \\
&\lesssim \eps \langle \|F_+\|_{\mathfrak D}  \rangle\langle \| f\|_{L^2_x(\mathcal H_\sigma)_\xi }\rangle + \eps \|(a,c)\|_{L^2_x} +\|b\|_{L^2_x} +  \|\mathcal P_\gamma^\perp f\|_{L^2_x(\mathcal H_\sigma)_\xi } +  \| f\|_{L^2_x (\mathcal H_{\sigma;\eps}^-)_\xi }
\end{align}
Thus
\begin{align}
|\eqref{eq:cBdT2}|&\lesssim  \eps^2 \langle  \|F_+\|_{\mathfrak D}  \rangle^2 + \eps \|(a, c)\|_{L^2_x}^2+ \|b\|_{L^2_x}^2+  \|\mathcal P_\gamma^\perp f\|_{L^2_x(\mathcal H_\sigma)_\xi }^2 +  \| f\|_{L^2_x (\mathcal H_{\sigma;\eps}^-)_\xi }^2 
\end{align}
Next, 
\begin{align}
|\eqref{eq:cBdT3}| +|\eqref{eq:cBdT4}| &\lesssim  \eps  \|c\|_{L^2_x} \|d\|_{L^2_x} \lesssim  \eps^2 +    \|\mathcal P_\gamma^\perp f\|_{L^2_x(\mathcal H_\sigma)_\xi }^2.
\end{align}
Finally,
\begin{align}
|\eqref{eq:cBdT5}|&\lesssim  \|(e^{-\gamma\psi} c)^\times\|_{L^2_x} (\|c\nabla_x \Delta_x^{-1}a\|_{L^2_x} + \|c \nabla_x \Delta_x^{-1}(n_+^\eps - n_+^0)\|_{L^2_x})\\
&\lesssim \|a\|_{H^s_x} \|(e^{-\gamma\psi} c)^\times\|_{L^2_x} (\|c\|_{L^2_x}+ \|F^\eps_+ - F^0_+\|_{\mathfrak E'}) \\
&\lesssim  \varsigma \|(e^{-\gamma\psi} c)^\times\|_{L^2_x}(\|c\|_{L^2_x}+ \|F^\eps_+ - F^0_+\|_{\mathfrak E'})
\end{align}

Combining these bounds for \eqref{eq:cBdT1} through \eqref{eq:cBdT5}, we conclude that 
\begin{align}
&-\eps\frac{d}{dt}  \langle  e^{-\gamma \psi} \nabla_x \Delta_x^{-1} (e^{-\gamma \psi} c)^\times, d \rangle_{L^2_x} +\frac{3\sqrt{3}}{5 \sqrt{\gamma}}\|(e^{-\gamma \psi} c)^\times\|_{L^2_x}^2\\
 &\quad \leq C(\eps + \varsigma+\lambda )\|(e^{\gamma \psi} c)^\times\|_{L^2_x}^2 \\
&\quad\quad + C_{\lambda}\{\eps^2 \langle \|F_+\|_{\mathfrak D}  \rangle^2+ (\eps + \varsigma) (\|a\|_{L^2_x}^2 + |\overline c|^2) \\
&\quad \quad \quad+ \|b\|_{L^2_x}^2 +  \|\mathcal P_\gamma^\perp f\|_{L^2_x(\mathcal H_\sigma)_\xi }^2+  \| f\|^2_{L^2_x (\mathcal H_{\sigma;\eps}^-)_\xi } \}
\end{align}
By taking $\eps,\varsigma$ and $\lambda$ sufficiently small, we conclude \eqref{eq:cxBdMain}.\\

\noindent \textbf{Step 5: Estimate on $c^{(s)}$:}
Following the same method as previous bounds, we have the estimate
\begin{align}
&-\eps\frac{d}{dt}  \langle   \nabla_x \Delta_x^{-1} c^{(s)}, d^{(s)} \rangle_{L^2_x} +\frac{1}{C}\|c\|_{H^s_x}^2 \label{eq:csBdMain} \\
&\quad \leq C\{\eps^2 \langle  \|F_+\|_{\mathfrak D}  \rangle^2+ (\eps + \varsigma) \|a\|_{L^2_x}^2\\
&\quad \quad  + \|c\|_{L^2_x}^2 +  \|b\|_{H^s_x}^2 +  \|\mathcal P_\gamma^\perp f^{(s)}\|_{L^2_x(\mathcal H_\sigma)_\xi }^2+  \|f^{(s)}\|^2_{L^2_x (\mathcal H_{\sigma;\eps}^-)_\xi }\}
\end{align}
The proof is similar to that of \eqref{eq:asBdMain}. \\

\noindent \textbf{Step 6: Combining the bounds:}
Combining the upper bounds on \eqref{eq:aBdMain}, \eqref{eq:asBdMain}, \eqref{eq:bBdMain}, \eqref{eq:cxBdMain},  \eqref{eq:csBdMain} there is a choice of constant $\kappa >0$ taken sufficiently small, such that the functional
\begin{align}
\mathscr G_{reg}(t)&:= \kappa \langle e^{-\gamma \psi} \nabla_x \Delta_x^{-1} a,b\rangle_{L^2_x} - \kappa^3 \langle \nabla_x\Delta_x^{-1}a^{(s)},b^{(s)}\rangle_{L^2_x}\\
&\quad  -  \langle e^{-\gamma \psi} \nabla_x \Delta_x^{-1} (e^{- \gamma \psi} c)^\times,d\rangle_{L^2_x} - \kappa^2 \langle \nabla_x \Delta_x^{-1} c^{(s)},d^{(s)}\rangle_{L^2_x}
\end{align}
satisfies
\begin{align}
&\eps \frac{d}{dt} \mathscr G_{reg}(t) + \frac{1}{C}\|\mathcal P_\gamma f^{(s)}\|_{L^2_{x}(\mathcal H_\sigma)_\xi }^2 \\
&\quad \leq C\{ \eps^2 \langle  \|F_+\|_{\mathfrak D}  \rangle^2+ (\eps + \varsigma) \|\mathcal P_\gamma f^{(s)}\|_{L^2_{x,\xi }}^2\\
&\quad \quad  +  \|\mathcal P_\gamma^\perp f^{(s)}\|_{L^2_x(\mathcal H_\sigma)_\xi }^2+  \|f^{(s)}\|^2_{L^2_x (\mathcal H_{\sigma;\eps}^-)_\xi } + \|F^\eps_+ - F^0_+\|_{\mathfrak E'}^2\}
\end{align}
On the other hand by \eqref{eq:projClose},  for each $\alpha\in \{1,2\}$,
\begin{align}
\|\mathcal P_\gamma^\perp f^{(s)}\|_{L^2_x(\mathcal H_\sigma)_\xi }^2+  \|f\|^2_{L^2_x (\mathcal H_{\sigma;\eps}^-)_\xi } &\lesssim \|\mathcal P_\gamma^\perp g_\alpha ^{(s)}\|_{L^2_x(\mathcal H_\sigma)_\xi }^2+  \|g_\alpha ^{(s)}\|^2_{L^2_x (\mathcal H_{\sigma;\eps}^-)_\xi }\\
&\quad + \eta \|\mathcal P_\gamma f^{(s)}\|_{L^2_x(\mathcal H_\sigma)_\xi }^2.
\end{align}
Thus, by taking $\eps, \varsigma$ and $\eta$ sufficiently small, and re-scaling $\mathscr G$ if necessary, we conclude with \eqref{eq:macroBd1}. To get \eqref{eq:macroBd0}, we take $\kappa >0$ and
\begin{align}
 \mathscr G(t) :=- \kappa \langle e^{-\gamma \psi} \nabla_x \Delta_x^{-1} a,b\rangle_{L^2_x} -  \langle e^{-\gamma \psi} \nabla_x \Delta_x^{-1} (e^{- \gamma \psi} c)^\times,d\rangle_{L^2_x}
\end{align}
so that
\begin{align}
\eps \frac{d}{dt} \mathscr G(t) + \frac{1}{C}\|\mathcal P_\gamma f\|_{L^2_{x}(\mathcal H_\sigma)_\xi }^2 &\leq C\{ \eps^2 \langle \|F_+\|_{\mathfrak D}  \rangle^2+ (\eps + \varsigma) \|\mathcal P_\gamma f\|_{L^2_{x,\xi }}^2\\
&\quad \quad  +  \|\mathcal P_\gamma^\perp f\|_{L^2_x(\mathcal H_\sigma)_\xi }^2+  \|f\|^2_{L^2_x (\mathcal H_{\sigma;\eps}^-)_\xi } \}.
\end{align}
And, once again,  by \eqref{eq:projClose},  for each $\alpha\in \{1,2,3\}$,
\begin{align}
\|\mathcal P_\gamma^\perp f\|_{L^2_x(\mathcal H_\sigma)_\xi }^2+  \|f\|^2_{L^2_x (\mathcal H_{\sigma;\eps}^-)_\xi } \lesssim \|\mathcal P_\gamma^\perp g_\alpha \|_{L^2_x(\mathcal H_\sigma)_\xi }^2+  \|g_\alpha \|^2_{L^2_x (\mathcal H_{\sigma;\eps}^-)_\xi } + \eta \|\mathcal P_\gamma f\|_{L^2_x(\mathcal H_\sigma)_\xi }^2 
\end{align}
Again, taking $\eps, \varsigma$ and $\eta$  sufficiently small, and re-scaling $\mathscr G$ if necessary, we deduce \eqref{eq:macroBd0}.
\end{proof}

We conclude this section with a corollary of Propositions \ref{prop:energyEstimates} and  \ref{prop:macroscopicEstimates}:
\begin{corollary}\label{cor:electronError}
Let $M \geq 1$, and let $\eta$ be sufficiently small depending on $M$. We let $ \varsigma$ and $\eps$ be sufficiently small depending on $M$ and $\eta$. Assume the bootstrap assumptions \eqref{eq:bootstrap1} through \eqref{eq:bootstrap4}.   Then, for each $M,\eta>0$, and $\alpha\in\{1,2,3\}$ there exists a function $\mathscr Y_{-,\alpha}  = \mathscr Y_{-,\alpha}^{(M,\eta)}:[0,T] \to \mathbf R_+$ such that
\begin{align}\label{eq:Y_jEquiv}
\mathscr Y_{-,\alpha} \sim_{M,\eta}  \widetilde {\mathscr E}_{-,\alpha}
\end{align}
and
\begin{align}\label{eq:Y_jBd}
&\eps \frac{d}{dt} \mathscr Y_{-,\alpha} +  \widetilde {\mathscr D}_{-,\alpha} \lesssim_{M,\eta}   \eps^2.
\end{align}
\end{corollary}

\begin{proof}
We take $\kappa_0,\kappa_1, \kappa_2$ and $\kappa_3 \in (0,1)$ be constants to be determined.
We combine \eqref{eq:gjEnergy} in the case $\alpha=2$ and \eqref{eq:macroBd0} as follows:
\begin{align}
&\eps\frac{d}{dt} \{\| g_\alpha \|_{L^2_{x,\xi }}^2 + 4\pi \sqrt{\gamma}\||\Delta_x|^{-1} a\|_{L^2_x}^2+ \kappa_1\mathscr G\}\\
&\quad + (\frac{1}{C} - C\kappa_1)\{ \|\mathcal P_\gamma^\perp g_\alpha \|_{L^2_x(\mathcal H_\sigma)_\xi }^2  + \|g_\alpha \|^2_{L^2_x(\mathcal H_{\sigma;\eps}^-)}\} +\kappa_0 \|\mathcal P_\gamma f\|_{L^2_x(\mathcal H_\sigma)_\xi }^2\\
&\quad  \leq  C( \eps +  \varsigma + \eta +  \kappa_0) \|\mathcal P_\gamma f\|_{L^2_x(\mathcal H_\sigma)_\xi }^2 \\
&\quad \quad  +C  \eps \langle  \|\partial_t a\|_{L^2_x}  + \|F_+\|_{\mathfrak D}  ^2\rangle\|g_\alpha \|_{L^2_{x,\xi }}^2\\
&\quad \quad +  C_{\kappa_0}\eps^2 \langle  \|F_+\|_{\mathfrak D}  \rangle^2.
\end{align}
Using \eqref{eq:projSplit}, we first fix $\kappa_1$ sufficiently small such that the  ``$\frac{1}{C} - C\kappa_1$" in the above is strictly positive, and 
\begin{align}
\mathscr X_{-,\alpha} := \| g_\alpha \|_{L^2_{x,\xi }}^2 + 4\pi \sqrt{\gamma}\||\Delta_x|^{-1} a\|_{L^2_x}^2+ \kappa_1\mathscr G
\end{align}
satisfies
\begin{align}
\mathscr X_{-,\alpha} \sim \|g_\alpha \|_{L^2_{x,\xi }}^2.
\end{align}
Then, we fix $\kappa_0$ to be a sufficiently small constant, and also require that $\eps,\varsigma$ and $\eta$ are sufficiently small so that 
\begin{align}
&\eps \frac{d}{dt} \mathscr X_{-,\alpha}  + \frac{1}{C} \|g_\alpha \|_{L^2_x(\mathcal H_\sigma \cap \mathcal H_{\sigma;\eps}^-)_\xi }^2\\
  &\quad \leq C\{  \eps \langle  \|\partial_t a\|_{L^2_x}  + \|F_+\|_{\mathfrak D}  ^2\rangle\mathscr X_{-,\alpha}  +  \eps^2 \langle  \|F_+\|_{\mathfrak D}  \rangle^2\}.
\end{align}
Next, we combine this bound with \eqref{eq:g1sEnergy},
\begin{align}
&\eps \frac{d}{dt} ( \| g_\alpha ^{(s)}\|_{L^2_{x,\xi }}^2 + \kappa_3 \mathscr G_{reg})\\
&\quad + (\frac{1}{C} - C\kappa_3)\{ \| \mathcal P_\gamma^\perp g_\alpha ^{(s)}\|_{L^2_x(\mathcal H_\sigma)_\xi }^2 + \|g_\alpha ^{(s)}\|^2_{L^2_x(\mathcal H_{\sigma;\eps}^-)}\} + \kappa_3 \|\mathcal P_\gamma f^{(s)}\|_{L^2_x(\mathcal H_\sigma)_\xi }^2\\
&\quad  \leq  C( \eps +  \varsigma + \eta +  \kappa_2)  \|\mathcal P_\gamma f^{(s)}\|_{L^2_x(\mathcal H_\sigma)_\xi }^2 + C_{\eta,\kappa_2} \|g_{\alpha+1}\|_{L^2_x(\mathcal H_\sigma\cap \mathcal H_{\sigma;\eps}^-)_\xi }^2 \\
&\quad \quad  +C  \eps \langle  \|\partial_t a\|_{H^{s-\frac{1}{4}}_x}  + \|F_+\|_{\mathfrak D}  ^2\rangle\| g_\alpha ^{(s)}\|_{L^2_{x,\xi }}^2\\
&\quad \quad +   C_{\kappa_2}\eps^2 \langle  \|F_+\|_{\mathfrak D}  \rangle^2
\end{align}
We now fix $\kappa_3$ small enough so that the ``$\frac{1}{C} - C\kappa_3$" in the above is positive, and 
\begin{align}
\| g_\alpha ^{(s)}\|_{L^2_{x,\xi }}^2 + \kappa_3 \mathscr G_{reg} \sim \| g_\alpha ^{(s)}\|_{L^2_{x,\xi }}^2,
\end{align}
and fix $\kappa_2$, and take $\eps,\varsigma$ and $\eta$ small enough that
\begin{align}
&\eps \frac{d}{dt} ( \| g_\alpha ^{(s)}\|_{L^2_{x,\xi }}^2 + \kappa_3 \mathscr G_{reg}) + \frac{1}{C} \|  g_\alpha ^{(s)}\|_{L^2_x(\mathcal H_\sigma\cap \mathcal H_{\sigma;\eps}^-)_\xi }^2 \\
&\quad  \leq  C_{\eta} \|g_{\alpha+1}\|_{L^2_x(\mathcal H_\sigma\cap \mathcal H_{\sigma;\eps}^-)_\xi }^2 \\
&\quad \quad  +C\{  \eps \langle  \|\partial_t a\|_{L^2_x}  +  \|F_+\|_{\mathfrak D}  ^2\rangle\| g_\alpha ^{(s)}\|_{L^2_{x,\xi }}^2  +\eps^2 \langle \|F_+\|_{\mathfrak D}  \rangle^2\}.
\end{align}
FInally, we take $\kappa_4^{(\eta)}$ sufficiently small depending on $\eta$ so that
\begin{align}
\mathscr X_{-,\alpha,reg}^{(\eta)} := \kappa_4^{(\eta)}(\| g_\alpha ^{(s)}\|_{L^2_{x,\xi }}^2 + \kappa_3 \mathscr G_{reg}) + \mathscr X_{-,\alpha+1}
\end{align}
satisfies
\begin{align}
&\frac{d}{dt} \mathscr X_{-,reg,\alpha}^{(\eta)} + \frac{1}{C_\eta} \widetilde{\mathscr D}_{-,\alpha} \\
&\quad \leq C_\eta \{  \eps \langle  \|\partial_t a\|_{L^2_x}  +  \|F_+\|_{\mathfrak D}  ^2\rangle \mathscr X_{-,reg,\alpha}^{(\eta)} + \eps^2 \langle \|F_+\|_{\mathfrak D}  \rangle^2\}.
\end{align}
By \eqref{eq:F_+apriori}, we have 
\begin{equation}\label{eq:F_+apriori2.0}
\frac{1}{C} \langle \|F_+\|_{\mathfrak D}  \rangle^2 \leq -\frac{d}{dt}\mathscr  X_++ C( 1+  \widetilde{\mathscr D}_{-,\alpha}).
\end{equation}
Taking $\eps$ and $\varsigma$ sufficiently small depending on $\eta$ now, we can hide the $\|g_\alpha ^{(s)}\|^2_{L^2_x(\mathcal H_\sigma)_\xi }$ term and get for some collection of constants $C_1^{(\eta)},C_2^{(\eta)}$ and $C_3^{(\eta)}$, 
\begin{align}
&\eps (\frac{d}{dt} + C_1^{(\eta)} \frac{d}{dt}\mathscr X_{+} - C_2^{(\eta)}\langle  \|\partial_t a\|_{L^2_x}\rangle) \mathscr X_{-,reg,\alpha}^{(\eta)} + \frac{1}{C_3^{(\eta)}}  \widetilde{\mathscr D}_{-,\alpha} \\
&\quad \leq C_3^{(\eta)}\eps^2.
\end{align}
We now define
\begin{align}
\Omega^{(\eta)}(t):= C_1^{(\eta)}\mathscr X_{+}(t) -  C_2^{(\eta)}\int_0^t \langle  \|\partial_t a(t')\|_{L^2_x}\rangle dt'
\end{align}
By the \eqref{eq:bootstrap1} and \eqref{eq:bootstrap4}, we have that $|\Omega^{(\eta)}|$ is less than or equal to some constant depending on $\eta$, say $C_4^{(\eta)} >0$.
Thus, if we set $\mathscr Y_j^{(\eta)}= C_3^{(\eta)} e^{\Omega + C_4^{(\eta)}}\mathscr X_{-,reg,\alpha}^{(\eta)}$, we have a function which satisfies \eqref{eq:Y_jEquiv} and \eqref{eq:Y_jBd}.
\end{proof}

\subsection{Proof of Proposition  \ref{prop:derivationRefined}}

We now conclude this section  with a proof of Proposition \ref{prop:derivationRefined}.\\

\noindent\textit{Proof of Proposition \ref{prop:derivationRefined}:} 
Throughout this proof, we take $\eta$ sufficiently small so as to satisfy the hypothesis of Corollary \ref{cor:electronError}. Any of the constants appearing in the estimates will implicitly depend on the constants $ M$, $\eta>0$ i.e. $C = C_{M,\eta}$.


For each $\alpha\in\{1,2\}$, we integrate the bound in Corollary \ref{cor:electronError} on $t \in[0,T]$ and use   \eqref{eq:Y_jEquiv} to get
\begin{align}\label{eq:gjIntegratedBd}
 &\eps \sup_{t \in [0,\hat T]}  \widetilde {\mathscr E}_{-,\alpha}(t) +\int_0^{\hat T} \widetilde {\mathscr D}_{-,\alpha}(t)dt\\
 &\quad \leq C(\eps^2 + \eps  \widetilde {\mathscr E}_{-,\alpha}(0)).
\end{align}
Taking $j = 2$, this implies \eqref{eq:stability1}.

We now show \eqref{eq:stability2}. 
 For this, we introduce $\theta>0$, to be chosen latter. Now, for any $h =h(\xi)$, observe that
\begin{align}
\|h\|_{\mathcal H_\sigma}^2& \gtrsim \int_{\mathbf R^3} \frac{1}{\langle \xi\rangle} h(\xi)^2d\xi  \\
&\geq \frac{1}{\theta t^{\frac{1}{3}}} \int_{\langle \xi\rangle\leq \theta t^{\frac{1}{3}}} h(\xi)^2 d\xi  \\
&\geq \frac{1}{\theta t^{\frac{1}{3}}} \left( \|h\|_{L^2}^2 - \int_{\langle \xi\rangle\geq \theta t^{\frac{1}{3}}} h(\xi)^2 d\xi \right)\\
&\geq \frac{1}{\theta t^{\frac{1}{3}}} \left( \|h\|_{L^2}^2 - \int_{\langle \xi\rangle\geq  \theta t^{\frac{1}{3}}}e^{\frac{1}{4}(e^\eta - 1)\beta_{in}(|\xi|^2 +1 - \theta^2 t^{\frac{2}{3}})} h(\xi)^2 d\xi \right) \\
&\geq  \frac{1}{\theta t^{\frac{1}{3}}} \left( \|h\|_{L^2}^2 - e^{-\frac{1}{4}(e^\eta - 1)\beta_{in} \theta^2 t^{\frac{2}{3}})} \|e^{\frac{1}{4}(e^\eta - 1)\beta_{in}(|\xi|^2 +1 )}h\|_{L^2}^2 \right)
\end{align}
Applying this to $h = g_1^{(s)}$ and $g_2$ and integrating in $x$, and re-scaling $\theta$, we have
\begin{align}
 \widetilde {\mathscr D}_{-,1} \gtrsim  \frac{1}{\theta t^{\frac{1}{3}}} \left(\mathscr Y_{-,1}  -e^{-\theta^2 t^{\frac{2}{3}}} \mathscr Y_{-,2}\right)
\end{align}
Thus, using \eqref{eq:bootstrap2} with \eqref{eq:Y_jBd}, we have 
\begin{align}
\eps \frac{d}{dt}\mathscr Y_{-,1} + \frac{2}{3C\theta t^{\frac{1}{3}}} \mathscr Y_{-,1} \leq C( \eps^2 +  \frac{2}{3\theta t^{\frac{1}{3}}} e^{-\theta^2 t^{\frac{2}{3}}}\sup_{t'\in[0,T]}{ \widetilde {\mathscr E}_{-,2}(t')} .
\end{align}
and so
\begin{align}
\frac{d}{dt} ( e^{\frac{1}{C\theta \eps} t^{\frac{2}{3}}}\mathscr Y_{-,1} ) \leq C(\eps e^{\frac{1}{C \theta \eps} t^{\frac{2}{3}}} + \frac{1}{\eps \theta t^{\frac{1}{3}}} e^{(\frac{1}{C\theta \eps } -\theta^2)t^{\frac{2}{3}}}\sup_{t'\in[0,T]}{ \widetilde {\mathscr E}_{-,2}(t')}.
\end{align}
Hence, by taking $\theta = \left(\frac{2}{C\eps}\right)^{\frac{1}{3}}$ with $C$ as in the above, we get some for some $C_0, C_1,C_2$,
\begin{align}\label{eq:stretchedGronwallSetup}
\frac{d}{dt} ( e^{\frac{1}{C_0} (\frac{t}{\eps})^{\frac{2}{3}}}\mathscr Y_{-,1} ) \leq C_2(\eps e^{\frac{1}{C_0 } (\frac{t}{\eps})^{\frac{2}{3}}} +  \frac{d}{dt} e^{-\frac{1}{C_1} (\frac{t}{\eps})^{\frac{2}{3}}}\sup_{t'\in[0,T]}{ \widetilde {\mathscr E}_{-,2}(t')}
\end{align}
Now,
\begin{align}
\int_0^t e^{\frac{1}{C_0 } (\frac{t'}{\eps})^{\frac{2}{3}}} dt' \leq \eps^{\frac{2}{3}} t^{\frac{1}{3}} \int_0^t \frac{1}{\eps^\frac{2}{3} (t')^{\frac{1}{3}}}e^{\frac{1}{C_0 } (\frac{t'}{\eps})^{\frac{2}{3}}} dt'  \lesssim \eps^{\frac{2}{3}}t^{\frac{1}{3}} e^{\frac{1}{C_0} (\frac{t}{\eps})^{\frac{2}{3}}}.
\end{align}
Thus, by integrating \eqref{eq:stretchedGronwallSetup}, we get
\begin{align}
\mathscr Y_{-,1}(t) \lesssim  e^{-\frac{1}{C} (\frac{t}{\eps})^{\frac{2}{3}}}  \sup_{t'\in[0,T]}{ \widetilde {\mathscr E}_{-,2}(t')} + \eps^\frac{5}{3} t^{\frac{1}{3}}.
\end{align}
for all $ t\in[0,T]$.  Using \eqref{eq:Y_jEquiv}, we conclude \eqref{eq:stability2}.  \qed

\section{Proof of Theorem \ref{thm:derivation}}\label{sec:proof}
In this section, we conclude with the proof of the main theorem.

\noindent \textit{Proof of Theorem \ref{thm:derivation}:} We break the proof of the Theorem into the proofs of parts (i) and (ii).\\

\noindent \textbf{Proof of part (i):} We break the proof of (i) into a number of steps: the setup of the bootstrap argument, combining the estimates of the preceding results, and closing the bootstrap argument.\\

\noindent \textit{Step 1 of part (i) (setup \& bootstrap assumptions):} We take $\kappa = \kappa(M) > 0$ to be a constant to be determined later. We take $(F_+^0, \beta ,\phi^0)$ to be a weak solution to \eqref{eq:VPLRescaled} on $[0,T_0]$ satisfying the hypotheses of the theorem. In addition,  we may as well take $T_0 > 0$ small enough so that
\begin{align}\label{eq:smallDiss}
\int_0^{T_0} \|F_+^0(t)\|_{\mathfrak D}^2 dt \leq \kappa^2.
\end{align}
Now, for all $\eps$ sufficiently small, there exists a unique weak solution $(F_+^\eps, F_-^\eps)$ to  \eqref{eq:VPLRescaled} on some interval containing zero, in the sense described in Lemma \ref{lem:LWP}.
Moreover, we let $(\psi^\eps,\gamma^\eps)$ be the solution to \eqref{eq:PP} in the sense of Lemma \ref{lem:PP}, defined on some time interval containing 0. 
We then take $\eta$, $\overline \epsilon$, and $ \varsigma$ as in Proposition \ref{prop:derivationRefined}, and take $\eps  \in(0, \overline \epsilon]$. 

 By continuity in time of the quantities considered, there exists  $ \hat T_{\eps} >0$ such that the following conditions hold: 
\begin{align}
\sup_{t\in[0,\hat T_\eps]}( \|\frac{1}{n_+^\eps(t)}\|_{L^\infty_x}+ \|F_+^\eps(t)\|_{\mathfrak E} )&\leq 2M, \label{eq:bootstrap1i}\\
\sup_{t \in [0,\hat T_\eps]} \|F_+^\eps(t) - F_-^0(t)\|_{\mathfrak E'} &\leq \kappa\label{eq:bootstrap1.5i}\\
\sup_{t\in[0,\hat T_\eps]} \sqrt{  {\mathscr E}_{-,2}^\eps(t)} & \leq \kappa,\label{eq:bootstrap2i}\\
\sup_{t\in[0,\hat T_\eps]} |\ln( \frac{ \gamma^\eps(t)}{\beta_{in}}) |&\leq \eta, \label{eq:bootstrap3i}\\
\int_0^{\hat T_\eps} \|\partial_t (n_-^\eps - e^{\gamma^\eps \psi^\eps})\|_{L^2_x} dt& \leq 1. \label{eq:bootstrap4i}
\end{align}
By Lemmas \ref{lem:LWP} (taking $\kappa < \frac{\varrho(2M)}{2}$) and \ref{lem:PP} (using condition \eqref{eq:bootstrap3i}), we may take $\kappa$ small enough so that the interval of existence of $(F_+^\eps ,F_-^\eps)$ and $(\gamma^\eps,\psi^\eps)$ is strictly larger than $[0,\hat T_\eps]$. Hence, we may take $\hat T_\eps > 0$ to be the largest such time, so that either at least one of the above conditions holds with equality, or $\hat T_\eps = T_0$.

Next, we note that by condition \eqref{eq:bootstrap2i} and \eqref{eq:bootstrap3i}, we have
\begin{align}
\sup_{t \in [0,T]} \{\|\frac{1}{n_-^\eps(t)}\|_{L^\infty_{x}} + \|F_-^\eps (t)\|_{\mathfrak E}\} \leq C_M,
\end{align}
so  the consequences of Propositions \ref{prop:apriori} and \ref{prop:ionError} both hold, up to replacing $M$ with some $C_M$ in their hypotheses.\\
%
%
%
%

\noindent \textit{Step 2 of part (i) (combining estimates):} We now combine the estimates of the preceding propositions and lemmas.
We first show that the hypotheses of Proposition \ref{prop:derivationRefined} are satisfied. By \eqref{eq:gammaBetaDiff}, for each $\alpha\in \{1,2\}$, and $T \in [0,\hat T_\eps]$, we have 
\begin{align}
&\sup_{t \in [0, T]} (\|e^{\frac{q_\alpha |\xi|^2}{2}}\langle \nabla_x\rangle^s\langle \nabla_\xi \rangle ( \mu_\beta e^{\beta\phi^0}- \mu_{\gamma^\eps} e^{\gamma^\eps \psi^\eps})\|_{L^2_{x,\xi }}\\
&\quad \quad + \|e^{\frac{q_{\alpha +1}|\xi|^2}{2}}\langle \nabla_\xi \rangle ( \mu_\beta e^{\beta\phi^0}- \mu_{\gamma^\eps} e^{\gamma^\eps \psi^\eps})\|_{L^2_{x,\xi }} )\\
&\lesssim_M  \sup_{t \in [0, T]}( |\beta - \gamma^\eps| + \|\phi^0 - \psi^\eps\|_{H^s_x}) \\
&\lesssim_M  \sup_{t \in [0, T]} \|F_+^\eps(t) - F_+^0(t)\|_{\mathfrak E'}.
\end{align}
Hence, for all such $T$, and for each $\alpha\in\{1,2\}$, we have
\begin{align}
 \sup_{t \in [0, T]} |\mathscr E_{-,\alpha}^\eps(t) -\widetilde{\mathscr E}_{-,\alpha}^\eps(t)|&\leq C_M\sup_{t \in [0, T]}  \|F_+^\eps(t) - F_+^0(t)\|_{\mathfrak E'}^2,\label{eq:EtoEtilde},\\
 \sup_{t \in [0, T]} |\mathscr D_{-,\alpha}^\eps(t) -\widetilde{\mathscr D}_{-,\alpha}^\eps(t)|&\leq C_M\sup_{t \in [0, T]}  \|F_+^\eps(t) - F_+^0(t)\|_{\mathfrak E'}^2. \label{eq:DtoDtilde}.
\end{align}
Using \eqref{eq:EtoEtilde} with $T = 0$, we have
\begin{align}
\widetilde{\mathscr E}_{-,\alpha,in}^\eps \lesssim_M  \mathscr E_{-,\alpha,in}^\eps + C_M \eps^2 \lesssim \eps^2. 
\end{align}
By the condition $\mathscr E_{-,2,in}^\eps \leq M\eps$, we have
\begin{align}
&\left|\int_{\mathbf R^3 \times \mathbf T^3} \frac{|\xi|^2}{2} (F_{-,in}^\eps - \mu_{\gamma^\eps_{in}}e^{\beta_{in} \psi^\eps_{in}}) dxd\xi  + \frac{1}{8\pi} \int_{\mathbf T^3}  |\nabla_x \phi_{in}^\eps|^2 - |\nabla_x \psi^\eps_{in}|^2dx \right| \\
&\quad \lesssim \sqrt{\widetilde {\mathscr E}_{-,2,in}}+  \widetilde {\mathscr E}_{-,2,in} \\
&\quad \lesssim C_M (\eps +\eps^2).
\end{align}
Thus, we have that \eqref{eq:energyCondition} holds in Proposition \ref{prop:derivationRefined}. On the other hand, by  \eqref{eq:bootstrap1.5i}, \eqref{eq:bootstrap2i} and \eqref{eq:EtoEtilde}, we have $ \sup_{t \in [0, T]} \sqrt{\widetilde{\mathscr E}_{-,\alpha}^\eps(t) } \lesssim_M  \kappa$. Taking $\kappa  \leq \frac{\varsigma}{C_M}$, we have that the conclusions of Proposition \ref{prop:derivationRefined} are hold on $[0,\hat T_\eps]$.

From \eqref{eq:stability1}, we have
\begin{align}\label{eq:stability1'}
\eps \sup_{t \in [0,\hat T_\eps]}   {\mathscr E}^\eps_{-,2}(t) + \int_0^{\hat T_\eps}   {\mathscr D}^\eps_{-,2}(t) dt \leq C_M (\eps   {\mathscr E}^\eps_{-,2,in}  + \eps^2).
\end{align}
Combining the above with \eqref{eq:stability2}, we have
and for all $t \in[0,\hat T_\eps]$, we have 
\begin{align}\label{eq:stability2'}
{\mathscr E}^\eps_{-,1}(t) \leq C_M( e^{-\frac{1}{C_M}(\frac{t}{\eps})^{\frac{2}{3}}}\eps +\eps^{\frac{5}{3}} t^\frac{1}{3} + \eps^2).
\end{align}

We now prove the estimate
\begin{align}\label{eq:ionDiff_eps}
\sup_{t \in [0,\hat T_\eps]}\|F_+^\eps(t) - F_+^0(t)\|_{\mathfrak E'}^2  + \int_0^{\hat T_\eps}\|F_+^\eps(t) - F_+^0(t)\|_{\mathfrak D'} ^2dt\lesssim_M \eps^2.
\end{align}
Using the above and \eqref{eq:stability1}, then for all $T \leq \hat T_{\eps}$, we  integrate \eqref{eq:Gbd} on $t\in[0,T]$ to get
\begin{align}
&\sup_{t \in [0,T]} \|F_+^\eps(t) - F_+^0(t)\|_{\mathfrak E'}^2 + \int_0^T\|F_+^\eps(t) - F_+^0(t)\|_{\mathfrak D'}^2 dt \\
&\quad  \lesssim_M  \|F_+^\eps(0) - F_+^0(0)\|_{\mathfrak E'}^2\\
&\quad \quad + \int_0^T \langle \|(F_+^\eps,F_+^0)\|_{\mathfrak D}\rangle^2 ( \eps^2 +\|F_+^\eps(t) - F_+^0(t)\|_{\mathfrak E'}) + \mathscr D_{-,2}(t) dt\\
&\quad \lesssim_M \int_0^T \langle \|(F_+^\eps,F_+^0)\|_{\mathfrak D}\rangle^2 ( \eps^2 +\sup_{t' \in [0,t]} \|F_+^\eps(t') - F_+^0(t')\|_{\mathfrak E'}^2)  dt + \eps^2.
\end{align}
In the final line, we used \eqref{eq:stability1'}.
Using the time integrated form of Gr\"onwall's inequality, we deduce \eqref{eq:ionDiff_eps}, provided the follow bound holds true: 
\begin{align}\label{eq:dissControl}
\int_0^{\hat T_\eps}\langle \|(F_+^0(t),F_+^\eps(t))\|_{\mathfrak D}\rangle^2 dt\lesssim_M 1.
\end{align}
Indeed, by integrating \eqref{eq:F_+apriori} in Proposition \ref{prop:apriori}, and using \eqref{eq:stability1'} to control the right-hand side, we have
\begin{align}
\int_0^{\hat T_\eps}  \|F_+^\eps(t)\|_{\mathfrak D}^2 dt\lesssim_M 1.
\end{align}
Combined with \eqref{eq:smallDiss}, we have \eqref{eq:dissControl}.  \\

\noindent \textit{Step 3 of part (i) (closing the bootstrap):} We now show that $\hat T := \inf_{\eps \in (0,\overline \eps]} \hat T_\eps \geq \kappa$. To show this, we suppose $\hat T_\eps < \kappa$ to yield a contradiction.
Now, one of
\eqref{eq:bootstrap1i} through \eqref{eq:bootstrap4i} holds with equality. We now show that in each of these five cases, the condition $\hat T_\eps < \kappa$ cannot hold.

 First, suppose  
 \eqref{eq:bootstrap1i} holds with equality.  By integrating \eqref{eq:F_+aprioriAlt},  and using \eqref{eq:stability1'} to control the right-hand side, we have for all $t \in [0,\hat T_\eps]$,
 \begin{align}\label{eq:F_+epsBootstrapClose}
\|F_+^\eps(t)\|_{\mathfrak E}^2 \leq \|F_{+,in}^\eps\|_{\mathfrak E}^2 + C_M (t+ \|\langle v\rangle^{m_1} F_+^\eps\|_{L^2_x(\dot{\mathcal H}_\sigma\cap \mathcal H_{\sigma;\eps}^+)_v}^2 dt')
 \end{align}
 We claim that
 \begin{align}\label{eq:smallF_+epsDiss}
\int_0^t \|\langle v\rangle^{m_1} F_+^\eps(t')\|_{L^2_x(\dot{\mathcal H}_\sigma\cap \mathcal H_{\sigma;\eps}^+)_v}^2 dt' \lesssim_M \kappa +  \eps .
 \end{align}
 Indeed, 
  \begin{align}
 \int_0^t \|\langle v\rangle^{m_1} F_+^\eps(t')\|_{L^2_x(\dot{\mathcal H}_\sigma\cap \mathcal H_{\sigma;\eps}^+)_v}^2 dt' &\leq \int_0^t \|F_+^\eps(t')\|_{\mathfrak D'}dt' + \int_0^t \|\langle v\rangle^{m_1} F_+^\eps(t')\|_{L^2_x( \mathcal H_{\sigma;\eps}^+)_v}^2 dt' 
  \end{align}
  Now, by \eqref{eq:smallDiss} and \eqref{eq:ionDiff_eps},
  \begin{align}
  \int_0^t \|F_+^\eps(t')\|_{\mathfrak D'}^2dt'  \lesssim   \int_0^t \|F_+^\eps\|_{\mathfrak D'}^2dt'  +   \int_0^t \|F_+^\eps(t') - F_+^0(t')\|_{\mathfrak D'}^2dt'  \lesssim_M \eps^2 + \kappa^2.
  \end{align}
  On the other hand, recalling \eqref{eq:H+}, we note that for all $\nu \in \mathbf S^2$, 
  \begin{align}
  \eps \sigma_{ij}(\eps v) \nu_i \nu_j \lesssim \eps \lesssim \eps \langle v\rangle^{3} \sigma_{ij}( v) \nu_i \nu_j.
  \end{align}
  Hence, using \eqref{eq:dissControl},
  \begin{align}
    \int_0^t \|\langle v\rangle^{m_1}F_+^\eps\|_{L^2_x( \mathcal H_{\sigma;\eps}^+)_v}^2dt'   &\lesssim  t +   \int_0^t \|\langle v\rangle^{m_2}F_+^\eps\|_{L^2_x( \mathcal H_{\sigma;\eps}^+)_v}^2dt'  \\
    &\lesssim_M t + \eps  \int_0^t \|F_+^\eps\|_{\mathfrak D}^2dt' \\
    &\lesssim_M t + \eps.
  \end{align}
  Using $t \leq \hat T_\eps \leq \kappa$,
 we conclude \eqref{eq:smallF_+epsDiss} holds true. Taking the supremum of \eqref{eq:F_+epsBootstrapClose} over all $t \in [0,\hat T_\eps]$, we deduce
 \begin{align}
 \sup_{t\in[0,\hat T\eps]} \|F_+^\eps (t)\|_{\mathfrak E} \leq  \|F_{+,in}^\eps \|_{\mathfrak E}  + C_M (\kappa + \eps).
 \end{align}
Next, by the continuity equation, we have for all $t \in [0,\hat T_\eps]$, 
\begin{align}
\frac{1}{n_+^\eps(t,x)} = \frac{1}{n^\eps_{+,in}(x) - \int_0^t\int_{\mathbf R^3} v\cdot \nabla_x F_+(t,x,v)dvdt}
\end{align}
Taking $\kappa$ smaller if necessary, we get
\begin{align}
\sup_{t \in [0,\hat T_{\eps} ]} \|\frac{1}{n_+^\eps(t)}\|_{L^\infty_x} \leq \|\frac{1}{n_{+,in}^\eps}\|_{L^\infty_x} \frac{1}{1-C\|\frac{1}{n_{+,in}^\eps}\|_{L^\infty_x}\|F_+^\eps\|_{\mathfrak E} \kappa} \leq\frac{3}{2} \|\frac{1}{n_{+,in}^\eps}\|_{L^\infty_x}.
\end{align}
Now, taking $\eps $ and $\kappa$ sufficiently small, we have
\begin{align}
2M = \sup_{t\in[0,\hat T_\eps]}(\|  \frac{1}{n_+^\eps(t)}\|_{L^\infty_x}+\|F_+^\eps(t)\|_{\mathfrak E} )&\leq\frac{3}{2}\|  \frac{1}{n_{+,in}^\eps}\|_{L^\infty_x}+\|F_{+,in}^\eps\|_{\mathfrak E} + C_M(\kappa + \eps) \\
&=\frac{3}{2}M.
\end{align}

Next, assume \eqref{eq:bootstrap1.5i} holds with equality. However, by \eqref{eq:ionDiff_eps}, we have $\sup_{t\in [0,\hat T_\eps]} \|F_+^\eps(t) - F_+^0(t)\|_{\mathfrak E'} \lesssim _M \eps$, so we may take $\eps$ sufficiently small to reach a contradiction.

Next, assume \eqref{eq:bootstrap2i} holds with equality.  However by \eqref{eq:stability1'}, $\sqrt{ {\mathscr E^\eps}_{-,2}} \lesssim_M\sqrt{\eps}$. Thus, by taking $\eps$ small enough, we ensure
\begin{align}
\kappa = \sqrt{ {\mathscr E^\eps}_{-,2}(\hat T_\eps)} \leq \frac{\kappa}{2},
\end{align}
a contradiction.

Next, assume \eqref{eq:bootstrap3i} holds with equality. By \eqref{eq:dtBigPhiandBetaBd} and \eqref{eq:stability1'},
\begin{align}
\eta &= |\ln(\frac{\gamma^\eps(\hat T_\eps)}{\beta_{in}})|\\
& \leq \int_0^{\hat T_\eps} |\frac{\dot\gamma^\eps(t)}{\gamma^\eps(t)}|dt\\
&\lesssim  \int_0^{\hat T_\eps} 1 + \sqrt{{\mathscr D}_{-,2}^\eps} dt  \\
& \leq \hat T_\eps + \hat T_\eps^{\frac{1}{2}}\left( \int_0^{T_\eps}{\mathscr D}_{-,2}^\eps dt\right)^{\frac{1}{2}}.
\end{align}
Combining with \eqref{eq:stability1'}, we conclude $\eta \lesssim \hat T_\eps$.  We conclude $\eta = \eta(M) \lesssim_M \hat T_\eps \leq \kappa$. Taking $\kappa$ sufficiently small, we reach a contradiction.  

Now, assume \eqref{eq:bootstrap4i} holds with equality. By \eqref{eq:massConsPert},  \eqref{eq:dtBigPhiandBetaBd},  \eqref{eq:stability1}, we have
\begin{align}
1 &= \int_0^{\hat T_\eps} \|\int_{\mathbf R^3} (F_-^\eps - \mu_{\gamma^\eps}e^{\gamma^\eps\psi^\eps})d\xi \|_{L^2_x} dt  \\
&\quad \lesssim_M \int_0^{\hat T_\eps}\frac{1}{\eps} \| \int_{\mathbf R^3} v\cdot \nabla_x (F_-^\eps - \mu_{\gamma^\eps}e^{\gamma^\eps\psi^\eps})d\xi  \|_{L^2_x} + \|\partial_t (\gamma^\eps \psi^\eps) e^{\gamma^\eps \psi^\eps}\|_{L^2_x}dt \\
& \quad\lesssim_M \int_0^{\hat T_\eps}1 + \frac{1}{\eps}\sqrt{\widetilde {\mathscr D}_{-,2}^\eps} dt\\
&\quad \lesssim_M \hat T_\eps + \hat T_\eps^\frac{1}{2} 
\end{align}
This implies $1\lesssim_M \hat T_\eps \leq \kappa$. Taking $\kappa$ small enough, we reach a contradiction. 

In conclusion, $\hat T_\eps \geq \kappa$ for all $\eps$ small enough, so $\hat T \geq \kappa$. The bounds \eqref{eq:stability1'} and \eqref{eq:stability2'} imply \eqref{eq:error1} and \eqref{eq:error2} respectively.
\\

\noindent \textbf{Proof of part (ii):} The proof is similar to part (i), so we omit some details. In the same way as in part (i), we take solutions $(F_+,\beta,\phi^0)$ to \eqref{eq:VPLIon} (satisfying \eqref{eq:smallDiss}). We also take the solution $(F_+^\eps,F_-^\eps)$ to \eqref{eq:VPLRescaled} and $(\gamma^\eps ,\psi^\eps)$ satisfying the conditions \eqref{eq:bootstrap1i} through \eqref{eq:bootstrap4i} for some $\hat T_\eps > 0$. We note that the conclusions of Propositions \ref{prop:apriori} and \ref{prop:ionError} both hold. 

By \eqref{eq:energyCondition0}, and the fact that $\psi^\eps_{in} = \phi^0_{in}$, we have that the expression in \eqref{eq:energyCondition} is exactly zero. The bounds \eqref{eq:EtoEtilde} and \eqref{eq:DtoDtilde} both hold in this context as well, so we can guarantee $\sup_{t \in [0,\hat T_\eps]}\sqrt{\mathscr E_{-,2}^\eps} < \varsigma$ by taking $\kappa$ small enough.
Thus all the hypotheses of Proposition \ref{prop:derivationRefined} are satisfied up to  $T = \hat T_\eps$. 
 Next, we have that $\mathscr E^\eps_{-,\alpha,in} = \widetilde{\mathscr E}^\eps_{-,\alpha,in} \lesssim_M \delta^2$ (and this expression is independent of $\eps$). Therefore, by  \eqref{eq:stability1}, we have
\begin{align}\label{eq:stability1''}
\eps \sup_{t \in [0,\hat T_\eps]}  \widetilde {\mathscr E}^\eps_{-,2}(t) + \int_0^{\hat T_\eps} \widetilde  {\mathscr D}^\eps_{-,2}(t) dt \leq C_M (\eps  \delta^2 + \eps^2).
\end{align}
On the other hand, the above and  \eqref{eq:stability2} give the bound 
\begin{align}\label{eq:stability2''}
\widetilde{\mathscr E}^\eps_{-,1}(t) \lesssim_M  e^{-\frac{1}{C_M}(\frac{t}{\eps})^{\frac{2}{3}}}\delta^2+\eps^{\frac{5}{3}} t^\frac{1}{3} + \eps^2.
\end{align}

Next, similarly to  \eqref{eq:ionDiff_eps} in part (i), we have
\begin{align}\label{eq:ionDiff_eps'}
\sup_{t \in [0,\hat T_\eps]} \|F_+^\eps(t) - F_+^0(t)\|_{\mathfrak E'}^2 + \int_0^{\hat T_\eps}\|F_+^\eps(t) - F_+^0(t)\|_{\mathfrak D'}^2 dt \lesssim \eps(\eps+\delta^2).
\end{align}

Now, combining \eqref{eq:ionDiff_eps'} with  \eqref{eq:EtoEtilde} , \eqref{eq:DtoDtilde}, and \eqref{eq:stability1''} and \eqref{eq:stability2''}, we have
\begin{align}\label{eq:stability1'''}
\eps \sup_{t \in [0,\hat T_\eps]}   {\mathscr E}^\eps_{-,2}(t) + \int_0^{\hat T_\eps}   {\mathscr D}^\eps_{-,2}(t) dt \leq C_M (\eps  \delta^2 + \eps^2)
\end{align}
and for all $t \in [0,\hat T_\eps]$, we have
\begin{align}\label{eq:stability2'''}
{\mathscr E}^\eps_{-,1}(t) \lesssim_M  e^{-\frac{1}{C_M}(\frac{t}{\eps})^{\frac{2}{3}}}\delta^2+ \eps^{\frac{5}{3}}t^{\frac{1}{3}} +  \eps \delta^2.
\end{align}

We now show that $\hat T = \inf_{\eps \in (0,\overline \eps]} \hat T_\eps$ is positive. As before, we take $\hat T_\eps < \kappa$, and show that each of \eqref{eq:bootstrap1i} through \eqref{eq:bootstrap4i} holding with equality yields a contradiction, provided $\kappa$ is taken small enough. In the cases of \eqref{eq:bootstrap1i}, \eqref{eq:bootstrap1.5i}, \eqref{eq:bootstrap2i} and \eqref{eq:bootstrap3i}, the proof is essentially same as in the case of part (i), up to taking $\delta$ sufficiently small in addition to $\kappa$ and $\eps$.

We now address the case of when \eqref{eq:bootstrap4i} holds with equality at $T = \hat T_\eps$.
Next, take $0<t_0  < \hat T_\eps$ to be chosen later. Through a trivial modification of the proof of Proposition \ref{prop:derivationRefined}, we have that the estimate \eqref{eq:stability1} holds with $\alpha=1$ instead of $\alpha=2$, and $[t_0,T]$ instead of $[0,T]$. This gives
\begin{align}
\eps \sup_{t \in [t_0,T]} \widetilde{\mathscr E}_{-,1}^\eps(t) +\int_{t_0}^{\hat T_\eps} \mathscr D_{-,1}(t) dt&\lesssim_M \mathscr E_{-,1}^\eps(t_0)  + \eps^2\\
&\lesssim_M \eps (e^{-\frac{1}{C_M}(\frac{t_0}{\eps})^{\frac{2}{3}}}\delta^2+\eps^{\frac{5}{3}} t_0^\frac{1}{3}) + \eps^2
\end{align}
Taking $t_0 = \eps (C_M|\ln(\eps)|)^{\frac{3}{2}}$, we can bound last expression by by $\eps^2$. On the other hand, integrating \eqref{eq:stability2''} on $[0,t_0]$, we have
\begin{align}
\int_0^{t_0} \sqrt{ \widetilde{\mathscr E}^\eps_{-,1}(t)} dt \lesssim_M \eps \delta + \eps^{\frac{5}{4}} t_0^{\frac{7}{6}} + \eps t_0 \lesssim_M \eps \delta + \eps^{\frac{3}{2}}
\end{align}  
Therefore, combining these bounds, we deduce
\begin{align}\label{eq:L1Stability}
&\int_0^{\hat T_\eps}\|\int_{\mathbf R^3} \xi\cdot \nabla_x (F_-^{\eps} - \mu_{\gamma^\eps} e^{\gamma^\eps \psi^\eps})d\xi \|_{L^2_{x}}^2 dt \\
&\quad\lesssim_M \int_0^{t_0} \sqrt{ \widetilde{\mathscr E}^\eps_{-,1}(t)} dt +\left(\int_0^{t_0} \widetilde{\mathscr D}^\eps_{-,1}(t) dt \right)^{\frac{1}{2}} \lesssim _M \eps.
\end{align}
Now, combining the above with \eqref{eq:massConsPert},  \eqref{eq:dtBigPhiandBetaBd}, and \eqref{eq:stability1'''},
\begin{align}
1 &= \int_0^t \|\int_{\mathbf R^3} (F_-^\eps - \mu_{\gamma^\eps}e^{\gamma^\eps\psi^\eps})d\xi \|_{L^2_x} dt  \\
&\quad \lesssim_M \int_0^{\hat T_\eps}\frac{1}{\eps} \| \int_{\mathbf R^3} \xi\cdot \nabla_x (F_-^\eps - \mu_{\gamma^\eps}e^{\gamma^\eps\psi^\eps})d\xi  \|_{L^2_x} + \|\partial_t (\gamma^\eps \psi^\eps) e^{\gamma^\eps \psi^\eps}\|_{L^2_x}dt \\
&\quad \lesssim_M \eps \delta + \eps^\frac{3}{2} + \hat T_\eps + \hat T_\eps^{\frac{1}{2}}\\
&\quad \leq \eps \delta +\eps^{\frac{3}{2}} + \kappa^{\frac{1}{2}}.
\end{align}
Taking $\eps,\delta$ and $\kappa$ sufficiently small, we have a contradiction.

In conclusion, we have $\hat T \geq \kappa >0$. Moreover, \eqref{eq:error1'}, \eqref{eq:error2'} and \eqref{eq:error3'} follow from \eqref{eq:ionDiff_eps'}, \eqref{eq:stability1'''}, and \eqref{eq:stability2'''} respectively. \qed

\bibliographystyle{abbrv} 
\bibliography{refs}

\begin{thebibliography}{10}

\bibitem{alexandre_landau_2004}
R.~Alexandre and C.~Villani.
\newblock On the {Landau} approximation in plasma physics.
\newblock {\em Annales de l'Institut Henri Poincar{\'e} C, Analyse non
  lin{\'e}aire}, 21(1):61--95, Jan. 2004.

\bibitem{balescu1988transport}
R.~Balescu.
\newblock {\em Transport processes in plasmas}.
\newblock 1988.

\bibitem{bardos_maxwellboltzmann_2018}
C.~Bardos, F.~Golse, T.~T. Nguyen, and R.~Sentis.
\newblock The {Maxwell}--{Boltzmann} approximation for ion kinetic modeling.
\newblock {\em Physica D: Nonlinear Phenomena}, 376-377:94--107, Aug. 2018.

\bibitem{berezin1967nonlinear}
Y.~A. Berezin and V.~Karpman.
\newblock Nonlinear evolution of disturbances in plasmas and other dispersive
  media.
\newblock {\em Soviet Physics JETP}, 24:1049--1056, 1967.

\bibitem{bouchut1991global}
F.~Bouchut.
\newblock Global weak solution of the {Vlasov}--{Poisson} system for small
  electrons mass.
\newblock {\em Communications in partial differential equations},
  16(8-9):1337--1365, 1991.

\bibitem{bouchut1995long}
F.~Bouchut and J.~Dolbeault.
\newblock On long time asymptotics of the {Vlasov}--{Fokker}--{Planck} equation
  and of the {Vlasov}-{Poisson}-{Fokker}-{Planck} system with {Coulombic} and
  {Newtonian} potentials.
\newblock {\em Differential and Integral Equations}, 8(3):487--514, 1995.

\bibitem{cesbron2021global}
L.~Cesbron and M.~Iacobelli.
\newblock Global well-posedness of vlasov-poisson-type systems in bounded
  domains.
\newblock {\em arXiv preprint arXiv:2108.11209}, 2021.

\bibitem{degond_dispersion_1997}
P.~Degond and M.~Lemou.
\newblock Dispersion {Relations} for the {Linearized} {Fokker}-{Planck}
  {Equation}.
\newblock {\em Archive for Rational Mechanics and Analysis}, 138(2):137--167,
  July 1997.

\bibitem{degond1996asymptotics}
P.~Degond and B.~Lucquin-Desreux.
\newblock The asymptotics of collision operators for two species of particles
  of disparate masses.
\newblock {\em Mathematical Models and Methods in Applied Sciences},
  6(03):405--436, 1996.

\bibitem{degond1996transport}
P.~Degond and B.~Lucquin-Desreux.
\newblock Transport coefficients of plasmas and disparate mass binary gases.
\newblock {\em Transport Theory and Statistical Physics}, 25(6):595--633, 1996.

\bibitem{duan2022compressible}
R.~Duan, D.~Yang, and H.~Yu.
\newblock Compressible fluid limit for smooth solutions to the {Landau}
  equation.
\newblock {\em arXiv preprint arXiv:2207.01184}, 2022.

\bibitem{flynn2023local}
P.~Flynn.
\newblock Local well-posedness for the {Vlasov}-{Poisson}-{Landau} system with
  massless electrons.
\newblock {\em In preparation}, 2023.

\bibitem{gagnebin2022landau}
A.~Gagnebin and M.~Iacobelli.
\newblock Landau damping on the torus for the {Vlasov}-{Poisson} system with
  massless electrons.
\newblock {\em arXiv preprint arXiv:2209.04676}, 2022.

\bibitem{grenier2020derivation}
E.~Grenier, Y.~Guo, B.~Pausader, and M.~Suzuki.
\newblock Derivation of the ion equation.
\newblock {\em Quarterly of applied mathematics}, 78(2), 2020.

\bibitem{griffin2018global}
M.~Griffin-Pickering and M.~Iacobelli.
\newblock Global well-posedness in 3-dimensions for the {Vlasov}-{Poisson}
  system with massless electrons.
\newblock {\em arXiv}, 2018, 2018.

\bibitem{griffin2019recent}
M.~Griffin-Pickering and M.~Iacobelli.
\newblock Recent developments on the well-posedness theory for {Vlasov}-type
  equations.
\newblock {\em From Particle Systems to Partial Differential Equations}, pages
  301--319, 2019.

\bibitem{griffin2020singular}
M.~Griffin-Pickering and M.~Iacobelli.
\newblock Singular limits for plasmas with thermalised electrons.
\newblock {\em Journal de Math{\'e}matiques Pures et Appliqu{\'e}es},
  135:199--255, 2020.

\bibitem{griffin2021global}
M.~Griffin-Pickering and M.~Iacobelli.
\newblock Global strong solutions in {$\mathbf R^3$} for ionic
  {Vlasov}-{Poisson} systems.
\newblock {\em Kinetic \& Related Models}, 14(4):571, 2021.

\bibitem{guo_landau_2002}
Y.~Guo.
\newblock The {Landau} {Equation} in a {Periodic} {Box}.
\newblock {\em Communications in Mathematical Physics}, 231(3):391--434, Dec.
  2002.

\bibitem{guo2012vlasov}
Y.~Guo.
\newblock The {Vlasov}-{Poisson}-{Landau} system in a periodic box.
\newblock {\em Journal of the American Mathematical Society}, 25(3):759--812,
  2012.

\bibitem{guo2010global}
Y.~Guo and J.~Jang.
\newblock Global hilbert expansion for the {Vlasov}-{Poisson}-{Boltzmann}
  system.
\newblock {\em Communications in Mathematical Physics}, 299(2):469--501, 2010.

\bibitem{guo2009local}
Y.~Guo, J.~Jang, and N.~Jiang.
\newblock Local {Hilbert} expansion for the {Boltzmann} equation.
\newblock {\em Kinetic \& Related Models}, 2(1):205, 2009.

\bibitem{guo2021global}
Y.~Guo and Q.~Xiao.
\newblock Global hilbert expansion for the relativistic
  {Vlasov}--{Maxwell}--{Boltzmann} system.
\newblock {\em Communications in Mathematical Physics}, 384(1):341--401, 2021.

\bibitem{han2011quasineutral}
D.~Han-Kwan.
\newblock Quasineutral limit of the {Vlasov}-{Poisson} system with massless
  electrons.
\newblock {\em Communications in Partial Differential Equations},
  36(8):1385--1425, 2011.

\bibitem{han2013vlasov}
D.~Han-Kwan.
\newblock From {Vlasov}--{Poisson} to {Korteweg}--de {Vries} and
  {Zakharov}--{Kuznetsov}.
\newblock {\em Communications in Mathematical Physics}, 324(3):961--993, 2013.

\bibitem{han2017quasineutral}
D.~Han-Kwan and M.~Iacobelli.
\newblock The quasineutral limit of the {Vlasov}--{Poisson} equation in
  wasserstein metric.
\newblock {\em Communications in Mathematical Sciences}, 15(2):481--509, 2017.

\bibitem{herda2016massless}
M.~Herda.
\newblock On massless electron limit for a multispecies kinetic system with
  external magnetic field.
\newblock {\em Journal of Differential Equations}, 260(11):7861--7891, 2016.

\bibitem{hinton1983collisional}
F.~L. Hinton.
\newblock Collisional transport in plasma.
\newblock {\em Handbook of Plasma Physics}, 1(147):331, 1983.

\bibitem{metivier2008differential}
G.~M{\'e}tivier.
\newblock {\em Para-differential calculus and applications to the {Cauchy}
  problem for nonlinear systems}.
\newblock 2008.

\bibitem{ouyang2022hilbert}
Z.~Ouyang, L.~Wu, and Q.~Xiao.
\newblock Hilbert expansion for {Coulomb} collisional kinetic models.
\newblock {\em arXiv preprint arXiv:2207.00126}, 2022.

\bibitem{shimizu1972automodulation}
K.~Shimizu and Y.~H.~Ichikawa.
\newblock Automodulation of ion oscillation modes in plasma.
\newblock {\em Journal of the Physical Society of Japan}, 33(3):789--792, 1972.

\bibitem{washimi1966propagation}
H.~Washimi and T.~Taniuti.
\newblock Propagation of ion-acoustic solitary waves of small amplitude.
\newblock {\em Physical Review Letters}, 17(19):996, 1966.

\bibitem{zakharov1972collapse}
V.~E. Zakharov et~al.
\newblock Collapse of {Langmuir} waves.
\newblock {\em Sov. Phys. JETP}, 35(5):908--914, 1972.

\bibitem{zhang2009stability}
M.~Zhang.
\newblock Stability of the {Vlasov}--{Poisson}--{Boltzmann} system in {$\mathbf
  R^3$}.
\newblock {\em Journal of Differential Equations}, 247(7):2027--2073, 2009.

\end{thebibliography}

\end{document}